\documentclass[a4paper,12pt,reqno]{amsart}

\usepackage{amsmath, amsfonts, amssymb, amsthm, mathrsfs, mathtools, mathalfa}
\usepackage{comment}
\usepackage{yfonts}
\usepackage{enumerate}
\usepackage{enumitem}
\usepackage{stmaryrd}
\usepackage{fancyhdr}
\usepackage{bm}
\usepackage[utf8]{inputenc}
\usepackage[pdfencoding=auto, hyperfootnotes=false]{hyperref}
\usepackage{tikz,tikz-cd}
\usepackage[hang,flushmargin]{footmisc}
\usepackage{graphicx}
\usepackage{xr}
\usepackage{extarrows}
\usepackage[dvistyle, colorinlistoftodos]{todonotes}

\usepackage{tikz}
\usetikzlibrary{matrix}
\tikzset{stretch/.initial=1}
\usetikzlibrary{calc}

\newcommand\drawloop[4][]%
   {\draw[shorten <=0pt, shorten >=0pt,#1]
      ($(#2)!\pgfkeysvalueof{/tikz/stretch}!(#2.#3)$)
      let \p1=($(#2.center)!\pgfkeysvalueof{/tikz/stretch}!(#2.north)-(#2)$),
          \n1= {veclen(\x1,\y1)*sin(0.5*(#4-#3))/sin(0.5*(180-#4+#3))}
      in arc [start angle={#3-90}, end angle={#4+90}, radius=\n1]%
   }

\makeatletter
\def\@tocline#1#2#3#4#5#6#7{\relax
  \ifnum #1>\c@tocdepth % then omit
  \else
    \par \addpenalty\@secpenalty\addvspace{#2}%
    \begingroup \hyphenpenalty\@M
    \@ifempty{#4}{%
      \@tempdima\csname r@tocindent\number#1\endcsname\relax
    }{%
      \@tempdima#4\relax
    }%
    \parindent\z@ \leftskip#3\relax \advance\leftskip\@tempdima\relax
    \rightskip\@pnumwidth plus4em \parfillskip-\@pnumwidth
    #5\leavevmode\hskip-\@tempdima
      \ifcase #1
       \or\or \hskip 1em \or \hskip 2em \else \hskip 3em \fi%
      #6\nobreak\relax
    \dotfill\hbox to\@pnumwidth{\@tocpagenum{#7}}\par
    \nobreak
    \endgroup
  \fi}
\makeatother

\newtheorem{theorem}{Theorem}[section]
\newtheorem{lemma}[theorem]{Lemma}

\newtheorem{corollary}[theorem]{Corollary}
\newtheorem{proposition}[theorem]{Proposition}
\newtheorem{question}[theorem]{Question}

\theoremstyle{definition}
\newtheorem{defn}[theorem]{Definition}
\newtheorem{remark}[theorem]{Remark}

\newcommand{\mc}{\mathcal}

\newcommand{\mf}{\mathbf}
\newcommand{\mb}{\mathbb}

\newcommand{\wh}{\widehat}

\newcommand{\ud}{\,\mathrm{d}}
\newcommand{\id}{\mathrm{id}}

\DeclareMathOperator{\abph}{\mc{U}}

\DeclareMathOperator{\ab}{Z}

\DeclareMathOperator{\tran}{\Theta}

\DeclareMathOperator{\poly}{poly}

\DeclareMathOperator{\supp}{supp}

\DeclareMathOperator{\adj}{adj}

\DeclareMathOperator{\codim}{codim}

\DeclareMathOperator{\q}{c}
\DeclareMathOperator{\ns}{X}
\DeclareMathOperator{\nss}{Y}
\DeclareMathOperator{\co}{\circ\hspace{-0.02 cm}}
\DeclareMathOperator{\cu}{C}
\DeclareMathOperator{\cor}{Cor}
\DeclareMathOperator{\cs}{s}

\newcommand*{\sbr}[1]{\scalebox{0.8}{$(#1)$}}
\newcommand*{\db}[1]{\llbracket #1\rrbracket}

\DeclareMathOperator{\tIm}{Im}
\DeclareMathOperator{\tnss}{\tilde{\nss}}

\makeatletter
\renewcommand*\env@matrix[1][*\c@MaxMatrixCols c]{%
  \hskip -\arraycolsep
  \let\@ifnextchar\new@ifnextchar
  \array{#1}}
\makeatother

\begin{document}
\vspace*{-0.5cm}

\title[On higher-order Fourier analysis in characteristic $p$]{On higher-order Fourier analysis in characteristic $p$}

\author{Pablo Candela}
\address{Universidad Aut\'onoma de Madrid and ICMAT\\ Madrid 28049\\ Spain}
\email{pablo.candela@uam.es}

\author{Diego Gonz\'alez-S\'anchez}
\address{MTA Alfr\'ed R\'enyi Institute of Mathematics\\ 
Budapest, Hungary, H-1053}
\email{diegogs@renyi.hu}

\author{Bal\'azs Szegedy}
\address{MTA Alfr\'ed R\'enyi Institute of Mathematics\\ 
Budapest, Hungary, H-1053}
\email{szegedyb@gmail.com}

\begin{abstract}
In this paper, the nilspace approach to higher-order Fourier analysis is developed in the setting of vector spaces over a prime field $\mb{F}_p$, with applications mainly in ergodic theory. A key requisite for this development is to identify a class of nilspaces adequate for this setting. We introduce such a class, whose members we call $p$-homogeneous nilspaces. One of our main results characterizes these objects in terms of a simple algebraic property. We then prove various further results on these nilspaces, leading to a structure theorem describing every finite $p$-homogeneous nilspace as the image, under a nilspace fibration, of a member of a simple family of filtered finite abelian $p$-groups. The applications include a description of the Host-Kra factors of ergodic $\mb{F}_p^\omega$-systems as $p$-homogeneous nilspace systems. This enables the analysis of these factors to be reduced to the study of such nilspace systems, with central questions on the factors thus becoming purely algebraic problems on finite nilspaces. We illustrate this approach by proving that for $k\leq p+1$ the $k$-th Host--Kra factor is an Abramov system of order $\leq k$, extending a result of Bergelson--Tao--Ziegler that holds for $k< p$. We illustrate the utility of $p$-homogeneous nilspaces also by showing that the above-mentioned structure theorem yields a new proof of the Tao--Ziegler inverse theorem for Gowers norms on $\mb{F}_p^n$.
\end{abstract}
\maketitle

\section{Introduction}
\noindent The theory of higher-order Fourier analysis, initiated by Gowers in his celebrated work on Szemer\'edi's theorem \cite{GSz}, has generated various fascinating developments in analysis and combinatorics in the last two decades. Many of these developments aim to understand the relation between the central objects in this theory, namely the uniformity norms (or Gowers norms), and structures involving nilpotent groups. This relation emerged early on, especially in the work \cite{HK-non-conv} of Host and Kra which introduced seminorms in ergodic theory analogous to the uniformity norms, and proved the Ergodic Structure Theorem, establishing a deep connection between these seminorms and nilmanifolds \cite[Theorem 10.1]{HK-non-conv} (see also \cite[Ch.\ 16]{HKbook}). This inspired further progress, notably in the work of Green and Tao \cite{GT08} and Green--Tao--Ziegler \cite{GTZ-U4} in arithmetic combinatorics, developing this connection between Gowers norms and nilmanifolds, leading to the proof by Green, Tao and Ziegler of the inverse theorem for Gowers norms on finite cyclic groups \cite{GTZ}. 

The search for further conceptual clarification of the above-mentioned connection also led to the discovery of interesting structures closely related to the uniformity norms, starting with the \emph{parallelepiped structures} introduced by Host and Kra \cite{HK-par}, leading to the concept of \emph{nilspaces} defined by Antol\'in Camarena and the third named author in \cite{CamSzeg}. 

\noindent Nilspace related topics have now grown into an active research area, including detailed treatments by the first named author \cite{Cand:Notes1,Cand:Notes2} and by Gutman, Manners and Varj\'u \cite{GMV1,GMV2,GMV3}, as well as further applications in arithmetic combinatorics, ergodic theory, probability theory, and topological dynamics \cite{CGSS,CScouplings,GGY,GMV1,GMV2,GMV3}.

Initial applications of nilspaces in higher-order Fourier analysis were obtained in \cite{SzegHigh}, where they were combined with analysis on ultraproducts to prove regularity and inverse theorems for the Gowers norms on various families of compact abelian groups. In \cite[end of Ch.\ 17]{HKbook}, Host and Kra suggested that the nilspace approach from \cite{SzegHigh} might be unified with the analysis of characteristic factors for uniformity seminorms from \cite{HK-non-conv}. A measure-theoretic framework enabling such a unification was introduced in \cite{CScouplings}, based on the notion of cubic couplings, inspired by the Host--Kra measures from \cite[\S 3]{HK-non-conv}. Applications of this framework included an extension of the Ergodic Structure Theorem to nilpotent group actions (see \cite[Theorem 5.12]{CScouplings} and \cite[Theorem 5.1]{CGSS}), and an extension of the inverse theorem to all compact abelian groups and also to nilmanifolds \cite{CSinverse}. 

In this paper we aim to demonstrate the utility of the above-mentioned framework in another principal setting for higher-order Fourier analysis, in which this approach had not been applied previously, namely the \emph{characteristic-$p$ setting}. Here, the uniformity norms are studied on vector spaces $\mb{F}_p^n$ over a field $\mb{F}_p$ of fixed prime order $p$, with dimension $n$ allowed to tend to infinity. This direction was fostered notably by Green \cite{GreenSurvey}, who promoted these vector spaces as useful models for various problems in arithmetic combinatorics that were originally posed in the integer setting, the latter setting being usually modeled by cyclic groups $\mb{Z}_N$ of large prime order $N$ allowed to tend to infinity. The usefulness of the vector space models relies mainly on the fact that they provide much richer algebraic structure than is available in the integer setting. The characteristic-$p$ setting has been very fruitful for higher-order Fourier analysis, with many interesting results in arithmetic combinatorics and in ergodic theory (for more background on this setting, see for instance the survey \cite{WolfSurvey}). Among these results, the present paper is related mainly to the work of Bergelson, Tao and Ziegler on ergodic actions of the (additive group of the) vector space $\mb{F}_p^\omega=\bigoplus_{i\in \mb{N}}\mb{F}_p$ \cite{BTZ, BTZ2}, and the related inverse theorems for Gowers norms proved by Tao and Ziegler in \cite{TZ-High,T&Z-Low}. Let us now describe the approach to these topics in this paper.
 
In the integer setting, a decisive conceptual step was to identify \emph{nilmanifolds} as adequate spaces with which to define basic harmonics (nilsequences) that could yield a useful inverse theorem for the Gowers norms. In characteristic $p$, the greater algebraic richness of this setting made it possible to obtain an inverse theorem with the corresponding harmonics being \emph{global phase polynomials}, easily definable directly on the initial spaces $\mb{F}_p^n$. Thus, inverse theorems in  characteristic $p$ have hitherto been obtained without a conceptual step similar to the above-mentioned one involving nilmanifolds. On the other hand, this has left open several questions that are relevant in order to clarify and strengthen the connections between the integer setting and the characteristic-$p$ setting. These questions also lead to new results in the latter setting in itself, and can be subsumed under the following initial and more general question.
\begin{question}\label{Q:main}
Which class of spaces analogous to compact nilmanifolds is adequate for higher-order Fourier analysis in characteristic $p$?
\end{question}
\noindent In this paper we show that the nilspace approach offers a useful answer to this question. More precisely, we identify and study the class of compact nilspaces that emerges when the above-mentioned framework from \cite{CScouplings,CSinverse} is applied in the characteristic-$p$ setting. We call these structures \emph{$p$-homogeneous nilspaces}. We show that these nilspaces yield a description of the characteristic factors for uniformity seminorms which is strong enough to give new proofs of central results in higher-order Fourier analysis in this setting, such as the inverse theorem for the Gowers norms from \cite{TZ-High,T&Z-Low}, and also new results concerning the Host--Kra factors of ergodic $\mb{F}_p^\omega$-systems.

To explain our results, first let us briefly recall the strategy used in \cite{CScouplings,CSinverse} to prove the inverse theorem for Gowers norms in the integer setting, as it is overall the same process that will lead to $p$-homogeneous nilspaces in characteristic $p$. To this end, let us recall that a nilspace is a set $\ns$ equipped with a sequence of cube sets $\cu^n(\ns)\subset \ns^{\{0,1\}^n}$, $n\geq 0$, whose elements are called the \emph{$n$-cubes} on $\ns$, satisfying three axioms, the most subtle of which is the \emph{completion axiom}, which states that any \emph{$n$-corner} on $\ns$ (roughly speaking, an $n$-cube missing one vertex) can always be completed to an $n$-cube. We say that $\ns$ is a \emph{$k$-step} nilspace if every $(k+1)$-corner on $\ns$ has a \emph{unique} completion. Instead of recalling these definitions in detail, it is more helpful intuitively at this point just to recall the standard example of how any abelian group $\ab$ can be viewed as a 1-step nilspace, denoted by $\mc{D}_1(\ab)$: for each $n$, the cube set $\cu^n(\mc{D}_1(\ab))$ consists of the \emph{standard $n$-cubes}, of the form $(x+v\sbr{1}h_1+\cdots+v\sbr{n}h_n)_{v\in\{0,1\}^n}\in \ab^{\{0,1\}^n}$, for any elements $x,h_i\in \ab$. Given nilspaces $\ns$, $\nss$, a \emph{nilspace morphism} from $\ns$ to $\nss$ is a cube-preserving map $\ns\to \nss$, and the set of all such morphisms is denoted by $\hom(\ns,\nss)$. A nilspace $\ns$ is \emph{compact} if the set $\ns$ is equipped with a compact second-countable Hausdorff topology which is compatible with the cubic structure in the sense that each cube set $\cu^n(\ns)$ is compact in the product topology on $\ns^{\{0,1\}^n}$. We refer to \cite{Cand:Notes1,Cand:Notes2} for more background on nilspaces.

The main result in \cite{CScouplings} is a structure theorem describing the characteristic factors, for a general type of uniformity seminorms, in terms of compact nilspaces \cite[Theorem 4.2]{CScouplings}. The strategy in \cite{CSinverse} is based on the fact that when this structure theorem is applied to the uniformity seminorms on ultraproducts of cyclic groups $\mb{Z}_N$ (for increasing primes $N$), the resulting characteristic factors are completely described by the class of compact nilspaces $\ns$ admitting morphisms $\mc{D}_1(\mb{Z}_N)\to \ns$ that are increasingly \emph{balanced} as $N$ grows (the notion of balanced morphism, recalled in detail in Section \ref{sec:bridge}, involves a quantitative form of equidistribution which also requires the $n$-cubic power of the morphism to be equidistributed in $\cu^n(\ns)$ for large $n$). A key step in this proof of the inverse theorem on $\mb{Z}_N$ is then to show that the relevant nilspaces arising this way are all connected nilmanifolds, more precisely, \emph{toral nilspaces} \cite[Theorem 6.1]{CSinverse}.

As signalled by Question \ref{Q:main}, previously there was no class of compact  nilspaces clearly identified as playing a role in characteristic $p$ similar  to the role of nilmanifolds in the integer setting. However, the above strategy indicates such a class in a natural way, namely, the class of compact nilspaces $\ns$ with the property of admitting increasingly balanced morphisms from (the additive group of) $\mb{F}_p^n$ into $\ns$ as $n$ grows. As we shall see, this property yields one of various equivalent ways of defining $p$-homogeneous compact nilspaces. Moreover, one of the main results in this paper shows that these nilspaces can also be identified by a much simpler and purely algebraic property. Because of its simplicity, we use this property to define $p$-homogeneous nilspaces, as follows. 
\begin{defn}\label{def:p-hom-intro}
Let $\ns$ be a nilspace and let $p$ be a prime. We say that $\ns$ is a \emph{$p$-homogeneous nilspace} if for every positive integer $n$, for every $f\in \hom(\mc{D}_1(\mb{Z}^n),\ns)$ the restriction $f|_{[0,p-1]^n}$ is in $\hom(\mc{D}_1(\mb{Z}_p^n),\ns)$. \footnote{Here $\mb{Z}_p^n$ is identified with $[0,p-1]^n$ equipped with addition mod $p$, the usual way.}
\end{defn}
\noindent To state the announced result linking Definition \ref{def:p-hom-intro} to balanced morphisms from $\mb{F}_p^n$, we need the notion of a \emph{finite-rank} nilspace. This involves the fact that a $k$-step compact nilspace $\ns$ can always be decomposed as a $k$-fold  \emph{compact abelian bundle} \cite[Definition 2.1.8]{Cand:Notes2}, where for each $i\in [k]$ the nilspace factor $\ns_i$ is an extension (in the sense of \cite[Definition 3.3.13]{Cand:Notes1}) of $\ns_{i-1}$ by a compact abelian group, called the $i$-th \emph{structure group} of $\ns$ and denoted by $\ab_i$ or $\ab_i(\ns)$; see \cite[Proposition 2.1.9]{Cand:Notes2}. If every structure group has finite rank then $\ns$ is called a  \emph{compact finite-rank nilspace}, which we abbreviate to ``$\textsc{cfr}$ nilspace". The \emph{inverse limit theorem} for nilspaces \cite{CamSzeg} states that every compact nilspace can be decomposed as an inverse limit of \textsc{cfr} nilspaces (see also \cite[Theorem 2.7.3]{Cand:Notes2}). This often enables the study of a  class of compact nilspaces to be reduced in a very useful way to the study of the \textsc{cfr} members of the class. We can now state the announced result.
\begin{theorem}\label{thm:main1-intro}
Let $\ns$ be a $k$-step \textsc{cfr} nilspace, let $d$ be a metric\footnote{The metric underlies the notion of balance for morphisms, see Remark \ref{rem:metconv}.} generating the topology on $\ns$, and let $p$ be a prime. There exists $b=b(\ns,d,p)>0$ such that the following holds: if for some $D$ there is a $b$-balanced morphism $\varphi:\mc{D}_1(\mb{Z}_p^D)\to \ns$, then $\ns$ is $p$-homogeneous.
\end{theorem}
\noindent The proof of Theorem \ref{thm:main1-intro} occupies Section \ref{sec:bridge} and involves several steps, using in particular a recent refinement of the Generalized Von Neumann Theorem \cite{CGSS-seq-CS-Eurocomb,CGSS-seq-CS} (see also the recent work of Manners \cite{MannersTC}).

With Theorem \ref{thm:main1-intro}, the theory of $p$-homogeneous nilspaces can be developed using the property in Definition \ref{def:p-hom-intro}, a property that has rather strong consequences, which we begin to develop in Section \ref{sec:p-hom}. The simplest examples of $p$-homogeneous nilspaces are the 1-step nilspaces based on finite elementary abelian $p$-groups, i.e.\ the nilspaces $\mc{D}_1(\mb{Z}_p^n)$. In Section \ref{sec:p-hom}, the property in Definition \ref{def:p-hom-intro} is used in particular to identify a larger family of examples of $p$-homogeneous nilspaces, within the standard class of so-called \emph{group nilspaces}. Recall that a $k$-step group nilspace is constructed by taking any group $G$ equipped with a filtration $G_\bullet$ of finite degree $k$, and equipping $G$ with the associated \emph{Host--Kra cube sets} $\cu^n(G_\bullet)$, $n\geq 0$; we call this nilspace the \emph{group nilspace} associated with the filtered group $(G,G_\bullet)$ (see for instance \cite[Ch.\ 6]{HKbook} for more background on Host--Kra cubes, and \cite[\S 2.2.1]{Cand:Notes1} for the group nilspace construction).

We say that a filtration $G_\bullet=(G_i)_{i\geq 0}$ is \emph{$p$-homogeneous} if for all $i$, for all $g\in G_i$ we have $g^p\in G_{i+p-1}$.  We obtain the following description of $p$-homogeneous group nilspaces.
\begin{theorem}\label{thm:intro-1}
Let $p$ be a prime and let $(G,G_{\bullet})$ be a filtered group. The associated group nilspace is a $p$-homogeneous nilspace if and only if $G_\bullet$ is a $p$-homogeneous filtration.
\end{theorem}
\noindent In particular, since morphisms between group nilspaces are the same thing as polynomial maps between the corresponding filtered groups (see e.g.\ \cite[Section \S 2.2.2]{Cand:Notes1}), Theorem \ref{thm:intro-1} implies that given any filtered group $(G,G_{\bullet})$, the filtration $G_\bullet$ is $p$-homogeneous if and only if for every polynomial $f\in \poly(\mb{Z}^n,G_{\bullet})$ the restriction $f|_{[0,p-1]^n}$ yields a polynomial map in $\poly(\mb{Z}_p^n,G_{\bullet})$. Theorem \ref{thm:intro-1} also implies in a simple way that for every $p$-homogeneous $k$-step nilspace defined on a finite cyclic group, the group must in fact be isomorphic to a subgroup of $\mb{Z}_{p^r}$ for $r=\lfloor \frac{k-1}{p-1}\rfloor+1$; see Proposition \ref{prop:cyclic-p-hom}. These groups $\mb{Z}_{p^r}$ underlie the \emph{non-classical polynomials} of degree $k$ on $\mb{F}_p^n$, identified by Tao and Ziegler in \cite{T&Z-Low} as adequate harmonics for an inverse theorem in characteristic $p$ that is valid even in the so-called \emph{low-characteristic setting}, i.e.\ for $p\leq k$. In the present approach, these cyclic groups also play a key role, but rather as basic objects used to describe more general $p$-homogeneous nilspaces. Using these general descriptions (detailed below), the inverse theorem can be deduced relatively easily, as explained at the end of this introduction.

Concerning more general $p$-homogeneous nilspaces (not necessarily group nilspaces), our main results in Section \ref{sec:p-hom} include the following proposition. In particular, this indicates that $p$-homogeneous nilspaces are natural generalizations of elementary abelian $p$-groups.

\begin{proposition}\label{prop:p-hom-str-gps-intro}
Let $\ns$ be a $k$-step $p$-homogeneous nilspace. Then for every $i\in [k]$, the structure group $\ab_i(\ns)$ is an elementary abelian $p$-group. In particular, a $p$-homogeneous nilspace is \textsc{cfr} if and only if it is a finite nilspace.\footnote{We say that a nilspace is \emph{finite} if its underlying set is finite.}
\end{proposition}
\noindent When the property of structure groups in this proposition is combined with a certain lifting property for morphisms from elementary abelian $p$-groups into $\ns$, we obtain a useful \emph{sufficient} condition for $\ns$ to be $p$-homogeneous; see Proposition \ref{prop:lifting-equiv}. 

In Section \ref{sec:gen-set}, we prove a structure theorem for $p$-homogeneous nilspaces, which is also a key ingredient in our applications. The theorem describes general finite $p$-homogeneous nilspaces as images, under \emph{nilspace fibrations}, of members of a much simpler class of $p$-homogeneous nilspaces, defined as follows. (To recall the notion of nilspace fibrations, also known by the original term \emph{fiber surjective morphisms}, see \cite[Definition 7.1]{GMV1}, \cite[Definition 3.3.7]{Cand:Notes1}; essentially, the role of fibrations for compact nilspaces generalizes the role of continuous surjective homomorphisms for compact abelian groups.)
\begin{defn}\label{def:bblocks-intro}
Let $p$ be a prime, let $k,\ell\in\mb{N}$ with $k\geq \ell$, and let $r=r(k,\ell,p):=\lfloor\frac{k-\ell}{p-1}\rfloor+1$. We define $\abph_{k,\ell}$ to be the $k$-step $p$-homogeneous group nilspace\footnote{When the prime $p$ needs to be specified we will write $\abph^{(p)}_{k,\ell}$, but usually we omit this superscript.} consisting of the cyclic group $G=\mb{Z}_{p^r}$ equipped with the $p$-homogeneous degree-$k$ filtration $(G_i)_{i\geq 0}$ where $G_i = \mb{Z}_{p^r}$ for $i\in [0,\ell]$ and $G_i=p^{\lfloor\frac{i-\ell-1}{p-1}\rfloor+1}\mb{Z}_{p^r}$ for $i\ge \ell+1$, that is, the filtration
\[\begin{array}{cccccccccccccccc}
G_1 &   & G_\ell & & G_{\ell+1} & & G_{\ell+p-1} & & G_{\ell+p}\\
\parallel &  & \parallel & & \parallel &  & \parallel & &  \parallel\\
\mb{Z}_{p^r}& = \cdots = & \mb{Z}_{p^r} & \ge & p\mb{Z}_{p^r} & =  \cdots  = & p\mb{Z}_{p^r} &\geq & p^2\mb{Z}_{p^r}\;\cdots
\end{array}.
\]
We define $\mc{Q}_{p,k}$ to be the set of all $p$-homogeneous nilspaces $\nss$ such that for some integers $a_\ell \ge 0$ ($\ell\in[k]$) we have that $\nss$ is isomorphic to the product\footnote{The definition of product nilspaces (or of powers $\ns^a$ of a nilspace $\ns$) is the natural one; see \cite[\S 3.1.1]{Cand:Notes1}.}  nilspace $\prod_{\ell= 1}^k \abph_{k,\ell}^{\,a_\ell}$.
\end{defn}
\noindent The cyclic groups underlying the nilspaces $\abph_{k,\ell}$ agree with those underlying the non-classical polynomials in \cite[Lemma 1.6 (vi)]{T&Z-Low}, as mentioned above. 

We can now state the structure theorem, which establishes that the abelian group nilspaces in $\mc{Q}_{p,k}$ suffice to describe all $p$-homogeneous \textsc{cfr} nilspaces via fibrations.
\begin{theorem}\label{thm:general-p-hom-intro}
Let $\ns$ be a $k$-step $p$-homogeneous finite nilspace. Then there exists $\nss\in \mc{Q}_{p,k}$ and a fibration $\psi:\nss\to \ns$ with the following property: for every morphism $f\in \hom(\mc{D}_1(\mb{Z}_p^n),\ns)$ there is a morphism $g\in \hom(\mc{D}_1(\mb{Z}_p^n),\nss)$ such that $\psi \co g = f$.
\end{theorem}
\noindent This theorem is a refinement of (and was inspired by) results on finite nilspaces that were obtained by the third named author in \cite{SzegFin}, especially \cite[Theorem 6]{SzegFin}. It is a central ingredient in our proofs of the regularity and inverse theorems, discussed below.

The second main result in Section \ref{sec:gen-set} refines Theorem \ref{thm:general-p-hom-intro} for $k\leq p$, as follows. This uses the so-called \emph{degree-$\ell$} nilspace structure on any abelian group $\ab$, denoted by $\mc{D}_\ell(\ab)$, which is a standard way to turn $\ab$ into an $\ell$-step nilspace; see \cite[Definition 2.2.30]{Cand:Notes1}.
\begin{theorem}\label{thm:intro-3}
Let $p$ be a prime and let $k\in \mb{N}$ with $k \leq p$. Let $\ns$ be a $k$-step $p$-homogeneous compact nilspace. Then for every $\ell\in [k]$ there exists $a_\ell\in \mb{N}\cup\{\infty\}$ such that $\ns$ is isomorphic to the product nilspace\footnote{Here $\mb{Z}_p^{\infty}$ denotes the direct product $\mb{Z}_p^\mb{N}$.} $\prod_{\ell=1}^k \mc{D}_\ell(\mb{Z}_p^{a_\ell})$.
\end{theorem}
\noindent Theorem \ref{thm:intro-3} is a key tool in our applications in ergodic theory. 
Note that the theorem covers the high-characteristic case $k<p$ but also the additional case $k=p$. For $k>p$, we do not yet know what would be a corresponding qualitatively optimal refinement of Theorem \ref{thm:general-p-hom-intro}. This leads to questions that we leave open in this paper; see Remark \ref{rem:open-Qs}.

We close this introduction with a summary of the main applications.

In Section \ref{sec:ergodic} we treat the applications in ergodic theory. Our  results here concern the ergodic measure-preserving actions of $\mb{F}_p^\omega$ studied by Bergelson, Tao and Ziegler in \cite{BTZ,BTZ2}, and specifically the \emph{Host--Kra factors} of such $\mb{F}_p^\omega$-systems, i.e.\ the characteristic factors for uniformity seminorms on these systems.\footnote{Recall from \cite[Definition 1.1]{BTZ} the notion of a $G$-system  for a locally compact abelian group $G$; see also \cite[Definition 5.9]{CScouplings} for a definition of Host--Kra factors valid for $G$-systems more generally.} In the setting of $\mb{Z}$-systems, the Ergodic Structure Theorem from \cite{HK-non-conv} describes the Host--Kra factors as inverse limits of \emph{nilsystems}. In the characteristic-$p$ setting, the Host--Kra factors have hitherto been described in terms of \emph{Weyl systems}, which were defined in \cite[Definition 1.5, Theorem 1.6]{BTZ2} specifically  for this setting (see also \cite[Theorem 4.8]{BTZ}). The different approach in the present paper unifies the descriptions of Host--Kra factors in these two settings via the common notion of a \emph{nilspace system}, introduced in \cite{CScouplings}. Let us recall that a ($k$-step) nilspace system $(\ns,G)$ is a specific type of $G$-system, consisting of a compact ($k$-step) nilspace $\ns$, and a topological group $G$ acting continuously on $\ns$ via a group homomorphism $G \to \tran(\ns)$, where $\tran(\ns)$ is the \emph{translation group}\footnote{The translation group $\tran(\ns)$ is a group of automorphisms naturally defined on any nilspace $\ns$, which can be viewed as a generalization of the regular action of abelian groups on themselves; see \cite[\S 3.2.4]{Cand:Notes1}.} of $\ns$. (The nilspace system can also be specified as a triple $(\ns,G,\phi)$, if the homomorphism $\phi:G\to\tran(\ns)$ needs to be emphasized; we can also add to the data an explicit filtration on $G$, preserved by $\phi$, writing $(\ns,(G,G_\bullet),\phi)$ and calling this a \emph{filtered nilspace system}; see \cite[Definition 5.10]{CScouplings}.) Such a system can be viewed as a topological dynamical system, and if we equip $\ns$ with its Haar probability measure then the nilspace system becomes a measure-preserving $G$-system. Nilspace systems were shown in \cite{CGSS,CScouplings} to yield extensions of the Ergodic Structure Theorem (in particular, ergodic nilspace $\mb{Z}$-systems are inverse limits of nilsystems \cite[Theorem 5.1]{CGSS}). Let us mention that there are other descriptions of the characteristic factors of ergodic $G$-systems. There is for instance the concept of \emph{nilpotent system} introduced in \cite[Definition 1.29]{shalom1} for the 2-step case with $G=\bigoplus_{p\in P}\mb{F}_p$ (where $P$ is a multiset of primes), and more generally, for any countable abelian group $G$ in the 2-step case, there is a description of the 2-factor as a double coset space \cite[Theorem 1.21]{shalom2}.
 
Let us say that a nilspace system is \emph{$p$-homogeneous} if the underlying compact nilspace is $p$-homogeneous. We then have the following result.
\begin{theorem}\label{thm:main-ergodic-intro}
For every $k\in \mb{N}$, the $k$-th Host--Kra factor of every ergodic $\mb{F}_p^{\omega}$-system is isomorphic to a $p$-homogeneous $k$-step ergodic nilspace system $(\ns, \mb{F}_p^\omega)$.
\end{theorem}
\noindent In fact, a property markedly stronger than ergodicity holds here: the action of $\mb{F}_p^\omega$ on $\ns$ is uniquely ergodic,\footnote{The notion of unique ergodicity may be recalled from \cite[p.\ 87, \S 4.3.a.]{HassKat}.} and for every $n$ the standard cube-set $\cu^n(\mb{F}_p^\omega)$ also has a uniquely ergodic action on $\cu^n(\ns)$; see Theorem \ref{thm:main-ergodic-strong}. This stronger form of ergodicity follows from the main results in \cite{CScouplings}, which we combine with Theorem \ref{thm:main1-intro} to prove Theorem \ref{thm:main-ergodic-intro}.

With Theorem \ref{thm:main-ergodic-intro}, the study of these Host--Kra factors can be reduced to the study of $p$-homogeneous nilspace systems. Focusing on the latter systems, we then obtain more precise descriptions of the factors. In particular, for $k\leq p$ we have the following result, showing that the $k$-th factor consists of an elementary abelian $p$-group equipped with a specific degree-$k$ filtration.
\begin{theorem}\label{thm:high-char-ergodic}
Let $(\ns, \mb{F}_p^\omega)$ be an ergodic $p$-homogeneous $k$-step nilspace system with $k\leq p$. Then this system is isomorphic to a nilspace system $\big(\prod_{\ell=1}^k \mc{D}_\ell(\mb{Z}_p^{a_\ell}),\mb{F}_p^\omega\big)$ where $a_\ell\in \mb{N}\cup \{\infty\}$ for each $\ell\in [k]$.
\end{theorem}
\noindent This relates to previous work as follows. In the high-characteristic case $k<p$, the results of \cite{BTZ} describe the $k$-th Host--Kra factor
 as a $k$-fold iterated abelian extension by elementary abelian $p$-groups, with the cocycle for the $j$-th extension being polynomial of degree $\leq j$; 
 see \cite[Corollary 8.7]{BTZ}. Theorem \ref{thm:high-char-ergodic} instead describes the factor as a nilspace system on the explicit product nilspace $\prod_{\ell=1}^k \mc{D}_\ell(\mb{Z}_p^{a_\ell})$, which is a $k$-fold abelian bundle where the $i$-th factor $\prod_{\ell=1}^i \mc{D}_\ell(\mb{Z}_p^{a_\ell})$ is a \emph{splitting extension} of degree $i$ (in the nilspace sense) of the previous factor $\prod_{\ell=1}^{i-1} \mc{D}_\ell(\mb{Z}_p^{a_\ell})$ by the group $\mb{Z}_p^{a_i}$, and this description holds even for $k=p$. We note that the translation group of this product nilspace can be described more explicitly, thus obtaining a complete description of the transformations in this factor in terms of polynomial maps between filtered elementary abelian $p$-groups; see Theorem \ref{thm:HiCharTransDesc}.
 
Theorem \ref{thm:main-ergodic-intro} also enables progress in a closely related direction for these applications in ergodic theory, namely the direction concerning Abramov systems. Recall that an ergodic $\mb{F}_p^\omega$-system is an \emph{Abramov system of order $\leq k$} if its $L^2$-space is the closure of the linear span of phase polynomials of degree at most $k$; see Definition  \ref{def:abramov}. The following interesting question arose in the work of Bergelson, Tao and Ziegler \cite{BTZ}.
\begin{question}\label{Q:Abramov}
Is the $k$-th Host--Kra factor of an ergodic $\mb{F}_p^\omega$-system always an Abramov system of order $\leq k$, for every $k\in \mb{N}$?
\end{question}
\noindent In \cite{BTZ}, an affirmative answer is given in the case $k<p$ (\cite[Theorem 1.19]{BTZ}), and this is believed to hold also for $k\geq p$; see \cite[Remark 1.21]{BTZ}. For $k\geq p$, a partial affirmative answer is given in \cite[Theorem 1.20]{BTZ}, showing that the factor is Abramov of order $O_k(1)$.

We extend the affirmative answer to Question \ref{Q:Abramov} as follows. 
\begin{theorem}\label{thm:Abramov}
For every ergodic $\mb{F}_p^{\omega}$-system and every $k\leq p+1$, the $k$-th Host--Kra factor of the system is Abramov of order $\leq k$.
\end{theorem}
\noindent The proof of this theorem uses a reformulation of Question \ref{Q:Abramov} as a problem purely about nilspaces, which we discuss in Section \ref{sec:ergodic} (see Proposition \ref{prop:Abrachar} and Question \ref{Q:SAphom}) and which seems of interest as a possible way to make further progress on this question.

Finally, Section \ref{sec:inverse} contains applications concerning arithmetic combinatorics, which further illustrate the relevance of $p$-homogeneous nilspaces to higher-order Fourier analysis in characteristic $p$. In particular,  we use these nilspaces to give a new proof of the inverse theorem for Gowers norms on $\mb{F}_p^n$ from \cite{T&Z-Low}, as well as regularity theorems in this setting. The idea of the proof is to start from the more general inverse theorem in the nilspace approach from \cite{CScouplings,CSinverse}, which tells us that a function $f$ with non-trivial Gowers $U^{k+1}$ norm on $\mb{F}_p^n$ correlates with a complex-valued function which factors through a highly-balanced morphism $\phi$ from $\mb{F}_p^n$ to a \textsc{cfr} $k$-step nilspace $\ns$. Thanks to Theorem \ref{thm:intro-1}, we deduce that $\ns$ is $p$-homogeneous. Then Theorem \ref{thm:general-p-hom-intro} enables us to lift the morphism $\phi$ through a simpler finite abelian group nilspace, belonging to the class $\mc{Q}_{p,k}$. This, combined with a standard Fourier decomposition on the abelian group underlying this nilspace, yields a non-classical phase polynomial of degree $\leq k$ correlating non-trivially with the original function $f$, as required. Moreover, for $k\leq p$, using Theorem \ref{thm:intro-3} instead of Theorem \ref{thm:general-p-hom-intro} we obtain the inverse theorem with \emph{classical} polynomials; see Theorem \ref{thm:InvThm-k=p}. Note that the case $k=p$ of this result was obtained only recently, independently, in \cite{Berger&al}.

The proof strategy on $\mb{F}_p^n$ outlined above is similar to the one used in the integer setting in \cite{CSinverse}; both are rooted in the general inverse theorem from \cite{CSinverse}, and the differences arise only once we have to deal with $p$-homogeneous nilspaces here, instead of toral nilspaces in the integer setting. So far, this general strategy does not yield effective bounds. Currently, the proofs of inverse theorems for Gowers norms with the best effective bounds work with much more specific strategies in each setting; see the recent breakthroughs of Manners in the integer setting \cite{M}, and of Gowers and Mili\'cevi\'c in the characteristic-$p$ setting \cite{GM2}. It would be very interesting to know if a more general strategy can also yield effective bounds, perhaps even an effective inverse theorem for general finite abelian groups.

\section{$p$-homogeneous nilspaces as images of highly balanced morphisms on $\mb{Z}_p^n$}\label{sec:bridge}
\noindent The goal of this section is to prove Theorem \ref{thm:main1-intro}. We first summarize the strategy, by formulating its main ingredients in the three propositions below and then explaining how these are combined to prove the theorem. Then the work divides into subsections dedicated to proving each of the propositions.

Before we state the three main propositions, we need to recall some terminology and an important initial assumption related to the notion of balanced morphisms. For any compact metrizable topological space $X$, we denote by $\mc{P}(X)$ the space of Borel probability measures on $X$  equipped with the weak topology, which is then also compact metrizable \cite[Theorem (17.22)]{K}. We can then recall the notion of balanced morphism used in \cite{CSinverse}.

\begin{defn}[Balance]\label{def:balance}
Let $\nss$ be a $k$-step compact nilspace. For each $n\in \mb{N}$ fix a metric $d_n$ on $\mc{P}(\cu^n(\nss))$. Let $\ns$ be a compact nilspace, and let $\phi:\ns\to \nss$ be a (continuous) morphism. Then for $b>0$ we say that $\phi$ is \emph{$b$-balanced} if for every $n\leq 1/b$ we have $d_n\big(\mu_{\cu^n(\ns)}\co(\phi^{\db{n}})^{-1}, \mu_{\cu^n(\nss)}\big)\leq b$.
\end{defn}

\begin{remark}\label{rem:metconv}
The notion of balance thus depends on the choice of metrics $d_n$. Throughout this paper we adopt the following convention: once we have fixed a metric $d$ on a compact nilspace $\nss$, this automatically induces a metric on each cube set $\cu^n(\nss)$, $n\in \mb{N}$ (we choose the $\ell_{\infty}$-metric $(\q,\q')\mapsto \max_{v\in \db{n}} d(\q(v),\q'(v))$), and this in turn induces a metric $d_n$ on $\mc{P}(\cu^n(\nss))$ for each $n$ in a standard way (e.g.\ the L\'evy-Prokhorov metric). Thus, fixing a metric on $\nss$ is enough to fix the notion of balanced morphisms into $\nss$.
\end{remark}
Let us now state the main propositions.
\begin{proposition}\label{prop:b-bal-p-hom-finite}
Let $\nss$ be a $k$-step \textsc{cfr} nilspace, let $d$ be a metric generating the topology on $\nss$, and let $p$ be a prime. There exists $b=b(\nss,d,p)>0$ such that the following holds: if for some $D$ there is a $b$-balanced morphism $\varphi:\mc{D}_1(\mb{Z}_p^D)\to \nss$, then every structure group of $\nss$ is a finite elementary abelian $p$-group, and in particular $\nss$ is a finite set.
\end{proposition}
\noindent The second proposition gives a sufficient condition for a nilspace to be $p$-homogeneous.
\begin{proposition}\label{prop:dimreduc}
Let $p$ be a prime and $k\in \mb{N}$. Then there exists $M\in \mb{N}$ such that the following holds. A $k$-step nilspace $\ns$ is $p$-homogeneous if it satisfies the following property:
\begin{eqnarray}\label{eq:dimreduc}
&&\hspace*{-1cm}\textrm{Every structure group of $\ns$ is an elementary abelian $p$-group, and for all $i\in [k]$,}\\
&&\hspace*{-1cm}\textrm{for every $f\in\hom\big(\mc{D}_1(\mb{Z}_p^M),\ns_i)$ there exists $\tilde{f}\in\hom\big(\mc{D}_1(\mb{Z}_p^M),\ns)$ such that $\pi_i\co \tilde{f}=f$.}\nonumber 
\end{eqnarray}
\end{proposition}
\noindent In the next section we will see with additional tools that, in fact, a converse to this proposition holds as well, so that property \eqref{eq:dimreduc} for $M$ sufficiently large can be used as an equivalent definition of $p$-homogeneous nilspaces; see Proposition \ref{prop:lifting-equiv}.

The last ingredient for Theorem \ref{thm:main1-intro} tells us that if $\ns'$ is $p$-homogeneous and $\varphi'\in \hom(\mc{D}_1(\mb{Z}_p^D), \ns')$ is sufficiently balanced, then any nilspace morphism $\mc{D}_1(\mb{Z}_p^M)\to\ns'$ can be factored through a much simpler morphism $\mc{D}_1(\mb{Z}_p^M)\to \mc{D}_1(\mb{Z}_p^D)$ (this latter morphism is thus an affine linear map $\mb{F}_p^M\to\mb{F}_p^D$).
%The last ingredient for Theorem \ref{thm:main1-intro} enables us to factor a sufficiently balanced nilspace morphism $\mc{D}_1(\mb{Z}_p^M)\to\ns'$ through a much simpler morphism $\mc{D}_1(\mb{Z}_p^M)\to \mc{D}_1(\mb{Z}_p^D)$ (this latter morphism is thus an affine linear map $\mb{F}_p^M\to\mb{F}_p^D$), provided that $\ns'$ is $p$-homogeneous.
\begin{proposition}\label{prop:main}
Let $\ns'$ be a $k$-step \textsc{cfr} $p$-homogeneous nilspace and let $M\in\mb{N}$. Then there exists $b'=b'(\ns',M)>0$ such that the following holds. If for some $D\in \mb{N}$ there is a $b'$-balanced morphism $\varphi'\in \hom(\mc{D}_1(\mb{Z}_p^D),\ns')$, then for every morphism $f\in \hom(\mc{D}_1(\mb{Z}_p^M),\ns')$ there exists a morphism $g\in \hom(\mc{D}_1(\mb{Z}_p^M),\mc{D}_1(\mb{Z}_p^D))$ such that $f=\varphi'\co g$.
\end{proposition}
\noindent Before we go into the proofs of the above three propositions, let us explain how these results imply Theorem \ref{thm:main1-intro}. The following diagram helps to visualize the argument.
\begin{center}
\begin{tikzpicture}
  \matrix (m) [matrix of math nodes,row sep=4em,column sep=4em,minimum width=2em]
  {
           &    & \mc{D}_1(\mb{Z}_p^M) \\
     \mc{D}_1(\mb{Z}_p^D) & \ns &  \\
      & \ns'=\ns_{k-1} &\\};
  \path[-stealth]
    (m-2-1) edge node [above] {$\varphi$} (m-2-2)
    (m-1-3) edge node [above] {$g$} (m-2-1)
    (m-1-3) edge node [below] {$\tilde{f}$} (m-2-2)
    (m-2-1) edge node [left] {$\varphi'=\pi_{k-1}\co \varphi$\;} (m-3-2)
    (m-1-3) edge node [right] {$f$} (m-3-2)
    (m-2-2) edge node [above] {$\pi_{k-1}\qquad\,$} (m-3-2);
\end{tikzpicture}
\end{center}
\begin{proof}[Proof of Theorem \ref{thm:main1-intro}]
We argue by induction on $k$. The case $k=1$ is clear since then by Proposition \ref{prop:b-bal-p-hom-finite} we have $\ns=\mc{D}_1(\mb{Z}_p^m)$ for some $m\in \mb{N}$, a $p$-homogeneous nilspace. For $k>1$, letting $\varphi$ be the supposed $b$-balanced morphism $\mc{D}_1(\mb{Z}_p^D)\to \ns$, let $\varphi'=\pi_{k-1}\co \varphi\in \hom(\mc{D}_1(\mb{Z}_p^D),\ns')$, where $\ns'$ is the $(k-1)$-step nilspace factor, i.e.\ $\ns'=\ns_{k-1}=\pi_{k-1}(\ns)$.

Note that, whatever compatible metric $d'$ we may have fixed on $\ns_{k-1}$, we have that for every $b^*>0$, if $b>0$ is sufficiently small then $\varphi'$ is $b^*$-balanced (relative to the metrics $d'_n$ on $\mc{P}(\cu^n (\ns'))$ induced by $d'$ as per Remark \ref{rem:metconv}). To see this, note that the pushforward map
$\mc{P}(\cu^n(\ns)) \rightarrow \mc{P}(\cu^n(\ns_{k-1}))$, $\nu \mapsto  \nu \co (\pi_{k-1}^{\db{n}})^{-1}$ is continuous, by definition of the weak topology and the continuity of $\pi_{k-1}$. Thus, for every $\epsilon>0$ there is $\delta_n>0$ such that the cube-set Haar measures $\mu_{\cu^n(\ns)}$, $\mu_{\cu^n(\ns_{k-1})}$ satisfy that for every $\nu\in \mc{P}(\cu^n (\ns))$, if $d_n(\nu,\mu_{\cu^n(\ns)})<\delta_n$ then $d_n'(\nu\co(\pi_{k-1}^{\db{n}})^{-1},\mu_{\cu^n(\ns_{k-1})})<\epsilon$ (where we use that $\mu_{\cu^n(\ns_{k-1})} = \mu_{\cu^n(\ns)}\co (\pi_{k-1}^{\db{n}})^{-1}$). Applying this with $\nu=\mu_{\cu^n(\mb{Z}_p^D)}\co(\varphi^{\db{n}})^{-1}$, we deduce that if $b< \min_{n\le 1/b^*} \delta_n(b^*)$, then $d_n'(\mu_{\cu^n(\mb{Z}_p^D)}\co ({\varphi'}^{\db{n}})^{-1},\mu_{\cu^n(\ns_{k-1})})\le b^*$ for all $n\le 1/b^*$. 

Hence, if $b$ is sufficiently small (depending on $\ns$ and in particular on the metric on $\ns_{k-1}$) then by induction $\ns'$ is $p$-homogeneous. By Proposition \ref{prop:b-bal-p-hom-finite}, if $b$ is small enough then the structure groups of $\ns$ are all finite elementary abelian $p$-groups. Then, by Proposition \ref{prop:dimreduc} it suffices to prove that the lifting property in \eqref{eq:dimreduc} holds. We claim that, since $\ns_{k-1}$ is $p$-homogeneous, the lifting property already holds for $i\leq k-1$. Indeed, for any $f\in\hom(\mc{D}_1(\mb{Z}_p^M),\ns_i)$, letting $q:\mb{Z}^M\to \mb{Z}_p^M$ be the natural quotient map, we have $f\co q \in\hom(\mc{D}_1(\mb{Z}^M),\ns_i)$. By Corollary \ref{cor:liftthrufibgen} applied with the fibration $\psi=\pi_i:\ns_{k-1}\to\ns_i$, there exists $f'\in\hom(\mc{D}_1(\mb{Z}^M),\ns_{k-1})$ such that $f\co q = \pi_i\co f'$. As $\ns_{k-1}$ is $p$-homogeneous, the restriction $f'|_{[0,p-1]^n}$ is a morphism $\tilde{f}\in \hom(\mc{D}_1(\mb{Z}_p^M),\ns_{k-1})$, and this morphism is a lift which proves our claim. Hence it now suffices to prove the lifting property for $i=k-1$, that is, that for some $M=M(k,p)$, for every $f\in\hom\big(\mc{D}_1(\mb{Z}_p^M),\ns_{k-1})$ there is $\tilde{f}\in\hom\big(\mc{D}_1(\mb{Z}_p^M),\ns)$ such that $\pi_{k-1}\co \tilde{f}=f$. Fix any such $f$. Applying Proposition \ref{prop:main} to $f$ with $\ns'=\ns_{k-1}$ and $\varphi'=\pi_{k-1}\co\varphi$ (with $b$ small so that $b^*$ is less than the parameter $b'(\ns',M)$ given by that proposition), we obtain $g\in \hom(\mc{D}_1(\mb{Z}_p^M),\mc{D}_1(\mb{Z}_p^D))$ such that $f=\varphi'\co g$. Letting $\tilde f = \varphi\co g$, we have $\pi_{k-1}\co \tilde f= \pi_{k-1}\co\varphi\co g = \varphi'\co g = f$.
\end{proof}

\addtocontents{toc}{\protect\setcounter{tocdepth}{1}}

\subsection{Proof of Proposition \ref{prop:b-bal-p-hom-finite}}\hfill\\
We argue by induction on $k$. The base case $k=0$ is trivial ($\nss$ is then the 1-point nilspace). Throughout this subsection let us denote $\mc{D}_1(\mb{Z}_p^D)$ simply by $\mb{Z}_p^D$ (the nilspace structure used on this group is $\mc{D}_1(\mb{Z}_p^D)$ throughout the proof).

First, by the same argument as in the above proof of Theorem \ref{thm:main1-intro}, for every $b^*>0$, if $b>0$ is sufficiently small then $\pi_{k-1}\co \varphi\in \hom(\mb{Z}_p^D, \nss_{k-1})$ is $b^*$-balanced. Hence, if $b$ is less than the constant $b(\nss_{k-1},d',p)$ (given by Proposition \ref{prop:b-bal-p-hom-finite} for $\nss_{k-1}$), then by induction all the structure groups $\ab_i(\nss)$ with $i<k$ are finite elementary abelian $p$-groups. In particular $\nss_{k-1}$ is a finite set. The proof thus reduces to showing that the last structure group $\ab_k=\ab_k(\nss)$ is also a finite elementary $p$-group. As $\ab_k$ is a (compact abelian) Lie group, by the no-small-subgroups property there is an open neighbourhood $U_0$ of $0\in\ab_k$ such that the only subgroup of $\ab_k$ included in $U_0$ is $\{0\}$. It suffices to show that for $b>0$ sufficiently small we have $pz\in U_0$ for all $z\in \ab_k$, as then the subgroup $p\cdot \ab_k$ is $\{0\}$. 

For $\epsilon>0$ and $\q\in \cu^k(\nss)$ let $B(\q,\epsilon):=\{\q'\in \cu^k(\nss): \forall\, v\in \db{k},\; d_{\nss}(\q'(v),\q(v))<\epsilon\}$.

Fix any $y\in\nss$. For every $z\in \ab_k$, let $\q_z$ denote the cube in $\cu^k(\nss)$ such that $\q_z(1^k)=y+z$ and $\q_z(v)=y$ for all $v\not=1^k$. Let $\epsilon>0$ be a parameter to be fixed later. 

Our first step in this proof is to use the balance property of $\varphi$ to show that
\begin{equation}\label{eq:Prop2.1-Step1}
\forall\, z\in \ab_k,\;\exists \q^*\in \cu^k(\mb{Z}_p^D)\textrm{ such that }\varphi\co\q^*\in B(\q_z,\epsilon).
\end{equation}
This will be a straightforward application of the following basic result.

\begin{lemma}\label{lem:bal->dens}
Let $V$ be a finite set and let $(W,d)$ be a compact metric space. Let $\mu$ be the uniform probability measure on $V$ and let $\nu$ be a strictly positive measure in $\mc{P}(W)$. Let $r$ be a metric on $\mc{P}(W)$. Then for every $\epsilon>0$ there exists $\delta=\delta(W,d,r,\epsilon)>0$ such that if $f:V\to W$ is a map satisfying $r(\mu\co f^{-1},\nu)\leq \delta$, then for every $y\in W$ there exists $x\in V$ such that $d(f(x),y)<\epsilon$. 
\end{lemma}

\begin{proof}
By compactness there is a finite set $F\subset W$ such that $\bigcup_{t\in F} B_{\epsilon/2}(t)=W$, where $B_{\epsilon/2}(t)$ denotes the open ball of radius $\epsilon/2$ and center $t$. For each $t\in F$ let $g_t$ denote the continuous function $y\mapsto \max\{0,\epsilon/2-d(t,y)\}$ on $W$. Since $\nu$ is strictly positive, we have $\int_W g_t\ud\nu>0$ for every $t\in F$. Hence, for some $\delta>0$, if $\kappa\in \mc{P}(W)$ satisfies $r(\kappa,\nu)\leq\delta$, then $\int_W g_t \ud\kappa > 0$  for every $t\in F$. In particular, if $f:V\to W$ satisfies $r(\mu\co f^{-1},\nu)\leq \delta$, then with $\kappa=\mu\co f^{-1}$ we have $\int_W g_t~\ud\kappa = |V|^{-1}\sum_{x\in V}g_t(f(x))>0$, so there exists $x\in V$ with $d(f(x),t)<\epsilon/2$. Now let $y\in W$ be arbitrary. We have that some $t\in F$ satisfies $d(t,y)<\epsilon/2$ and some $x\in V$ satisfies $d(f(x),t)<\epsilon/2$, so $d(f(x),y)<\epsilon$.
\end{proof}

\noindent Applying Lemma \ref{lem:bal->dens} with $V=\cu^k(\mb{Z}_p^D)$, $W=\cu^k(\nss)$, and $f=\varphi^{\db{k}}$, we deduce that there exists $\tilde b=\tilde b(\nss,\epsilon)>0$ (having fixed the metrics as per Remark \ref{rem:metconv}) such that if $\varphi$ is $b$-balanced for $b<\tilde b$, then \eqref{eq:Prop2.1-Step1} holds. Since $\nss_{k-1}$ is a finite set and each fiber $\pi_{k-1}^{-1}(y)$, $y\in \nss_{k-1}$ is compact, we have $\min\{d_{\nss}(x,y):\pi_{k-1}(x)\neq \pi_{k-1}(y)\}>0$. Using this and the fact that each fiber $\pi_{k-1}^{-1}(y)$ is homeomorphic to $\ab_k$, for any open neighbourhood $U_1$ of $0$ in $\ab_k$ we can choose $\epsilon>0$ small enough to ensure that, for each $z\in \ab_k$, the cube $\varphi\co\q^*$ given by \eqref{eq:Prop2.1-Step1} equals $v\mapsto \q_z(v)+q_z(v)$, for some map $q_z:\db{k}\to \ab_k$ such that $q_z(v)\in U_1$ for all $v\in\db{k}$. Also, since the map $q_z$ is a difference of two cubes on $\nss$, we have $q_z\in \cu^k(\mc{D}_k(\ab_k))$ \cite[Definition 3.2.18 and Theorem 3.2.19]{Cand:Notes1}. 

In the second main step of the proof, we shall now extend the cube $\varphi \co \q^*$ in two different ways, thus producing two morphisms that will be used in the final step below (the combinatorial core of the proof) to deduce that $pz\in U_0$ as required. 

Our first extension of $\varphi \co \q^*$ is to a morphism $g\in \hom(\mc{D}_1(\mb{Z}^k),\nss)$. Recall from \cite[Corollary 2.2.17]{Cand:Notes1} that the cube $q_z$ can be extended to a morphism (polynomial map) $f\in \hom(\mc{D}_1(\mb{Z}^k),\mc{D}_k(\ab_k))$ of the form $f(t)=\sum_{w\in \db{k}} a_w\binom{t}{w}$ for $t\in \mb{Z}^k$, where the coefficients $a_w\in\ab_k$ are determined as finite sums and differences of the values of $q_z$ (see \cite[Lemma 2.2.5]{Cand:Notes1}). It follows that for any open neighbourhood $U_2$ of $0\in\ab_k$, we can set the previous neighbourhood $U_1$ so that, if $q_z$ is $U_1$-valued, then the extension $f$ satisfies $f(v)\in U_2$ for all $v=(v\sbr{1},\ldots,v\sbr{k})\in [0,p]^k$. We then define our first extension of $\varphi \co \q^*$ to be $g(v):=y+f(v\sbr{1},\ldots,v\sbr{k})+z\,v\sbr{1}\cdots v\sbr{k}$ in $\hom(\mc{D}_1(\mb{Z}^k),\nss)$.

Our second extension of $\varphi \co \q^*$ consists in extending to a morphism not from $\mb{Z}^k$ but from $\mb{Z}_p^k$. To do this we note that the standard $k$-cube $\q^*$ on $\mb{Z}_p^D$ is trivially extendable to a morphism $h\in \hom(\mb{Z}_p^k,\mb{Z}_p^D)$, namely, if $\q^*(v)=x+v\sbr{1}h_1+\cdots +v\sbr{k}h_k$ for $v\in \db{k}$, then $h(v)=x+v\sbr{1}h_1+\cdots +v\sbr{k}h_k$ for $v\in \mb{Z}_p^k$. The extension is then $\varphi \co h\in \hom(\mb{Z}_p^k,\nss)$.

We now come to the combinatorial core of the proof. Here we shall use the morphisms $g$ and $\varphi \co h$, and concatenations of cubes, to deduce that $pz\in U_0$. We first note that the morphism properties of $g$ and $\varphi \co h$ yield two chains of $p$ consecutive $k$-cubes which are usefully interrelated. More precisely, for each $i\in [p]$ let $\q_i\in \cu^k(\nss)$ be the cube obtained by restricting $g$ to $\{i-1,i\}\times\db{k-1}$, and similarly let $\q_i'\in\cu^k(\nss)$ be the restriction of $\varphi\co h$ to $\{i-1,i\} \times \db{k-1}$, thus obtaining the two chains of cubes $(\q_i)_{i\in [p]}$, $(\q_i')_{i\in [p]}$. We define a map $\wh{\q}$ which combines usefully the four $(k-1)$-cubes that form the ``extreme faces" of these two chains: let $\wh{\q}:\db{k+1}\to\nss$ be defined by $\wh{\q}(0,v_2,\ldots,v_{k+1}) = g(0,v_2,\ldots,v_k) = \q^*(0,v_2,\ldots,v_k)$, $\wh{\q}(1,v_2,\ldots,v_{k},0)=\varphi \co h (p,v_2,\ldots,v_k)= \q^*(0,v_2,\ldots,v_k)$ (where the last equality uses the $p$-periodicity of $h$), and finally $\wh{\q}(1,v_2,\ldots,v_k,1) = g(p,v_2,\ldots,v_k)$. We shall now show that $\wh{\q}\in \cu^{k+1}(\nss)$, using that $g$ and $\varphi\co h$ are morphisms, and using concatenations of cubes. 
To this end, we note the following useful fact: let us define a relation $\sim$ on $\cu^k(\nss)$ by declaring that $\q\sim\q'$ if the map $\tilde{\q}:\db{k+1}\to\nss$,  $\tilde\q(v,0)=\q(v)$, $\tilde\q(v,1)=\q'(v)$ ($v\in\db{k}$) is in $\cu^{k+1}(\nss)$; then the morphism property of $g$ implies that the cubes $\q_i$ defined above satisfy $\q_i\sim \q_j$ for each $i,j\in [p]$ (since there is a $(k+1)$-cube on $\mc{D}_1(\mb{Z}^k)$ with image $\{i-1,i\}\times\db{k-1}$ on one $k$-face and image $\{j-1,j\}\times\db{k-1}$ on the opposite $k$-face). Similarly, the morphism property of $\varphi\co h$ implies that $\q_i'\sim\q_j'$ for each $i,j\in [p]$. Now note that, by concatenation of cubes \cite[Lemma 3.1.7.]{Cand:Notes1}, the relation $\sim$ is transitive (it is also clearly reflexive and symmetric, by the nilspace axioms, so it is an equivalence relation). Applying this transitivity via the cubes $\q_1=\q_1'$ at one end of the chains, we deduce that for every $i\in [p]$ we have $\q_i\sim \q_i'$. Hence, for each $i\in [p]$, the map $\tilde\q_i:\db{k+1}\to\nss$ defined by $\tilde\q_i(v,0)=\q_i(v)$ and $\tilde\q_i(v,1)=\q'_i(v)$ ($v\in \db{k}$) is a cube. Note also that, for each $i\in[p-1]$, the cubes $\tilde \q_i$, $\tilde \q_{i+1}$ are adjacent in the sense that $\tilde \q_i(1,v)=\tilde\q_{i+1}(0,v)$ for every $v\in \db{k}$. Moreover, by the $p$-periodicity of $h$, we have $\varphi\co h(p,v)=\varphi\co h(0,v)$ for every $v\in \db{k-1}$. This implies that the map $\wh{\q}$ defined above is indeed in $\cu^{k+1}(\nss)$ as we claimed, since it is the concatenation of the cubes $\tilde \q_i$, $i\in [p]$. Now, to conclude this combinatorial argument, we note that the map $\db{k+1}\to\nss$ with constant value $y$ is also in $\cu^{k+1}(\nss)$, so the difference between $\wh{\q}$ and this constant cube must be a cube in $\cu^{k+1}\big(\mc{D}_k(\ab_k)\big)$, and must therefore have Gray-code alternating sum equal to 0 (see \cite[Definitions 3.2.18 and 2.2.30]{Cand:Notes1}). Hence, if $\epsilon>0$ is small enough, then the neighbourhood $U_2$ in the construction of $g$ (i.e.\ the smallness of the values of $f$) is small enough so that $pz\in U_0$. This proves that $\ab_k$ is an elementary abelian $p$-group.

To complete the proof of Proposition \ref{prop:b-bal-p-hom-finite}, it now only remains to prove that $\nss$ is finite. Since $\nss$ is a \textsc{cfr} nilspace, its structure groups are compact abelian Lie groups, so they are of the form $\mb{T}^{n_j}\times A_j$ where $n_j\ge 0$ are integers and $A_j$ are finite abelian groups for $j\in [k]$. Since we now know that each structure group is an elementary abelian $p$-group, we have $n_j=0$ for all $j\in [k]$, i.e.\ the structure groups are all finite. Now note, more generally, that if all the structure groups of a $k$-step nilspace $\ns$ are finite then $\ns$ is a finite set. This can be seen by induction on $k$, using the fact that for the factor map $\pi_{k-1}:\ns\to \ns_{k-1}$, each preimage $\pi_{k-1}^{-1}(x)$, $x\in\ns_{k-1}$ is in bijection with the structure group $\ab_k$ (see \cite[\S 3.2.3]{Cand:Notes1}).

This completes the proof or Proposition \ref{prop:b-bal-p-hom-finite}.

\begin{remark}\label{rem:gen-b-bal}
A straightforward adaptation of the above proof yields a generalization of Proposition \ref{prop:b-bal-p-hom-finite} where the 1-step $p$-homogeneous nilspaces $\mc{D}_1(\mb{Z}_p^D)$ can be replaced by the more general class of \textsc{cfr} \emph{w-$p$-homogeneous} nilspaces, introduced in Section \ref{sec:p-hom}. We do not need this generalization in this paper.
\end{remark}

\subsection{Proof of Proposition \ref{prop:dimreduc}}\hfill\\
We prove Proposition \ref{prop:dimreduc} by induction on the step of the nilspace. The argument uses several ingredients, each of which is relevant to $p$-homogeneous nilspaces in themselves.

The first main ingredient is a result which gives a further equivalent description of $p$-homogeneous nilspaces: Proposition \ref{prop:FD-equiv} below. To prove this, we want a useful way to decide whether a given map from the set $[0,p-1]^n\subset \mb{Z}^n$ into a nilspace $\ns$ can be extended to a morphism $\mc{D}_1(\mb{Z}^n)\to \ns$. The following tool will help to obtain a useful sufficient condition for such an extension.
\begin{defn}[Maximal cube]\label{def:max-cube-p}
For each integer $n\geq 0$ and prime $p$ we define the \emph{maximal cube} $\q^*_{p,n}$ as the following element of $\cu^{n(p-1)}(\mc{D}_1(\mb{Z}^n))$:
\begin{equation}\label{eq:maxcube-sp}
\forall\,v\in \db{n(p-1)},\;\, \q^*_{p,n}(v)  :=  \sum_{i=0}^{n-1} \big(v\sbr{i(p-1)+1}+v\sbr{i(p-1)+2}+\cdots+v\sbr{(i+1)(p-1)}\big)\, e_{i+1},
\end{equation}
where $(e_i)_{i\in [n]}$ is the standard basis of $\mb{Z}^n$.
\end{defn}
\noindent The above-mentioned sufficient condition goes as follows.
\begin{lemma}\label{lem:FD-lift}
Let $\ns$ be a nilspace. If $g:[0,p-1]^n\to \ns$ satisfies $g\co \q^*_{p,n}\in \cu^{n(p-1)}(\ns)$, then there exists $f\in \hom(\mc{D}_1(\mb{Z}^n),\ns)$ such that $g=f|_{[0,p-1]^n}$.
\end{lemma}
\noindent This is a special case of a result concerning general nilspace theory rather than just $p$-homogeneous nilspaces. Because of this, we leave the proof for Appendix \ref{app:res-nil} (specifically, Lemma \ref{lem:FD-lift} is the special case of Corollary \ref{cor:liftthrufibgen} with $\nss$ equal to the 1-point nilspace).

As the sufficient condition in Lemma \ref{lem:FD-lift} will be used repeatedly below, let us introduce the following notation for it.
\begin{defn}\label{def:hom-n-p-set}
Let $p$ be a prime, let $n\ge 0$ be an integer and $\ns$ be a nilspace. Then we define the set $\hom_p^n(\ns):=\{f:[0,p-1]^n\to \ns: f\co \q^*_{p,n}\in \cu^{n(p-1)}(\ns)\}$.
\end{defn}

We can now state and prove the first main ingredient for the proof of Proposition \ref{prop:dimreduc}.
\begin{proposition}\label{prop:FD-equiv}
A nilspace $\ns$ is $p$-homogeneous if and only if for every integer $n\geq 0$ we have $\hom_p^n(\ns)\subset \hom(\mc{D}_1(\mb{Z}_p^n),\ns)$.
\end{proposition} 
\begin{proof}
To see the backward implication note that, given $f\in \hom(\mc{D}_1(\mb{Z}^n),\ns)$ we have $f|_{[0,p-1]^n}\in \hom_p^n(\ns)$, so by the assumed inclusion we have $f|_{[0,p-1]^n}\in \hom(\mc{D}_1(\mb{Z}_p^n),\ns)$, so $\ns$ is $p$-homogeneous. The forward implication follows from Lemma \ref{lem:FD-lift}.
\end{proof}
\noindent The second main ingredient is the following result. This strengthens the backward implication in Proposition \ref{prop:FD-equiv}, as the assumption is made only up to a bounded dimension.
\begin{proposition}\label{prop:fordimreduc}
For every prime $p$ and $k\in \mb{N}$ there is $M=M(p,k)>0$ such that the following holds. If a  $k$-step nilspace  $\ns$ satisfies $\hom_p^M(\ns)\subset \hom(\mc{D}_1(\mb{Z}_p^M),\ns)$, then $\ns$ is $p$-homogeneous.
\end{proposition}
\noindent The proof will use the following definition, which extends the notion of discrete cube morphism from \cite{CamSzeg} (see also \cite[Definition 1.1.1.]{Cand:Notes1}).
\begin{defn}
Let $p$ be a prime. A function $\phi:[0,p-1]^n\to [0,p-1]^m$ is a \emph{$p$-discrete-cube morphism} if it is the restriction of an affine homomorphism $\mb{Z}^n\to \mb{Z}^m$. A $p$-discrete-cube morphism $\phi:[0,p-1]^n\to [0,p-1]^m$ with $n\le m$ is a \emph{$p$-face-map} of dimension $n$ if it is injective and fixes $m-n$ coordinates.
\end{defn}
\noindent The case $p=2$ of this definition yields the usual discrete cube morphisms. It can be seen (e.g.\ by a straightforward adaption of the proof of \cite[Lemma 1.1.2.]{Cand:Notes1}) that each coordinate function $\phi_j(v)$, $j\in [m]$, is either $v_{i_j}$, or $p-1-v_{i_j}$ for some $i_j\in [n]$, or is a constant $k\in\{0,\ldots,p-1\}$.

The proof of Proposition \ref{prop:fordimreduc} relies on the following couple of lemmas.
\begin{lemma}\label{lem:fordimreduc-1}
Let $f\in \hom_p^n(\ns)$ and suppose that for every $T\in \hom(\mc{D}_1(\mb{Z}_p^m),\mc{D}_1(\mb{Z}_p^n))$ we have $f\co T\in \hom_p^m(\ns)$. Then $f\in \hom(\mc{D}_1(\mb{Z}_p^n),\ns)$.
\end{lemma}

\begin{proof}
By definition $f\in \hom(\mc{D}_1(\mb{Z}_p^n),\ns)$ if for every $\q\in \cu^m(\mc{D}_1(\mb{Z}_p^n))$ we have $f\co \q\in \cu^m(\ns)$. As $\mc{D}_1(\mb{Z}_p^n)$ is $p$-homogeneous we know that $\q$ extends to a morphism $T\in \hom(\mc{D}_1(\mb{Z}_p^m),\mc{D}_1(\mb{Z}_p^n))$. By our assumption we have $f\co T\in \hom_p^m(\ns)$. Since $f\co \q=f\co T \co i$, where $i:\db{m}\to \mb{Z}_p^m$ is the identity embedding (which is in $\cu^m(\mb{Z}_p^m)$) we have $f\co \q\in \cu^m(\ns)$ as required.
\end{proof}

\begin{lemma}\label{lem:fordimreduc-2}
Let $\ns$ be a $k$-step nilspace and $n\ge k+1$. Let $f:[0,p-1]^n\to \ns$ satisfy $f\co \phi\in \hom_p^{k+1}(\ns)$ for every $p$-face-map $\phi:[0,p-1]^{k+1}\to [0,p-1]^n$. Then $f\in \hom_p^n(\ns)$.
\end{lemma}
\noindent This second lemma has a longer and more technical proof and concerns nilspaces more generally, so we leave it to Appendix \ref{app:res-nil}; see Lemma \ref{lem:fordimreduc-2-app}.

\begin{proof}[Proof of Proposition \ref{prop:fordimreduc}] Let $M(p,k):=p^{k+2}$. By Proposition \ref{prop:FD-equiv}, it is enough to check that for all $n\ge 0$ we have $ \hom_p^n(\ns)\subset \hom(\mc{D}_1(\mb{Z}_p^n),\ns)$. Thus let $f$ be any element of $\hom_p^n(\ns)$, and let us distinguish the following two cases.

If $n\leq M$, then consider the map $\phi:[0,p-1]^M\to [0,p-1]^n$, $(v_1,\ldots,v_n,\ldots,v_M)\mapsto (v_1,\ldots,v_n)$, which clearly extends to a morphism $\mc{D}_1(\mb{Z}^M)\to \mc{D}_1(\mb{Z}^n)$. It follows that $f\co \phi\in \hom_p^M(\ns)$. Hence by our assumption $f\co \phi\in \hom(\mc{D}_1(\mb{Z}_p^M),\ns)$. Now consider $\psi:[0,p-1]^n\to [0,p-1]^M$, $(v_1,\ldots,v_n)\mapsto (v_1,\ldots,v_n,0,\ldots,0)$. As $\psi\in \hom(\mc{D}_1(\mb{Z}_p^n),\mc{D}_1(\mb{Z}_p^M))$, we have that $f=f\co\phi\co \psi \in \hom(\mc{D}_1(\mb{Z}_p^n),\ns)$, as required.

If $n>M$, then by Lemma \ref{lem:fordimreduc-1} it suffices to check that for the given $f\in \hom_p^n(\ns)$ and any $T\in \hom(\mc{D}_1(\mb{Z}_p^m),\mc{D}_1(\mb{Z}_p^n))$, we have $f\co T\in \hom_p^m(\ns)$. For this, by Lemma \ref{lem:fordimreduc-2} it suffices to check that for every $p$-face-map $\phi:[0,p-1]^{k+1}\to [0,p-1]^m$ we have $f\co T\co \phi\in \hom_p^n(\ns)$. Note that $\phi\in \hom(\mc{D}_1(\mb{Z}_p^{k+1}),\mc{D}_1(\mb{Z}_p^m))$. Therefore $T\co \phi \in \hom(\mc{D}_1(\mb{Z}_p^{k+1}),\mc{D}_1(\mb{Z}_p^n))$, so for each $i\in [n]$ the $i$-th coordinate of $T\co \phi$ is an affine-linear map $\mb{Z}_p^{k+1}\to \mb{Z}_p$ of the form $(v_1,\ldots,v_{k+1})\mapsto a_0^{(i)}+a_1^{(i)}v_1+\cdots+ a_{k+1}^{(i)}v_{k+1}$ for some coefficients $a_0^{(i)},\ldots,a_{k+1}^{(i)}$. In total there are $p^{k+2}=M(p,k)$ possible such maps. Therefore there is an affine-linear map $T^*\in \hom(\mc{D}_1(\mb{Z}_p^{k+1}),\mc{D}_1(\mb{Z}_p^M))$, whose coordinates are all these possible maps, and there is a $p$-discrete-cube morphism $\psi:[0,p-1]^M\to [0,p-1]^n$ that selects the correct entry of $T^*$ in order to have $T\co \phi = \psi\co T^*$ (actually $\psi$ can be given as an $n\times M$ matrix with each row having some entry equal to 1 and all others equal to 0). Since $f\co \psi\in \hom_p^M(\ns)$, by hypothesis we have $f\co\psi\in \hom(\mc{D}_1(\mb{Z}_p^M),\ns)$. Hence $ f\co T\co \phi  = (f\co \psi)\co T^*\in \hom(\mc{D}_1(\mb{Z}_p^{k+1}),\ns)$.
\end{proof}
\noindent We need one more tool for the proof of Proposition \ref{prop:dimreduc}, namely the following lemma esta- blishing the $p$-homogeneity of elementary abelian $p$-groups with higher-degree filtrations.

\begin{lemma}\label{lem:elemabpolycase}
Let $\ab$ be an elementary abelian $p$-group, and let $k\in \mb{N}$. Then the $k$-step nilspace $\mc{D}_k(\ab)$ is $p$-homogeneous.
\end{lemma}
\begin{proof}
We check that Definition \ref{def:p-hom-intro} holds for $\ns=\mc{D}_k(\ab)$. We know (see e.g.\ \cite[Theorem 2.2.14]{Cand:Notes1}) that $f$ is in $\hom(\mc{D}_1(\mb{Z}^n),\mc{D}_k(\ab))$ if and only if $f$ is a degree-$k$ polynomial map $\mb{Z}^n\to \ab$. Moreover (e.g.\ by \cite[Lemma A.1]{GTar}), for each $\mf{i}\in \mb{N}^n$ with \emph{height} $|\mf{i}|:= i_1+\cdots + i_n\leq k$, there is an element $a_{\mf{i}}\in \ab$ such that $f(\mf{n})=\sum_{\mf{i}\in \mb{N}^n:|\mf{i}|\leq k} a_{\mf{i}} \binom{\mf{n}}{\mf{i}}$ for every $\mf{n}\in\mb{Z}^n$. For each $\mf{i}$ we have $p\, a_{\mf{i}}=0$. This implies that the restriction $f|_{[0,p-1]^n}$ yields a well-defined map $\mb{Z}_p^n\to \ab$, which is readily seen to be a polynomial of degree at most $k$.
\end{proof}

We can now complete the main goal of this section.

\begin{proof}[Proof of Proposition \ref{prop:dimreduc}]
We argue by induction on $k$. The case $k=0$ is trivial. Let $M(p,k)=p^{k+2}$. We need to prove that every $f\in \hom_p^{M(p,k)}(\ns)$ is in $\hom(\mc{D}_1(\mb{Z}_p^{M(p,k)}),\ns)$.

First we prove by induction that $\ns_{k-1}$ can be assumed to be $p$-homogeneous. Since $M(p,k)\ge M(p,k-1)$, it suffices to show that for every $g\in \hom(\mc{D}_1(\mb{Z}_p^{M(p,k-1)}),\ns_i)$ there exists $\tilde{g}\in \hom(\mc{D}_1(\mb{Z}_p^{M(p,k-1)}),\ns_{k-1})$ such that $g = \pi_i\co \tilde{g}$. Writing elements of $\mb{Z}_p^{M(p,k)}$ as $(v,w)$ with $v\in \mb{Z}_p^{M(p,k-1)}$, let $\phi:\mb{Z}_p^{M(p,k)}\to \mb{Z}_p^{M(p,k-1)}$, $(v,w)\mapsto v$ (the projection to the first $M(p,k-1)$ coordinates). Then $g\co \phi\in \hom(\mc{D}_1(\mb{Z}_p^{M(p,k)}),\ns_i)$. By hypothesis, there exists $h\in \hom(\mc{D}_1(\mb{Z}_p^{M(p,k)}),\ns)$ that lifts $g\co \phi$. Let $i:\mb{Z}_p^{M(p,k-1)}\to \mb{Z}_p^{M(p,k)}$ be the inclusion map $v\mapsto (v,0)$. Then $h\co i$ is in $\hom(\mc{D}_1(\mb{Z}_p^{M(p,k-1)}),\ns)$ and lifts $g$, i.e.\ $\pi_i\co h\co i =g$. Hence $\tilde{g}:=\pi_{k-1}\co h\co  i$ is in $\hom(\mc{D}_1(\mb{Z}_p^{M(p,k-1)}),\ns_{k-1})$ and $\pi_i\co \tilde{g}= g$.

Now consider $\pi_{k-1}\co f\in \hom_p^{M(p,k)}(\ns_{k-1})$. As $\ns_{k-1}$ is $p$-homogeneous, by Proposition \ref{prop:FD-equiv} we have $\pi_{k-1}\co f\in \hom(\mc{D}_1(\mb{Z}_p^{M(p,k)}),\ns_{k-1})$. By hypothesis, there exists $\tilde{f}\in \hom(\mc{D}_1(\mb{Z}_p^{M(p,k)}),\ns)$ that lifts $\pi_{k-1}\co f$, i.e., $\pi_{k-1}\co\tilde{f} = \pi_{k-1}\co f$. This implies that $f-\tilde{f}\in \hom_p^{M(p,k)}(\mc{D}_k(\ab_k(\ns)))$. Since $\ab_k(\ns)$ is an elementary abelian $p$-group, by Lemma \ref{lem:elemabpolycase} the nilspace $\mc{D}_k(\ab_k(\ns))$ is $p$-homogeneous. Then by Proposition \ref{prop:FD-equiv} we deduce that $f-\tilde{f} \in \hom\big(\mc{D}_1(\mb{Z}_p^{M(p,k)}),\mc{D}_k(\ab_k(\ns))\big)$. Hence $f = \tilde{f}+(f-\tilde{f})\in \hom(\mc{D}_1(\mb{Z}_p^{M(p,k)}),\ns)$.
\end{proof}

\subsection{Proof of Proposition \ref{prop:main}}\hfill\\
\noindent To simplify the notation, we shall prove Proposition \ref{prop:main} with $\ns'$ relabeled as $\ns$. 

\addtocontents{toc}{\protect\setcounter{tocdepth}{2}}

We begin with a useful construction of a function with small $U^k$-norm on $\mb{Z}_p^D$ using a $p$-homogeneous nilspace $\ns$: it suffices to compose a highly balanced morphism $\mb{Z}_p^D\to\ns$ with a function on $\ns$ that has average 0 on every orbit of the structure group $\ab_k(\ns)$.

\begin{lemma}\label{lem:XcharUknorm}
Let $\ns$ be a $k$-step \textsc{cfr} $p$-homogeneous nilspace with a compatible metric. For any $\eta>0$ there is $b=b(\ns,p,\eta)>0$ such that if $\varphi\in \hom(\mc{D}_1(\mb{Z}_p^D),\ns)$ is $b$-balanced, then for any 1-bounded function\footnote{A complex-valued function $h$ is said to be \emph{1-bounded} if $|h|\leq 1$ everywhere on the domain of $h$.} $f:\ns\to \mb{C}$ such that $\forall\,x\in\ns$, $\int_{\ab_k} f(x+z)\ud\mu_{\ab_k}(z)=0$, we have $\|f\co \varphi\|_{U^k}\le \eta$.
\end{lemma}
\begin{proof}
By Proposition \ref{prop:b-bal-p-hom-finite}, if $b$ is small enough then $\ns$ is finite. 

Suppose for a contradiction that there exists $\eta_0>0$ such that for all $n\in\mb{N}$ we have a $\frac{1}{n}$-balanced morphism $\varphi_n$ and a 1-bounded function $f_n:\ns\to \mb{C}$ such that $\|f_n\co \varphi_n\|_{U^k}\ge \eta_0$. As the set $T:=\{f:\ns\to \mb{C}:|f|\le 1\textrm{ and }\forall x\in\ns,\, \int_{\ab_k} f(x+z)\ud\mu_{\ab_k}(z)=0\}$ is compact, we may assume without loss of generality that for some $f\in T$ we have $\max_{x\in \ns} |f_n(x)- f(x)|\to 0$ as $n\to \infty$. Then
$\eta_0 \le \|f_n\co \varphi_n\|_{U^k}\le \|(f_n-f)\co \varphi_n\|_{U^k}+\|f\co \varphi_n\|_{U^k}$. For large enough $n$, we have $\|(f_n-f)\co \varphi_n\|_{U^k}\le \frac{\eta_0}{3}$ and, since $\varphi_n$ is $\frac{1}{n}$-balanced, also 
$\left| \int \prod_{v\in \db{k}} C^{|v|} f\co\varphi_n(\q(v)) \ud\mu_{\cu^k(\mb{Z}_p^D)}(\q)- \int \prod_{v\in \db{k}} C^{|v|} f(\q(v)) \ud\mu_{\cu^k(\ns)}(\q)
\right| \le \frac{\eta_0}{3}$, where $C$ denotes the complex-conjugation operator.

By construction of this Haar measure \cite{Cand:Notes2}, this last integral equals
\[
\int_{\cu^k(\ns)}\int_{\cu^k(\mc{D}_k(\ab_k))} \prod_{v\in \db{k}} C^{|v|} f(\q(v)+\q'(v)) \ud\mu_{\cu^k(\mc{D}_k(\ab_k))}(\q')\ud\mu_{\cu^k(\ns)}(\q).
\]
But we know that $\cu^k(\mc{D}_k(\ab_k))$ is just the direct product $\ab_k^{\db{k}}$ with its Haar measure being the $\db{k}$-power of the Haar measure on $\ab_k$. Hence for each fixed $\q$ the inner integral above is
$\prod_{v\in \db{k}} C^{|v|}  \int_{\ab_k} f(\q(v) +z)  \ud\mu_{\ab_k}(z) =0$. 
This yields a contradiction.
\end{proof}
\noindent Now let us turn to the core of the proof of Proposition \ref{prop:main}. Let us recall briefly the situation. We have the abelian group $A:=\hom(\mc{D}_1(\mb{Z}_p^M),\mc{D}_1(\mb{Z}_p^D))$ of affine homomorphisms $\mb{Z}_p^M\to \mb{Z}_p^D$. Note that $A\cong  \mb{Z}_p^D\oplus \mb{Z}_p^{D\times M}$, so that in particular we can represent the elements of $A$ as $(M+1)$-tuples $(x,t_1,\ldots,t_M)\in (\mb{Z}_p^D)^{M+1}$, where such a tuple represents uniquely the affine homomorphism $g_{(x,t=(t_1,\ldots,t_M))}\in A$ defined by
\[
g_{(x,t)}(z)=x+t \cdot z = x+t_1z_1+\cdots +t_M z_M,\textrm{ for }z\in\mb{Z}_p^M.
\]
(The $t_i$ can also be viewed as the $M$ columns of a matrix $\tau\in \mb{Z}_p^{D\times M}$ defining the linear map $z\mapsto \tau z$, and then $g_{(x,t)}$ is this linear map composed with translation by $x$.) This mapping of elements of $A$ to $(M+1)$-tuples $\big(x,t=(t_1,\ldots,t_M)\big)$ is a group isomorphism.

On the other hand we have the abelian bundle $B:=\hom(\mc{D}_1(\mb{Z}_p^M),\ns)$, where $\ns$ is a $k$-step $p$-homogeneous nilspace, and we have a $b$-balanced morphism $\varphi\in \hom(\mc{D}_1(\mb{Z}_p^D),\ns)$.

We then consider the map $F:A\to B$, $g\mapsto \varphi\co g$, and our goal is to prove that if $\varphi$ is $b$-balanced with $b$ sufficiently small (where $b$ can depend on $\ns$), then $F$ is surjective. 

Note that for each element $(x,t)$ in $A$, the image $F(x,t)=\varphi\co g_{(x,t)}$ can be viewed as a point in $\ns^{\mb{Z}_p^M}$, namely as the point
$F\big(x,t=(t_1,\ldots,t_M)\big)= \big(\varphi(x+t_1 z_1+\cdots+t_Mz_M) \big)_{z\in \mb{Z}_p^M}$. 

The first step in our strategy is to use induction on $k$ to reduce the task of proving that $F$ is surjective onto $B$ to the task of proving a simpler looking statement about the distribution of the orbit $(F(x,t))_{(x,t)\in A}$ in $k$-level fibers of some power of $\ns$. The reduction goes as follows. Let $f$ be any element in $B$. Then, by induction on $k$ (using that $\pi_{k-1}\co\varphi$ is $b'(b)$-balanced where $b'\to 0$ as $b\to 0$), there is an element $g\in A$ such that  the map $\pi_{k-1}\co\varphi\co g$ is equal to $\pi_{k-1}\co f$. In other words, our goal of showing that $f$ is the image of some $g\in A$ under $F$ is already achieved modulo $\pi_{k-1}$, i.e.\ we have $\pi_{k-1}\co\varphi\co g = \pi_{k-1}\co f$. Then, since $\varphi\co g(z)$ and $f(z)$ are in the same $\pi_{k-1}$-fiber in $\ns$ for every $z\in \mb{Z}_p^M$, we can take the difference of these two maps, which must then be a morphism into the $k$-th structure group $\ab_k$ of $\ns$, namely $q:z \; \mapsto \; \varphi\co g(z) - f(z)\, \in \, \hom\big(\mc{D}_1(\mb{Z}_p^M),\mc{D}_k(\ab_k)\big)$. In other words, this map $q$ is a $\ab_k$-valued polynomial of degree at most $k$ in $M$ variables.

Let us identify $\mb{Z}_p^M$ with $[0,p-1]^M$, and define for $r\in \mb{N}$ the  set
\begin{equation}\label{eq:simp}
S_{r,M}:= [0,p-1]^M_{< r}=\{z\in [0,p-1]^M: |z|:= z_1+\cdots+z_M < r\}.
\end{equation}
Then $q$ is entirely determined by its values on $S_{k+1,M}$. Actually, this holds for more general morphisms, in the following sense which will be used later.
\begin{lemma}\label{lem:polydet}
Let $\ns$ be a $k$-step nilspace and let $q\in\hom(\mc{D}_1(\mb{Z}_p^M),\ns)$. Then the values of $q$ on $S_{k+1,M}$ determine the full map $q$.
\end{lemma}
\begin{proof}
We use the uniqueness of completion of $(k+1)$-corners on $\ns$. To this end, for $z\in S_{k,M}$, we call  $|z|:=z_1+\cdots+z_M$ the \emph{height} of $z$, as in the previous subsection.

If all elements in $[0,p-1]^M$ have height at most $k$ (which happens when $(p-1)M\leq k$) then there is nothing to prove (since $q|_{S_{k+1,M}}$ is then already the full map $q$). Otherwise, we argue by induction on the height. We start with any $z\in [0,p-1]^M$ satisfying $|z|=k+1$. Note that there is a $(k+1)$-cube $\q$ on $\mc{D}_1(\mb{Z}_p^M)$ such that $|\q(v)|\leq k$ for $v\neq 1^{k+1}$ and $\q(1^{k+1})=z$. Indeed, letting $z_1,\ldots,z_M$ be the coordinates of $z$, we can take $\q(v)=v\sbr{1} h_1 + \cdots + v\sbr{k+1}h_{k+1}$, where, out of the $k+1$ elements $h_i\in [0,p-1]^M$, we set the first $z_1$ of them to be equal to $e_1$, the next $z_2$ of them to be equal to $e_2$, etc., the last $z_M$ of them to be equal to $e_M$. Then, letting $\q'$ be the $(k+1)$-corner obtained by restricting $\q$ to $\db{k+1}\setminus\{1^{k+1}\}$, we have that $q\co \q'$  is a $(k+1)$-corner on $\ns$, so its completion is unique, and this completion is $q\co \q$ by the morphism property, so the value $q(z)=q\co \q(1^{k+1})$ is uniquely determined. 

Repeating this argument with every $z$ of height $k+1$, then by induction with height $k+2$, and so on, we determine all values $q(z)$, $z\in [0,p-1]^M$.
\end{proof}
\noindent Recall that in our ongoing argument, the polynomial $q$ is the ``error" that we would like to correct in $\varphi\co g$ to get the desired map $f$ as an image $\varphi\co g'$ and thus prove the surjectivity claimed in Proposition \ref{prop:main}. Thanks to Lemma \ref{lem:polydet}, we can focus on correcting errors in $S_{k+1,M}$. It then suffices to prove the following quantitative equidistribution result.
\begin{lemma}\label{lem:charfreekey}
Let $\ns$ be a $k$-step \textsc{cfr} $p$-homogeneous nilspace with a compatible metric, and let $S=S_{k+1,M}$. For every $\varepsilon>0$, there exists $b>0$ such that the following holds. Let $\varphi\in \hom(\mc{D}_1(\mb{Z}_p^D),\ns)$ be $b$-balanced, let $(x,t=(t_1,\ldots,t_M))\in A:=\hom(\mc{D}_1(\mb{Z}_p^M),\mc{D}_1(\mb{Z}_p^D))$, and let $y=\big( \varphi(x+t\cdot z)\big)_{z\in S} \in \ns^S$. Then the map $F':A\to \ns^S,\; (x',t')\,\mapsto\, \big( \varphi(x'+t'\cdot z)\big)_{z\in S}$ is $\varepsilon$-equidistributed in the fiber $y + \ab_k^S \subset \ns^S$, in the following sense:
\begin{equation}\label{eq:charfreekey} 
\forall\,w\in \ab_k^S,\quad \Big| \frac{\mu_A\big(F'^{-1}(y+w)\big)}{\mu_A\big(F'^{-1}(y+\ab_k^S)\big)} - \frac{1}{|\ab_k^S|}\Big| < \varepsilon.
\end{equation}
\end{lemma}
\noindent Note that the quantity $\frac{\mu_A\big(F'^{-1}(y+w)\big)}{\mu_A\big(F'^{-1}(y+\ab_k^S)\big)}$ in \eqref{eq:charfreekey} is a conditional probability of $F'(x,t)$ equalling $y+w$ given that $F'(x,t)$ is in the fiber $y+\ab_k^S$. Thus \eqref{eq:charfreekey} tells us that this probability is $\varepsilon$-close to the Haar-probability of the singleton $\{y+w\}$ in this fiber.

To see that Lemma \ref{lem:charfreekey} implies Proposition \ref{prop:main}, recall that so far we had found $g=g_{x,t_1,\ldots,t_M}$ such that $f(z) = \varphi\co g(z)-q(z)$, and now we just need to ``correct" the polynomial difference $q(z)$ in order to conclude the desired surjectivity. By Lemma \ref{lem:charfreekey} and the finiteness of $\ab_k$ (given by Proposition \ref{prop:b-bal-p-hom-finite}), for $\varepsilon$ sufficiently small (namely $\varepsilon < \frac{1}{|\ab_k^S|}$), the $\varepsilon$-equidistribution implies surjectivity \emph{in this $\ab_k^S$-fiber}, so there is $(x',t')\in A$ such that $\big( \varphi(x'+t'\cdot z)\big)_{z\in S} = \big( \varphi(x+t\cdot z)-q(z)\big)_{z\in S} =  \big( f(z)\big)_{z\in S}$. In other words, letting $g'$ be the element of $A$ corresponding to $(x',t')$, we now have that the morphisms $\varphi\co g'$ and $f$ in $\hom(\mc{D}_1(\mb{Z}_p^M),\ns)$ agree on the simplicial set $S\subset \mb{Z}_p^M$. But then by Lemma \ref{lem:polydet} we deduce that $\varphi\co g'=f$ on all of $\mb{Z}_p^M$, which gives our desired surjectivity conclusion, completing the proof of Proposition \ref{prop:main}.

Let us now turn to the proof of Lemma \ref{lem:charfreekey}. 

\begin{defn}\label{def:W}
We define $W=W_{\ns,k,p,M}$ to be the vector space of functions $h:\ns^S\to\mb{C}$ with the property that for every point $y\in \ns^S$ we have $\int_{\ab_k^S} h(y+w)\ud\mu_{\ab_k^S}(w)=0$.
\end{defn}
\noindent Note that $W$ has finite dimension because by Proposition \ref{prop:b-bal-p-hom-finite} we know that $\ns$ is finite. The dimension of $W$ thus depends on $|\ns|$, and note that it also increases as $M$ grows (but this poses no problem, as $M$ will be fixed in terms of $p$ and $k$).

Our next step is to reduce the proof of Lemma \ref{lem:charfreekey} to establishing the following result.

\begin{proposition}\label{prop:keyWeq}
Let $\ns$ be a $k$-step $p$-homogeneous \textsc{cfr} nilspace with a compatible metric, and let $W$ be the vector space in Definition \ref{def:W}. For every $\delta>0$, there exists $b=b(\ns,M,\delta)>0$  such that if $\varphi\in \hom(\mc{D}_1(\mb{Z}_p^D),\ns)$ is $b$-balanced then
\begin{equation}\label{eq:keyWeq}
\forall\, h\in W\textrm{ with }\|h\|_\infty\leq 1,\quad\textrm{we have }\quad \Big|\mb{E}_{x,t_1,\ldots,t_M\in \mb{Z}_p^D}\; h\Big(\big(\varphi(x+t\cdot z)\big)_{z\in S}\Big)\Big| \leq \delta.
\end{equation}
\end{proposition}
\begin{proof}[Proof of Lemma \ref{lem:charfreekey} using Proposition \ref{prop:keyWeq}]
Suppose that Lemma \ref{lem:charfreekey} fails for some $\varepsilon>0$. Then this failure is witnessed by a point in some $\ab_k^S$-fiber; more precisely, there exist $x,t_1,\ldots,t_M \in \mb{Z}_p^D$ such that, letting $y=\big(\varphi(x+t\cdot  z)\big)_{z\in S}$, there is a point in the fiber $y+\ab_k^S$, i.e., some point $y' = y + w$ for some $w\in \ab_k^S$, such that the following holds: let $\mu_A'$ denote the measure of the preimage ${F'}^{-1}(y')$ in $A$ \emph{conditioned} on the event ${F'}^{-1}(y+\ab_k^S)$, that is $\mu_A'({F'}^{-1}(y'))= \mu_A({F'}^{-1}(y')) / \mu_A({F'}^{-1}(y+\ab_k^S))$. Then $\big|\mu'_A({F'}^{-1}(y))-\frac{1}{|\ab_k^S|}\big|\geq \varepsilon$. 

Suppose that $\mu_A'({F'}^{-1}(y))\leq \frac{1}{|\ab_k^S|}- \varepsilon$ (the case $\mu_A'({F'}^{-1}(y))\geq \frac{1}{|\ab_k^S|} + \varepsilon$ is handled similarly). Note that there is some other point $\tilde y$ in the same fiber such that $\mu_A'({F'}^{-1}(\tilde{y}))\geq \frac{1}{|\ab_k^S|}- \frac{\varepsilon}{2}$. Indeed, otherwise we would have $1=\mu_A'({F'}^{-1}(y+\ab_k^S))\leq \sum_{w\in \ab_k^S} \mu_A'({F'}^{-1}(y+w)) \leq 1-\frac{\varepsilon}{2}|\ab_k^S|$, a contradiction.

Now let $h$ be the function on $\ns^S$ which is 0 in every $\ab_k^S$-fiber other than the fiber $y + \ab_k^S$, and in this fiber let $h(y')=-1$, $h(\tilde y)=1$, and $h=0$ otherwise. Then clearly $h\in W$. However, for this function $h$ the conclusion \eqref{eq:keyWeq} fails because the left side of \eqref{eq:keyWeq} for $h$ is $
\mu_A ({F'}^{-1}(y'))- \mu_A ({F'}^{-1}(\tilde y)) = \mu_A({F'}^{-1}\big(y+\ab_k^S)\big)\Big(\mu_A' ({F'}^{-1}(y'))- \mu_A' ({F'}^{-1}(y))\Big)\geq  \mu_A({F'}^{-1}\big(y+\ab_k^S)\big) \varepsilon/2$,
so indeed \eqref{eq:keyWeq} fails with $\delta=\mu_A({F'}^{-1}\big(y+\ab_k^S)\big)\,\varepsilon/2$. Now we want this $\delta$ to depend on $\varepsilon$, and perhaps $\ns$ and $M$, but not on $D$, so we have to ensure that the quantity $\mu_A({F'}^{-1}\big(y+\ab_k^S)\big)$ is bounded away from $0$ independently of $D$. 

By induction we may assume that Lemma \ref{lem:charfreekey} holds for step at most $k-1$. Let $F'_{k-1}:A\to \ns_{k-1}^{S_{k-1}}$, and $y_{k-1}:=(\pi_{k-1}(\varphi(x+t\cdot z))_{z\in S_{k-1}}$. Then for $b$ small enough we have
\[
\left| 
\frac{\mu_A(F'_{k-1}(y_{k-1}+w))}{\mu_A(F'_{k-1}(y_{k-1}+\ab_{k-1}^{S_{k-1}}))}-\frac{1}{|\ab_{k-1}^{S_{k-1}}|}
\right|\le \frac{1}{2|\ab_{k-1}^{S_{k-1}}|}
\]
(recall that $\pi_{k-1}\co \varphi$ is $b_{k-1}$-balanced with $b_{k-1}(b)\to 0$ as $b\to 0$). 

Using that $F'^{-1}(y+\ab_k^S)\supset F_{k-1}'^{-1}(y_{k-1})$ we have
\[
\mu_A(F'^{-1}(y+\ab_k^S))\ge \mu_A(F_{k-1}'^{-1}(y_{k-1})) \ge \frac{1}{2|\ab_{k-1}^{S_{k-1}}|}\mu_A(F'_{k-1}(y_{k-1}+\ab_{k-1}^{S_{k-1}})).
\]
\noindent Thus, repeating this argument iteratively we conclude that, for $b$ small enough,
\[
\mu_A(F'^{-1}(y+\ab_k^S))\ge \frac{1}{2|\ab_{k-1}^{S_{k-1}}|}\cdots \frac{1}{2|\ab_1^{S_{1}}|}
\]
which is a quantity that depends on $\ns$, $p$ and $M$, but not on $D$.

We thus deduce that the conclusion of Proposition \ref{prop:keyWeq} fails, as required.
\end{proof}

\noindent We now turn to the proof of \eqref{eq:keyWeq}. For this purpose let us first observe the following useful decomposition of functions in $W$.

\begin{lemma}\label{lem:Wdecomp}
Let $\ns$ be a $k$-step \textsc{cfr} $p$-homogeneous nilspace. For every 1-bounded function $h$ in $W$ we have a decomposition $h=\sum_{r\in [R]} h_r$ where $R=R(\ns,p,M)$ and $h_r: y\mapsto \prod_{z\in S} h_{r,z}(y_z)$ is a rank 1 function such that each $h_{r,z}$ is 1-bounded and, for each $r$, for some $z'=z'(h,r)\in S$, the function $h_{r,z'}$ has integral 0 in each $\ab_k$-fiber.
\end{lemma}
\begin{proof}
Since $\ns$ is finite, there are finitely many fibers $y+\ab_k^S$ partitioning $\ns^S$. Denoting these fibers $F_1,\ldots,F_N$, we have $h=h 1_{F_1} + \cdots + h 1_{F_N}$, where each $h 1_{F_j}$ has 0 average on $F_j$ and is 0 outside $F_j$. Therefore it suffices to show that given a \emph{single} such fiber, every 1-bounded function with 0 average on \emph{this} fiber is a sum of rank 1 functions $y\mapsto \prod_{z\in S} h_z(y_z)$ with every $h_z$ being 1-bounded and with at least one of the $h_z$ having average 0 on this fiber (we can then extend $h_z$ by 0 to a function on all of $\ns$, which then clearly has average 0 on every $\ab_k$-fiber in $\ns$).

Thus, we have reduced the problem to proving the lemma in the case where all of $\ns$ is a single $\ab_k$-fiber: let $\ns$ be a finite set, let $K$ be a positive integer (we will apply this with $K=|S|$),  and let $W$ be the vector space of functions with average 0 on all of $\ns^K$; then every 1-bounded function in $W$ is a sum of rank-1 functions where at least one of the factor-functions has 0 average on $\ns$ and all factor-functions are 1-bounded.

We can prove this claim by induction on $K$. For $K=1$ the claim holds tautologically. Suppose then that the claim holds for the similarly-defined vector space $W'\leq \mb{C}^{\ns^{K-1}}$. Take a 1-bounded function $h:\ns^K=\ns^{K-1} \times \ns\to \mb{C}$ having average 0. Suppose we could show that this is a sum of functions of the form $(y',y)\in \ns^{K-1} \times \ns \mapsto h_1(y')h_2(y)$ where at least one of $h_1,h_2$ has 0 average and both are 1-bounded. Then if it is $h_2$ that has 0 average we are done (as $h_1$ is a sum of 1-bounded rank-1 functions by standard results) and if it is $h_1$ that has 0 average then by induction $h_1$ is a sum of rank 1 functions with 1-bounded factor-functions, one of which has 0 average, so together with $h_2$ we get rank 1 functions decomposing $h$ as required.

Thus, we have reduced the problem even more, to proving that if $\ns,\nss$ are finite sets and $f: \ns \times \nss \to \mb{C}$ has 0 average and is 1-bounded, then it is a combination of rank 1 functions $u(x)v(y)$ where $u$ and $v$ are both 1-bounded and one of them has average 0. Note that any such function $f$ is a sum of functions $f'$ which equal some value $\alpha\in[-1,1]$ in some entry $(x_0,y_0)$, then $-\alpha$ in some other entry $(x_1,y_1)$, and 0 everywhere else (just make a cycle of such functions $f'$ through consecutive pairs of entries, the second non-zero entry of one such $f'$ being corrected by the first non-zero entry of the next such $f'$). Hence it suffices to show that any of these functions $f'$ is a sum of rank 1 functions of the desired kind. But if $x_0=x_1$ or $y_0=y_1$ then $f$ is already a rank 1 function of the desired kind (for example if $x_0=x_1$ then $f'(x,y)=\alpha \,1_{x_0}(x)v(y)$ where $v(y_0)=1=-v(y_1)$ and $v(y)=0$ otherwise). If instead $x_0\neq x_1$ and $y_0\neq y_1$, then $f'(x,y)=u_1(x)v_1(y)+u_2(x)v_2(y)$ where $u_1=\alpha(1_{x_0}-1_{x_1})$,  $v_1=1_{y_0}$, and $u_2=1_{x_1}$, $v_2=\alpha(1_{y_0}-1_{y_1})$. This completes the proof.
\end{proof}
\noindent Given Lemma \ref{lem:Wdecomp}, to prove Proposition \ref{prop:keyWeq} we can first use the decomposition $h=\sum_{r\in [R]} h_r$ where for each $r$ there is a system of 1-bounded functions $(h_{r,z}:\ns\to\mb{C})_{z\in S}$ with some $h_{r,z}$ having average 0 in each $\ab_k$-fiber and $h_r(y)=\prod_{z\in S}h_{r,z}(y_z)$. Hence the right side of \eqref{eq:keyWeq} is at most
\begin{equation}\label{eq:keyWeq2}
\sum_{r\in [R]}\Big| \mb{E}_{x,t_1,\ldots,t_M\in \mb{Z}_p^D} \prod_{z\in S}h_{r,z}\co\varphi(x+t_1 z_1+\cdots+t_M z_M) \Big|.
\end{equation}
Note that $R$ depends on the dimension of $W$, hence on $|\ns^S|$, and this depends on $\ns$ (in particular on $k$) but also on $p$ and $M$.

Now by Lemma \ref{lem:XcharUknorm}, for each $r$, the 1-bounded function $h_{r,z'}:\ns\to\mb{C}$ that has average 0 on every $\ab_k$-fiber satisfies $\|h_{r,z'}\co \varphi\|_{U^k}\leq \eta$ for $b$ small enough (where $\varphi$ is the initial $b$-balanced morphism in $\hom(\mc{D}_1(\mb{Z}_p^D),\ns)$). Note that, by Lemma \ref{lem:XcharUknorm}, the parameter $b$ depends only on $\eta,p,k,M$ and $\ns$ but not on the particular function $h_{r,z'}$.

Therefore, the proof is now completed by applying the following result, which extends the Generalized Von Neumann Theorem and was proved recently in \cite[Theorem 1.10]{CGSS-seq-CS}.
\begin{theorem}\label{thm:GGVN}
Let $p$ be a prime, let $M\in \mb{N}$ and let $k\in [M(p-1)]$. Then there exists $c>0$ such that for every collection of 1-bounded functions $(f_z:\mb{Z}_p^D\to\mb{C})_{z\in S_{k+1,M}}$,
\begin{equation}\label{eq:GGVN}
\Big| \mb{E}_{x,t_1,\ldots,t_M\in \mb{Z}_p^D} \prod_{z\in S_{k+1,M}} f_z(x+t_1z_1+\cdots +t_M z_M)  \Big| \leq  \min_{z\in S_{k+1,M}} \|f_z\|_{U^k}^c.
\end{equation}
\end{theorem}
\noindent Indeed, this result implies that the sum in \eqref{eq:keyWeq2} is at most $R\,\eta^c$, and then we can take $b$ to be small enough (in terms of everything that goes into $R$ and $c$, i.e.\ $k$, $p$, $\ns$, and $M$) so that $R\eta^c\leq \delta$, as required in the conclusion of Proposition \ref{prop:keyWeq}. 

This completes the proof of Proposition \ref{prop:main}.

\begin{remark}
In Proposition \ref{prop:main}, for general step $k$ the parameter $b'$ must depend on the dimension $M$. This can be seen using the following fact from nilspace theory (not detailed in this paper): if $\varphi$ is a cube-surjective morphism from a finite nilspace $\ns$ to a nilspace $\nss$, then $\varphi$ must be a fibration. Using this, we see that if $b'$ were independent of $M$, then we could deduce that the morphism $\varphi'$ is a fibration, which would force the step of $\ns'$ to be at most the step of $\mc{D}_1(\mb{Z}_p^D)$, a contradiction if $k>1$.
\end{remark}

\section{Further properties of $p$-homogeneous nilspaces}\label{sec:p-hom}

\noindent In this section we prove additional results about $p$-homogeneous nilspaces, which will be used in the next section to obtain the main structure theorem. It turns out that some of these results hold for a (potentially) larger class of nilspaces, which we call \emph{weak-$p$-homogeneous} (or \emph{w-$p$-homogeneous}) nilspaces.

\begin{defn}[w-$p$-homogeneous nilspace]
Let $p$ be a prime. A nilspace $\ns$ is \emph{w-$p$-homogeneous} if for every $\q\in \cu^n(\ns)$ there exists $f\in \hom(\mc{D}_1(\mb{Z}_p^n),\ns)$ such that $f|_{\db{n}} = \q$. 
\end{defn}
\noindent The fact that every $p$-homogeneous nilspace is a w-$p$-homogeneous nilspace is a direct consequence of the following fact.

\begin{lemma}\label{lem:ext-cube} Let $\ns$ be a nilspace and let $\q\in \cu^n(\ns)$ for any $n\ge 0$. Then there exists a morphism $f\in\hom(\mc{D}_1(\mb{Z}^n),\ns)$ such that $f|_{\db{n}}=\q$.
\end{lemma}
\noindent Since this lemma concerns general nilspaces, we leave its proof to Appendix \ref{app:res-nil} (specifically, Lemma \ref{lem:ext-cube} is the special case of Corollary \ref{cor:liftthrufib} with $\nss$ equal to the 1-point nilspace).

\begin{remark}
For $p=2$ the notions of $p$-homogeneous and w-$p$-homogeneous nilspaces are readily seen to be equivalent. For $p>2$ we do not know whether every w-$p$-homogeneous nilspace is $p$-homogeneous. Within certain classes of nilspaces, we can prove that the two notions are indeed equivalent. This holds for example for group nilspaces, as established in Theorem \ref{thm:strong-p-hom-group} below. It can also be proved that every 2-step w-$p$-homogeneous nilspace is $p$-homogeneous (since this is not used in the sequel, we omit the details). Thus, this paper leaves open the following question, which seems of independent interest despite not being crucial for our purposes here.
\end{remark}
\begin{question}
Is every w-$p$-homogeneous nilspace also $p$-homogeneous, for all primes $p$?
\end{question}
\noindent Our first result about w-$p$-homogeneous nilspaces is that they are generalizations of elementary abelian $p$-groups, in the sense of the following result, which immediately implies Proposition \ref{prop:p-hom-str-gps-intro}.
\begin{proposition}\label{prop:strgps}
Let $\ns$ be a $k$-step w-$p$-homogeneous nilspace and let $i\in [k]$. Then $\ab_i(\ns)$ is an elementary abelian $p$-group. If $\ns$ is also a \textsc{cfr} nilspace, then $\ns$ is finite.
\end{proposition}

\begin{proof}
We argue by induction on $k$. For $k=0$ the result is trivial ($\ns$ is then the 1-point nilspace). For $k>0$, by induction it suffices to prove that if $\ns$ is of step $k$, then every element of its $k$-th structure group $\ab_k$  has order $p$. Fix any $z\in\ab_k$ and any $x\in \ns$. Let $g\in \hom(\mc{D}_1(\mb{Z}^k),\ns)$ be the morphism defined by $g(v\sbr{1},\ldots,v\sbr{k}):=x+v\sbr{1}\cdots v\sbr{k}z$ (this is indeed a morphism, since it is the composition of the polynomial map $\mb{Z}^k\to\ab_k$, $v\mapsto v\sbr{1}\cdots v\sbr{k}z$ with the morphism $z\mapsto x+z$). For each $i\in [0,p-1]$, let $\q_i\in \cu^k(\ns)$ be the cube obtained by restricting $g$ to $\db{k-1}\times\{i,i+1\}$, that is, we have $\q_i(1^{k-1},0)=x+i z$, $\q_i(1^k)=x+(i+1)z$, and $\q(v)=x$ otherwise. We have the adjacency of cubes $\q_i\prec \q_{i+1}$ for each $i\in [0,p-2]$. Moreover, we have the following useful property: define a relation $\sim$ on $\cu^k(\ns)$ by declaring that $\q\sim\q'$ if the map $\tilde{\q}:\db{k+1}\to\ns$,  $\tilde\q(v,0)=\q(v)$, $\tilde\q(v,1)=\q'(v)$ ($v\in\db{k}$) is in $\cu^{k+1}(\ns)$; then the fact that $g$ is a morphism implies that $\q_i\sim \q_j$ for each $i,j\in [0,p-1]$ (since there is a $(k+1)$-cube on $\mb{Z}^k$ with image $\db{k-1}\times\{i,i+1\}$ on one $k$-face and image $\db{k-1}\times\{j,j+1\}$ on the opposite $k$-face). 

By w-$p$-homogeneity, there is a morphism $f\in\hom(\mc{D}_1(\mb{Z}_p^k),\ns)$ with $f|_{\db{k}}=\q_0$. 
For each $i\in [0,p-1]$, let $\q_i'\in\cu^k(\ns)$ similarly be the restriction of $f$ to $\db{k-1}\times \{i,i+1\}$. As above, the morphism property implies that $\q_i'\sim\q_j'$ for each $i,j\in [0,p-1]$. Now, repeating the argument using concatenations that was the combinatorial core of the proof of Proposition \ref{prop:b-bal-p-hom-finite}, we deduce here that $pz=0$.

Finally, if $\ns$ is also a \textsc{cfr} nilspace, then its finiteness is deduced exactly as in the end of the proof of Proposition \ref{prop:b-bal-p-hom-finite}.
\end{proof}
\noindent We shall now work toward the proof of Theorem \ref{thm:intro-1}, establishing the equivalence between $p$-homogeneity of a \emph{group} nilspace and $p$-homogeneity of the associated filtration. For one of the directions in this equivalence, we shall in fact prove the following more general result giving a similar algebraic property for w-$p$-homogeneous \emph{coset} nilspaces. Recall that a coset nilspace consists of a coset space $G/\Gamma$ where $G$ is a group with a filtration $G_\bullet$  and $\Gamma$ is a subgroup of $G$, the cubes being the projections to $G/\Gamma$ of the Host--Kra cubes in $\cu^n(G_\bullet)$, thus every $n$-cube on $G/\Gamma$ is of the form $v\mapsto \q(v)\Gamma$ for some $\q\in\cu^n(G_\bullet)$; see \cite[Proposition 2.3.1]{Cand:Notes1} (in particular, filtered nilmanifolds are central examples of coset nilspaces). Given a group $H$ and $k\in\mb{N}$ we denote by $H^k$ the set of $k$-th powers $\{h^k:h\in H\}$.
\begin{lemma}\label{lem:p-hom-group}
Let $p$ be a prime, let $(G,G_{\bullet})$ be a filtered group, and let $\Gamma$ be a subgroup of $G$. Let $\ns=G/\Gamma$ be the associated coset nilspace. If $\ns$ is w-$p$-homogeneous then for every $m\ge 0$ we have $G_m^p\subset G_{m+p-1}\Gamma$.
\end{lemma}
\noindent Recall (e.g.\ from \cite[Lemma B.9]{GTZ}) that for $\underline{i}= (i_1,\ldots,i_m)\in \mb{Z}_{\geq 0}^m$ and $\underline{n} = (n_1,\ldots,n_m) \in\mb{Z}^m$, the multiparameter binomial coefficient $\binom{\underline{n}}{\underline{i}}\in\mb{Z}$ is defined by $\binom{\underline{n}}{\underline{i}}= \binom{n_1}{i_1}\cdots \binom{n_m}{i_m}$. Recall also (e.g.\ from the proof of Lemma \ref{lem:polydet}), that we define the \emph{height} of $\underline{i}$ to be $|\underline{i}|=i_1+\cdots+i_m$. We recall the following useful properties of the polynomial maps $\underline{n}\mapsto h_{\underline{i}}^{\binom{\underline{n}}{\underline{i}}}$ (where for $m\in \mb{Z}$ and $h\in G$ we write $h^m$ for the $m$-th power  $h\cdot h\cdots h$ in $G$):
\begin{enumerate}[leftmargin=0.8cm]
    \item\label{it:i} The map $\underline{n}\mapsto h_{\underline{i}}^{\binom{\underline{n}}{\underline{i}}}$ is in $\poly(\mc{D}_1(\mb{Z}^m),G)$ if and only if $h_{\underline{i}}\in G_{|\underline{i}|}$.
    \item\label{it:ii} If $n_j<i_j$ for some $j\in [m]$,  then $h_{\underline{i}}^{\binom{\underline{n}}{\underline{i}}}=\id$.
    \item\label{it:iv} Let $\underline{i}\in \{0,\ldots,p\}\times \{0,1\}^{m-1}$ and $\underline{n}=(p,\underline{v})$ for some $\underline{v}\in \{0,1\}^{m-1}$. If $v_j<i_j$ for some $j\in \{2,\ldots,m\}$, then $h_{\underline{i}}^{\binom{\underline{n}}{\underline{i}}}=\id$. Otherwise $h_{\underline{i}}^{\binom{\underline{n}}{\underline{i}}} =h_{\underline{i}}^{\binom{(p,i_2,\ldots,i_m)}{\underline{i}}} = h_{\underline{i}}^{\binom{p}{i_1}}$.
\end{enumerate}
\begin{proof}[Proof of Lemma  \ref{lem:p-hom-group}]
We prove by induction on $j$ that for all $j\in [0,p-1]$, for every $m\ge 0$ we have $G_m^p\subset G_{m+j}\Gamma$. The case $j=0$ is trivial, so we assume that $G_m^p\subset G_{m+j}\Gamma$ for all $m\ge 0$ and we need to show that $G_m^p\subset G_{m+j+1}\Gamma$ for all $m\ge 0$.

Fix any $m\ge 0$ and $g\in G_m$. Consider the cube $\q\in\cu^m(G/\Gamma)$ such that $\q(1^{m})=g\Gamma$ and $\q(\underline{v})=\Gamma$ for all $\underline{v}\not=1^m$. As $G/\Gamma$ is w-$p$-homogeneous, there exists an extension $f\in \hom(\mc{D}_1(\mb{Z}_p^m), G/\Gamma)$ such that $f|_{\db{m}} = \q$. The proof will go as follows: using Lemma \ref{lem:taylor} we shall construct a morphism $f' = \prod h_{\underline{i}}^{\binom{\underline{n}}{\underline{i}}}\Gamma\in \hom(\mc{D}_1(\mb{Z}^m), G/\Gamma)$ such that $f'|_{\{0,\ldots,p\}\times \{0,1\}^{m-1}} = f|_{\{0,\ldots,p\}\times \{0,1\}^{m-1}}$, and we will deduce the desired conclusion from the resulting expression of $f'$.

We construct $f'$ in three steps, where the second step involves an iterative argument. Rather than accumulating notation for each new modified version of the function that we produce in the argument, we just use the same notation $f'$ throughout the process, which means that $f'$ denotes a different function as we progress through the argument (essentially, each round of the iteration modifies the previous function $f'$ by multiplying it on the left by polynomial maps $h_{\underline{i}}^{\binom{\underline{n}}{\underline{i}}}$).

\textbf{Step 1:} Define $f'(\underline{n}):=g^{n_1\cdots n_m}\Gamma$. This function $f'$ has the following important features:
\begin{itemize}
    \item $f'|_{\db{m}} = f|_{\db{m}}=\q$.
    \item $f'(p,1^{m-1})=g^p\Gamma$ and $f'(p,\underline{v})=\Gamma$ for all $\underline{v}\not=1^{m-1}$.
\end{itemize}
\noindent It may be useful to have an example of the process we are applying. For $m=3$ and $p=3$, at this stage $f'|_{\{0,1,2,3\}\times \{0,1\}^2}$ looks as follows.
\begin{center}
\includegraphics[width=90mm]{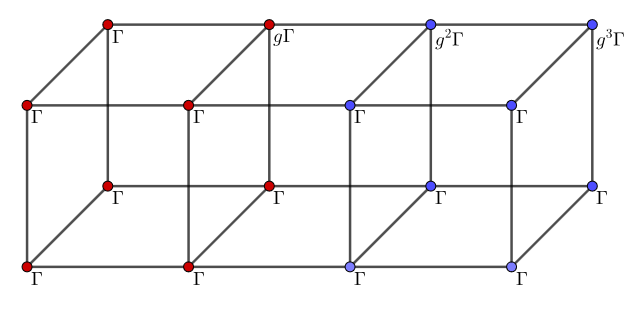}
\end{center}

\noindent In red we have the vertices where $f'=f$. The values of $f'$ at these vertices will not change for the rest of the proof.

\textbf{Step 2:} this step involves an inductive argument, each round of which is an operation that we  call \emph{correcting the line at \underline{v}}, for $\underline{v}\in \db{m-1}\setminus \{1^{m-1}\}$. Let us describe this process.

\textbf{Correcting the line at $\underline{v}$:} suppose that for all $\underline{w}\in \db{m-1}\setminus\{\underline{v}\}$ such that $w_j\le v_j$ for all $j\in \{2,\ldots,m\}$ (instead of labeling the elements for $j\in \{1,\ldots,m-1\}$, the elements of $\db{m-1}$ will be labeled for $j\in \{2,\ldots,m\}$) we have already done the operation of \emph{correcting the line at $\underline{w}$}. Furthermore, suppose that $f'(p,1^{m-1})=\gamma g^p\Gamma$ (for some $\gamma\in\Gamma$ that may not be equal in all the rounds of the process)  and $f'(p,\underline{t})=\Gamma$ for all $\underline{t}\not=1^{m-1}$. By Lemma \ref{lem:taylor}, we can multiply $f'$ (on the left) by elements $h_{(s,\underline{v})}^{\binom{\underline{n}}{(s,\underline{v})}}$ in such a way that the product agrees with $f$ first at the vertex $(2,\underline{v})$, then at $(3,\underline{v})$, and so on all the way to $(p,\underline{v})$. Thus, let $f''(\underline{n}):=h_{(p,\underline{v})}^{\binom{\underline{n}}{(p,\underline{v})}}\cdots h_{(2,\underline{v})}^{\binom{\underline{n}}{(2,\underline{v})}}f'(\underline{n})$.

The above properties of the polynomials $h_{\underline{i}}^{\binom{\underline{n}}{\underline{i}}}$ imply the following useful facts:
\begin{itemize}[leftmargin=0.6cm]
    \item Let $\underline{n}\in \mb{Z}_{\ge 0}^m$ with $n_j < v_j$ for some $j\in [m]$. Then $f''(\underline{n})=f'(\underline{n})$, by property \eqref{it:ii}. In particular, correcting the line at $\underline{v}$ preserves the previous corrections of lines at $\underline{w}$.
    \item $f''(p,1^{m-1})=\gamma  g^p\Gamma$ for some $\gamma\in\Gamma$, and $f''(p,\underline{t})=\Gamma$ for all $\underline{t}\in \db{m-1}\setminus\{1^{m-1}\}$. Indeed, when we multiply by the last factor $h_{(p,\underline{v})}^{\binom{\underline{n}}{(p,\underline{v})}}$, since this ensures that $f''(p,\underline{v})=f(p,\underline{v})=\Gamma$ and $f'(p,\underline{v})=\Gamma$, we must have that $h_{(p,\underline{v})}h_{(p-1,\underline{v})}^{\binom{(p,\underline{v})}{(p-1,\underline{v})}}\cdots h_{(2,\underline{v})}^{\binom{(p,\underline{v})}{(2,\underline{v})}}$ is an element $\gamma''\in\Gamma$. Moreover, property \eqref{it:iv} above implies that for all $\underline{t}\in \db{m-1}\setminus\{\underline{v}\}$ we have $h_{(p,\underline{v})}^{\binom{(p,\underline{t})}{(p,\underline{v})}} h_{(p-1,\underline{v})}^{\binom{(p,\underline{t})}{(p-1,\underline{v})}}\cdots h_{(2,\underline{v})}^{\binom{(p,\underline{t})}{(2,\underline{v})}}\in\Gamma$, indeed this product is $\gamma''$ if $t_i\geq v_i$ for all $i\in [2,m]$, and is the identity otherwise. In particular, we have $f''(p,1^{m-1})=\gamma''f'(p,1^{m-1})=\gamma'' \gamma' g^p\Gamma$ (where $\gamma'$ comes from previous line corrections). Hence our claim holds with $\gamma=\gamma''\gamma'$.
\end{itemize}

\noindent To conclude \emph{correcting the line at \underline{v}}, we set $f''$ to be the new $f'$. 

To complete Step 2, we now correct the lines at $\underline{v}$ for all $\underline{v}\in \db{m-1}$. In order to be able to apply Lemma \ref{lem:taylor} in this process, these corrections have to be done in an order such that for $\underline{v},\underline{v}'\in \db{m-1}$, if $v_j\le v'_j$ for all $j\in\{2,\ldots,p\}$, then we correct the line at $\underline{v}$ before  we  correct the line at $\underline{v}'$ (we can take the lexicographic order, for example).

To visualize this with our example above, after correcting the line at $(0,0)$ we would have $f'$ as follows.
\begin{center}
\includegraphics[width=90mm]{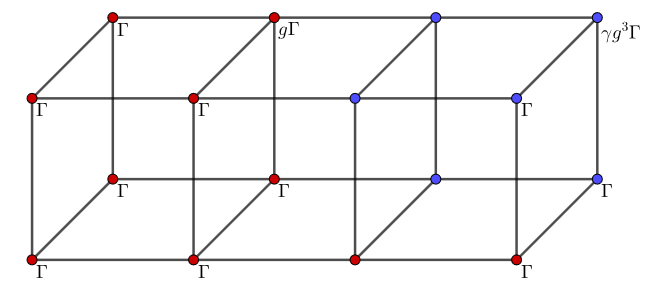}
\end{center}
\noindent Again, the vertices in red represent the ones at which $f'=f$ and whose values will not change for the rest of the proof. After the next two corrections, $f'$ looks as follows (recall that the value of $\gamma$ may be different in each appearance).
\begin{center}
\includegraphics[width=90mm]{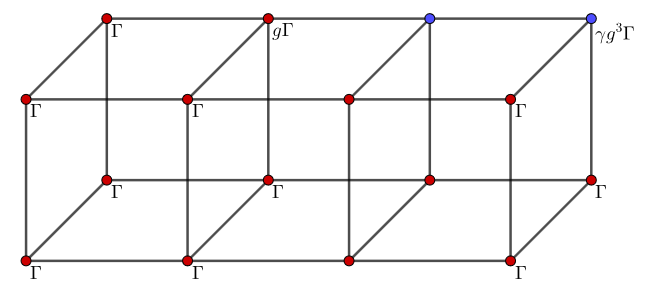}
\end{center}
\textbf{Step 3:} we now correct the line at $1^{m-1}$. The final properties that we obtain are different this time. The important part is now the value $f'(p,1^{m-1})$. By construction we have  $f'(p,1^{m-1}) = f(p,1^{m-1})=\Gamma$, which implies
\begin{equation}
    h_{(p,1^{m-1})}h_{(p-1,1^{m-1})}^{\binom{p}{p-1}} \cdots h_{(2,1^{m-1})}^{\binom{p}{2}}\gamma g^p\in \Gamma,
\end{equation}
where the terms $h_{(j,1^{m-1})}^{\binom{p}{j}}$ are the factors involved in correcting the line at $1^{m-1}$, and $\gamma\in\Gamma$.

Now, using the induction hypothesis, the fact that $\binom{p}{j}$ is a multiple of $p$ for $j\in [2,p-1]$, and the fact that each group $G_\ell$ is normal in $G$, we deduce that $h_{(p-1,1^{m-1})}^{\binom{p}{p-1}} \cdots h_{(2,1^{m-1})}^{\binom{p}{2}} \in G_{|(2,1^{m-1})|+j}\Gamma=G_{m+j+1}\Gamma$. Since the last term $h_{(p,1^{m-1})}$ is already in $G_{m+p-1}\Gamma$, we conclude that $g^p\in G_{\min(m+j+1,m+p-1)}\Gamma = G_{m+j+1}\Gamma$. This completes the inductive step.
\end{proof}

\begin{remark}\label{rem:p-hom-coset} 
By a straightforward generalization of the above proof it can be shown that if a coset nilspace $G/\Gamma$ is w-$p$-homogeneous then, more generally, for all $\ell,m\ge 0$ we have $G_m^{p^\ell}\subset G_{m+\ell(p-1)}\Gamma$. We omit the details as this will not be needed in the sequel.
\end{remark}
\noindent We are now ready to prove Theorem \ref{thm:intro-1}, which we restate here in a refined form.
\begin{theorem}\label{thm:strong-p-hom-group} Let $(G,G_{\bullet})$ be a filtered group and let $\ns$ be the associated group nilspace. Then the following properties are equivalent:
\begin{enumerate}
    \item $\ns$ is $p$-homogeneous.
    \item $\ns$ is w-$p$-homogeneous.
    \item The filtration $G_\bullet$ is $p$-homogeneous.
\end{enumerate}
\end{theorem}
\begin{proof}
The implication $(i)\Rightarrow (ii)$ follows from Lemma \ref{lem:ext-cube} as observed at the beginning of this section. The implication $(ii)\Rightarrow (iii)$ follows from Lemma \ref{lem:p-hom-group} applied with the trivial subgroup $\Gamma = \{\id\}$.

We now prove $(iii)\Rightarrow (i)$, using the following strategy. Given any $f\in\hom(\mc{D}_1(\mb{Z}^n),\ns)$,  note that there exists a $p$-periodic morphism $g_0\in \hom(\mc{D}_1(\mb{Z}_p^n),\ns)$ such that $g_0(0^n)=f(0^n)$ (we can take $g_0$ to be the constant map with value $f(0^n)$). Then, writing $g_0^{-1}$ for the map sending each $z$ to the inverse of the group element $g_0(z)$, we have $fg_0^{-1}(0^n)=\id$. Now suppose that there exists $g_1\in \hom(\mc{D}_1(\mb{Z}_p^n),\ns)$ such that $g_1(0^n)=\id$ and $g_1(1,0,\ldots,0)=(fg_0^{-1})(1,0,\ldots,0)$. Then $fg_0^{-1}g_1^{-1}$ is in $\hom(\mc{D}_1(\mb{Z}_p^n),\ns)$ and equals $\id$ at $0^n$ and at $(1,0,\ldots,0)$. Repeating this process, we will end up with a morphism $fg_0^{-1}g_1^{-1}\cdots g_\ell^{-1}\in \hom(\mc{D}_1(\mb{Z}^n),\ns)$ such that $(fg_0^{-1}g_1^{-1}\cdots g_\ell^{-1})|_{[0,p-1]^n} =\id$, thus showing that $f|_{[0,p-1]^n}=g_0\cdots g_\ell\in \hom(\mc{D}_1(\mb{Z}_p^n),\ns)$ as required.

We now prove that each step of the process can be carried out. First, for $i\ge 1$ we define the group nilspace $H^{(p)}_i:=\mb{Z}$ with filtration 
\begin{equation}\label{eq:gnsdef}
\left(H^{(p)}_i\right)_j = 
     \begin{cases}
       \mb{Z} &\quad\text{if } j=0,1,\ldots,i\\
       p^{\lfloor\frac{j-i-1}{p-1}\rfloor+1}\mb{Z} &\quad\text{if } j\ge i+1. \\ 
     \end{cases}
\end{equation}
\noindent
Now, for $i\in \{0,\ldots,p-1\}$ we define $m_i^{(p)}:\mc{D}_1(\mb{Z})\to H^{(p)}_i$ as follows:
\[   
m_i^{(p)}(x) = 
     \begin{cases}
       0 &\quad\text{if } (x)^*_p\in \{0,1,\ldots,i-1\}\\
       (-1)^{(x)^*_p-i}\binom{p-i-1}{(x)^*_p-i} &\quad\text{if } (x)^*_p\in\{i,\ldots,p-1\} \\
     \end{cases}
\]
where $(x)^*_p$ is the residue modulo $p$ of $x$ in $[0,p-1]$. It is easy to see that\footnote{Where $\partial_1 f(x):=f(x+1)-f(x)$ for any $f:\mb{Z}\to Z$ where $Z$ is an abelian group.} $\partial_1 m_i^{(p)} = m_{i-1}^{(p)}$ for $i\ge 1$. Moreover, for all $i\in[0,p-1]$ we have that $\partial_1^{i+1} m_i^{(p)}$ is a circular vector (viewed as the element $(\partial_1^{i+1} m_i^{(p)}(0),\ldots,\partial_1^{i+1} m_i^{(p)}(p-1))$ of $\mb{Z}^p$, see Definition \ref{def:circular-vector}) such that all its entries are multiples of $p$. To see this, note that $\partial_1^{i+1} m_i^{(p)} = \partial_1^{p-1} m_{p-2}^{(p)}$ and we can apply Corollary \ref{cor:der-circ} to the latter. Indeed, applying this Corollary as many times as required we get that $m_i^{(p)}$ is a morphism (with the filtrations $\mc{D}_1(\mb{Z})$ and $(H_i^{(p)})_{\bullet}$, by \cite[Theorem 2.2.14]{Cand:Notes1}). Note also that by construction each morphism $m_i^{(p)}$ is a $p$-periodic map on $\mb{Z}$.

These morphisms $m_i^{(p)}$ will be the basic tool to define the morphisms $g$ mentioned above. But first, it is convenient to see how to use them in dimension larger than 1. Let $n\in \mb{N}$ and, given a vector of indices $\underline{t}=(t_1,\ldots,t_n)\in \{0,\ldots,p-1\}^n$, define the functions $g'_{\underline{t}}:\mc{D}_1(\mb{Z}^n)\to H^{(p)}_{|\underline{t}|}$ as follows: $g'_{\underline{t}}(\underline{x}):=m_{t_1}^{(p)}(x_1)m_{t_2}^{(p)}(x_2)\cdots m_{t_n}^{(p)}(x_n)$, where $\underline{x}=(x_1,\ldots,x_n)\in \mb{Z}^n$ and by definition we take $H^{(p)}_0:=H^{(p)}_1$. The proof that this is a morphism follows from Lemma \ref{lem:prod-m}. Before continuing, let us note two useful properties of these morphisms:
\begin{itemize}
    \item $g'_{\underline{t}}(\underline{t})=1$ and
    \item $g'_{\underline{t}}(\underline{x})=0$ if $x_j<t_j$ for some $j\in [n]$.
\end{itemize}
Now we are ready to complete the argument. We argue by induction on $\underline{t}\in \{0,\ldots,p-1\}^{n}$, using the colexicographic order on this set. Suppose that for a fixed $\underline{t'}=(t'_1,\ldots,t'_n)$ we have been able to find a morphism $h\in \hom(\mc{D}_1(\mb{Z}_p^n),\ns)$ such that $(fh)(\underline{x}) = \id$ for all $\underline{x}  \le \underline{t'}$. Let $\underline{t}=(t_1,\ldots,t_n)$ be the next vector after $\underline{t'}$ in the colex order. Note that $w:=(fh)(\underline{t})\in G_{|\underline{t}|}$. This can be seen by composing the morphism $fh$ with the maximal cube $\q_{0^n,\underline{t}}$ (see Definition \ref{def:max-cube}).

We now define $g\in \hom(\mc{D}_1(\mb{Z}_p^n),\ns)$ by setting  $g(\underline{x}):= w^{g'_{\underline{t}}(\underline{x})}$, where note that this is indeed a morphism from $\mc{D}_1(\mb{Z}_p^n)$ since it is a morphism from $\mc{D}_1(\mb{Z}^n)$ and is $p$-periodic in each coordinate (since the $m_i^{(p)}$ are $p$-periodic).
To conclude this inductive step, we define $h'\in \hom(\mc{D}_1(\mb{Z}_p^n),\ns)$ as $h':=hg^{-1}$. By the two mentioned properties of $g'_{\underline{t}}$ we have that $fh'(\underline{x})=\id$ now for all $\underline{x}\leq \underline{t}$. 
\end{proof}
\noindent As a first consequence of Theorem \ref{thm:strong-p-hom-group} we obtain the following simple description of $p$-homogeneous nilspaces defined on finite cyclic groups. Let us say that a set of integers is \emph{$t$-separated} if  every pair of integers $a,b$ in this set satisfies $|a-b|\geq t$.
\begin{proposition}\label{prop:cyclic-p-hom}
Let $G$ be a finite cyclic group equipped with a filtration $G_\bullet$ of degree exactly $k$ \textup{(}i.e.\ $G_{k+1}=\{0\}\neq G_k$\textup{)}, and such that the associated group nilspace is $p$-homogeneous. Then $G\cong\mb{Z}_{p^d}$ for some positive integer $d \leq \lfloor \frac{k-1}{p-1}\rfloor +1$, and there is a $(p-1)$-separated set $\Delta\subset [k]$ with $|\Delta|=d$ and $k\in \Delta$, such that $G_{i+1}=p\cdot G_i$ for every $i\in\Delta$ and $G_{i+1}=G_i$ otherwise.
\end{proposition}
\begin{proof}
By Proposition \ref{prop:strgps}, every structure group of $\ns$, i.e. every quotient $G_i/G_{i+1}$, is an elementary abelian $p$-group, which must then be cyclic, so must be $\{0\}$ or $\mb{Z}_p$. It follows that $G=\mb{Z}_{p^d}$ for some $d\geq 0$. Let $\Delta=\{i\in [k]:G_i/G_{i+1}\cong \mb{Z}_p\}$. It is then clear that $G_{i+1}=p\cdot G_i$ for $i\in \Delta$ and $G_{i+1}=G_i$ otherwise, and also that $k\in \Delta$ (since $G_\bullet$ has degree exactly $k$). Since $G=\mb{Z}_{p^d}$, it is also clear that $|\Delta|=d$. 

To see that $\Delta$ is $(p-1)$-separated, let $i<j$ be any two elements of $\Delta$, so that $G_{j+1}\subset p^2\cdot G_i$, and suppose for a contradiction that $j-i< p-1$. Then we would have $j+1\leq i+ p-1$, so $G_{i+p-1}\leq G_{j+1} = p^2\cdot G_i$, and since $p\cdot G_i$ is not the trivial subgroup, we would also have $p^2\cdot G_i\subsetneq p\cdot G_i$, so $G_{i+p-1}\subsetneq p\cdot G_i$, and so $G_\bullet$ would not be $p$-homogeneous, contradicting Theorem \ref{thm:strong-p-hom-group}.

Finally, by the previous paragraphs $1+ (d-1)(p-1)\leq (\min \Delta) + (|\Delta|-1)(p-1)$ $\leq \max \Delta = k$. This implies $d \leq \lfloor \frac{k-1}{p-1}\rfloor +1$.
\end{proof}
\begin{remark}
Proposition \ref{prop:cyclic-p-hom} implies that the nilspace $\abph_{k,1}$ from Definition \ref{def:bblocks-intro}, defined on the cyclic group $\mb{Z}/p^{\lfloor \frac{k-1}{p-1}\rfloor +1}\mb{Z}$, is the largest nilspace among $k$-step $p$-homogeneous nilspaces defined on finite cyclic groups.
\end{remark}
\noindent Next we use Theorem \ref{thm:strong-p-hom-group} to prove that the translation group of a $p$-homogeneous nilspace is also $p$-homogeneous, a fact that we shall use in Section \ref{sec:gen-set}.

\begin{proposition}\label{prop:trans-of-p-strong}
Let $\ns$ be a $p$-homogeneous nilspace. Then the group nilspace consis- ting of the translation group $\tran(\ns)$ with the filtration $\big(\!\tran_i(\ns)\big)_{i\geq 0}$ is also $p$-homogeneous.
\end{proposition}

\begin{proof}
By Theorem \ref{thm:strong-p-hom-group} it suffices to prove that for every $\alpha\in \tran_i(\ns)$ we have $\alpha^p \in \tran_{i+p-1}(\ns)$. To do this, given any cube $\q\in \cu^{p+n}(\ns)$ where $n\ge i-1$ (otherwise we may not have enough dimensions) we want to show that applying $\alpha^p$ to any face of codimension $i+p-1$ gives again a cube. By the symmetries of cubes it suffices to show this for just one  particular face. Let us write $\db{p+n}=\db{i-1}\times \db{p} \times \db{n-i+1}$, so any element $y\in\db{p+n}$ is of the form $(u,w,v) \in \db{i-1}\times \db{p} \times \db{n-i+1}$. Let us define the faces $F_j:=\{(u,w,v)\in \db{p+n}: u=1^{i-1}, w(j)=1\}$ and let $C:= \cap_{j=1}^p F_j$. It is clear that $\codim(C)=i+p-1$, so it suffices to prove that $(\alpha^p)^C(\q)\in\cu^{n+p}(\ns)$. First, let us take $\q':= \alpha^{F_1}\co \cdots\co \alpha^{F_p} (\q)$, i.e., the cube obtained applying $\alpha$ to $F_j$ for all $j\in [p]$. Note that in $\q'$ we have $\alpha^p$ applied to $\q(y)$ for each $y\in C$, as we need. However there are also ``errors" in $\q'$, i.e., applications of non-zero powers of $\alpha$ to $\q(y)$ for some elements $y\not\in C$. Our aim now is to use the additional symmetries provided by $p$-homogeneity to cancel these errors. To this end, we  define a matrix $T\in M^{(p+n)\times (p+n)}(\mb{Z}_p)$ as follows: the submatrix of $T$ formed by the first $i-1$ rows and columns is the identity matrix, and similarly for the submatrix formed by the last $n-i+1$ rows and columns of $T$; the submatrix of $T$ formed by the $p$ rows and columns indexed by $[i,i+p-1]$ is an identity matrix too except for its first row, where we set all entries equal to 1. All other entries of $T$ are 0.

Let $f\in \hom(\mc{D}_1(\mb{Z}_p^{p+n}),\ns)$ be a periodic extension of the cube $\q'$. Then also $f \co T^{-1} \in \hom(\mc{D}_1(\mb{Z}_p^{p+n}),\ns)$. Recall that we can interpret $f \co T^{-1}$ as a $p$-periodic morphism in $\hom(\mc{D}_1(\mb{Z}^{p+n}),\ns)$. Now let us define the morphism $g\in \hom(\mc{D}_1(\mb{Z}^{p+n}),\mc{D}_i(\mb{Z}))$ by setting $g(u,w,v):=-u(1)u(2)\cdots u(i-1)w(1)$ (it is easy to see that this is a morphism). 

Now, given $h\in \hom(\mc{D}_1(\mb{Z}^{p+n}),\ns)$, $g\in \hom(\mc{D}_1(\mb{Z}^{p+n}),\mc{D}_i(\mb{Z}))$ and $\alpha\in\tran_i(\ns)$, for every $x\in \mb{Z}^{p+n}$ let $(\alpha^g * h)(x) := \alpha^{g(x)}(h(x))$. Note that $\alpha^g * h \in \hom(\mc{D}_1(\mb{Z}^{p+n}),\ns)$, since when we compose this with a cube, we apply $\alpha$ to faces of codimension $i$ of the cube. 

By the above observation, we have $\alpha^g * (f\co T^{-1}) \in \hom(\mc{D}_1(\mb{Z}^{p+n}),\ns)$. As $\ns$ is $p$-homogeneous, the restriction $\alpha^g * (f\co T^{-1})|_{[0,p-1]^{p+n}}$ is in $\hom(\mc{D}_1(\mb{Z}_p^{p+n}),\ns)$. Let us denote this restriction by $m$. Since $m\in \hom(\mc{D}_1(\mb{Z}_p^{p+n}),\ns)$, we know that $m \co T \in \hom(\mc{D}_1(\mb{Z}_p^{p+n}),\ns)$. We now complete the proof by showing that  $m \co T|_{\db{p+n}} = (\alpha^p)^C (\q)$.

Let $y$ be any element of $\db{p+n}$, so $m\co T|_{\db{p+n}}(y)=m(T(y))=\alpha^{g(x)}(f\co T^{-1}(x))$, where $x$ is $T(y)$ with coordinates reduced mod $p$ into $[0,p-1]$. To see that this is equal to $(\alpha^p)^C(\q)(y)$, first note that if $y\in C$ then $T(y) \!\mod p$ has $w(1)=0$, so $g(x)=0$ and therefore $\alpha^{g(x)}(f\co T^{-1}(x))=f(y)=\q'(y)=\alpha^p\q(y)$. Now, if $y\not\in C$, then let $s\in [0,p-1]$ be the number of coordinates that are 1 in the $w$ part of $y$. Consider the case $y\not\in \cup_i F_i$, i.e., we have $s=0$ or some coordinate $u(j)$ is 0. Then again the element $x=T(y) \!\mod p$ satisfies $g(x)=0$ (either because $w(1)=0$ or because $u(j)=0$) and as above we then have $\alpha^{g(x)}(f\co T^{-1}(x))=f(y)$, which is $\q(y)$ (since $y$ is not in any face $F_i$ in this case), as required. The remaining case is $y\in (\cup_i F_i)\setminus C$, i.e., that all $u(j)$ are 1 and $s\in [p-1]$. Then note that the element $x=T(y) \!\mod p$ satisfies $g(x)=-s$. Hence $\alpha^{g(x)}(f\co T^{-1}(x))=\alpha^{-s}\q'(y)$, and this equals $\q(y)$ since $y$ is in the intersection of $s$ faces and so $\q'(y)=\alpha^s(\q)(y)$. This completes the proof that $m \co T|_{\db{p+n}} = (\alpha^p)^C(\q)$.
\end{proof}
\noindent We close this section with the observation that, using the results above, Proposition \ref{prop:dimreduc} can be upgraded by showing that its converse also holds. This yields the following further equivalent description of $p$-homogeneous nilspaces.

\begin{proposition}\label{prop:lifting-equiv}
For every prime $p$ and $k\in \mb{N}$, there exists $M>0$ such that the following holds. A $k$-step nilspace $\ns$ is $p$-homogeneous if and only if 
every structure group of $\ns$ is an elementary abelian $p$-group and for all $i\in [k]$, for every $f\in\hom(\mc{D}_1(\mb{Z}_p^M),\ns_i)$ there is $\tilde{f}\in\hom(\mc{D}_1(\mb{Z}_p^M),\ns)$ such that $\pi_i\co \tilde{f}=f$.
\end{proposition}
\begin{proof}
The backward implication is Proposition \ref{prop:dimreduc}. For the forward implication, the claim concerning the structure groups is given by Proposition \ref{prop:strgps} (using that $\ns$ is w-$p$-homogeneous). To see the lifting property for morphisms, we argue using Corollary \ref{cor:liftthrufibgen} as in the proof of Theorem 
\ref{thm:main1-intro} in Section \ref{sec:bridge}.
\end{proof}

\section{A structure theorem for $p$-homogeneous nilspaces}\label{sec:gen-set}
\noindent In this section we prove Theorem \ref{thm:general-p-hom-intro}, describing $p$-homogeneous finite nilspaces as fibration-images of nilspaces from the simple class in Definition \ref{def:bblocks-intro}. In  the introduction we motivated this theorem mainly through its applications. The theorem is also motivated by the following fact.
\begin{lemma}\label{lem:fibs-preserve-phoms}
Let $\ns,\nss$ be nilspaces, suppose $\ns$ is $p$-homogeneous, and let $\varphi:\ns\to \nss$ be a fibration. Then $\nss$ is also $p$-homogeneous.
\end{lemma}
\begin{proof}
Let $f\in \hom(\mc{D}_1(\mb{Z}^n),\nss)$. By Corollary \ref{cor:liftthrufib} it follows that there exists $g\in \hom(\mc{D}_1(\mb{Z}^n),\ns)$ such that $\varphi\co g = f$. Since $\ns$ is $p$-homogeneous, we have $g|_{[0,p-1]^n}\in \hom(\mc{D}_1(\mb{Z}_p^n),\ns)$, whence $f|_{[0,p-1]^n} = \varphi \co g|_{[0,p-1]^n}\in \hom(\mc{D}_1(\mb{Z}_p^n),\nss)$, as required. \footnote{An argument similar to the proof of Lemma \ref{lem:fibs-preserve-phoms} shows that fibration-images of w-$p$-homogeneous nilspaces are also w-$p$-homogeneous.}
\end{proof}
\noindent Indeed, Lemma \ref{lem:fibs-preserve-phoms} suggests that  finite $p$-homogeneous nilspaces may all emanate through fibrations from a much simpler class of nilspaces, and Theorem \ref{thm:general-p-hom-intro} confirms this. 

Recall from Definition \ref{def:bblocks-intro} that $\abph_{k,\ell}$ is the $k$-step $p$-homogeneous nilspace consisting of the group $G=\mb{Z}_{p^r}$ with $r=\lfloor\frac{k-\ell}{p-1}\rfloor+1$, equipped with the filtration
\[\begin{array}{cccccccccccccccc}
G_1 &   & G_\ell & & G_{\ell+1} & & G_{\ell+p-1} & & G_{\ell+p}\\
\parallel &  & \parallel & & \parallel &  & \parallel & &  \parallel\\
\mb{Z}_{p^r}& = \cdots = & \mb{Z}_{p^r} & \ge & p\mb{Z}_{p^r} & =  \cdots  = & p\mb{Z}_{p^r} &\geq & p^2\mb{Z}_{p^r}\cdots
\end{array}.
\] 
Note that $\abph_{k,\ell}$ is a special case of the nilspaces described in Proposition \ref{prop:cyclic-p-hom}, with the filtration chosen to ensure that this special case is an \emph{$\ell$-fold ergodic} nilspace, meaning that its $\ell$-cube set is the whole set of maps $\db{\ell}\to \abph_{k,\ell}$ (see \cite[Definition 1.2.3]{Cand:Notes1}).

\begin{remark}\label{rem:quo-z-k-l} We leave as an exercise for the reader to check that the nilspace factor map $\pi_{k-1}:\abph_{k,\ell}\to \left( \abph_{k,\ell} \right)_{k-1}$ is the quotient by the $k$-th structure group of $\abph_{k,\ell}$ (which is isomorphic to either $\mb{Z}_p$ or $\{0\}$), and that $\left( \abph_{k,\ell} \right)_{k-1}$ is isomorphic to $\abph_{k-1,\ell}$ if $\ell<k$ and is the trivial group $\{0\}$ if $\ell=k$.
\end{remark}

\noindent To begin proving Theorem \ref{thm:general-p-hom-intro}, let us note that the second sentence in the theorem, concerning lifting morphisms through the fibration $\psi$, follows from Corollary \ref{cor:liftthrufib}. Thus our main task is to prove the existence of this fibration $\psi:\nss\to\ns$, for some $\nss\in \mc{Q}_{p,k}$. The main ingredient for this is the following result which tells us that, in the class of $p$-homogeneous nilspaces, the family $\mc{Q}_{p,k}$ is closed under taking degree-$k$ extensions by finite elementary abelian $p$-groups.

\begin{proposition}\label{prop:ext-split} 
Let $Q$ be a nilspace in $\mc{Q}_{p,k}$. Let $\nss$ be a $k$-step $p$-homogeneous nilspace that is a degree-$k$ extension of $Q$ by a finite elementary abelian $p$-group. Then $\nss$ is a split extension of $Q$. In particular we have $\nss \in \mc{Q}_{p,k}$.
\end{proposition}
\noindent We shall prove this proposition by induction on the step $k$. In the induction, we shall apply the following lemma with $k'=k-1$.

\begin{lemma}\label{lem:proof-spli-aux}
Let $\ns'\in \mc{Q}_{p,k'}$, and let $\ns$ be a $k'$-step $p$-homogeneous nilspace that is a degree-$t$ extension of $\ns'$ by a finite elementary abelian $p$-group, for some $t\le k'$. Assume that Proposition \ref{prop:ext-split} holds for all steps at most $k'$. Then $\ns$ is a split extension of $\ns'$.
\end{lemma}

\begin{proof} Note that by assumption the case $t=k'$ holds. For smaller values of $t$, let $P':\ns\to\ns'$ be the projection associated with the degree-$t$ extension, let us denote both factor maps $\ns\to\ns_t$ and $\ns'\to\ns_t'$ by $\pi_t$, and let $P_t$ be the induced projection $\ns_t\to\ns_t'$, so that the following diagram commutes:
\begin{center}
\begin{tikzpicture}
  \matrix (m) [matrix of math nodes,row sep=2em,column sep=4em,minimum width=2em]
  {
     \ns & \ns'  \\
     \ns_t & \ns'_t. \\};
  \path[-stealth]
    (m-1-1) edge node [above] {$P'$} (m-1-2)
    (m-1-1) edge node [right] {$\pi_t$} (m-2-1)
    (m-1-2) edge node [right] {$\pi_t$} (m-2-2)
    (m-2-1) edge node [above] {$P'_t$} (m-2-2);
\end{tikzpicture}
\end{center}
\noindent We have $\ns'=\prod_{\ell=1}^{k'} \abph_{k',\ell}^{\,a_\ell}$ for some $a_\ell\in \mb{Z}_{\geq 0}$, whence $\ns'_t=\prod_{\ell=1}^t \abph_{t,\ell}^{\,a_\ell}$ (see Remark \ref{rem:quo-z-k-l}). Proposition \ref{prop:proj-of-ext} implies that $\ns_t$ is a degree-$t$ extension of $\ns'_t$. By the assumed Proposition \ref{prop:ext-split} for step $t$,  this extension splits, so there exists a cross-section $\gamma:\ns_t'\to \ns_t$ which is also a morphism. Let $\varphi:\ns \to \ns'\times_{\ns'_t} \ns_t$ be the map $x\mapsto (P'(x),\pi_t(x))$. By Proposition \ref{prop:proj-of-ext}, this map is an isomorphism. We can therefore define the map $\Phi:\ns'\to \ns$, $x'\mapsto \varphi^{-1}(x',\gamma(\pi_t(x')))$. Now we just have to check that this is a split extension. First, let us check that it is well-defined, i.e., we need to check that $(x',\gamma(\pi_t(x')))\in \ns'\times_{\ns'_t} \ns_t$. But as $P'_t \co \gamma = \id$, we have that $P'_t(\gamma(\pi_t(x'))) = \pi_t(x')$. Next, we need to check that $P' \co \Phi = \id$, but this follows from the definition of $\varphi$ and the fact that it is an isomorphism. Finally, as $\varphi^{-1},\gamma,$ and $\pi_t$ are morphisms, so is $\Phi$.\end{proof}

\begin{proof}[Proof of Proposition \ref{prop:ext-split}]
We argue by induction on $k$. The case $k=1$ follows from Proposition \ref{prop:strgps} since $\nss$ must then be $\mc{D}_1(\mb{Z}_p^m)$ for some $m\geq 0$, and $Q$ is also of this form, so $\nss$ is clearly a split extension of $Q$. 

Thus we suppose that $k>1$ and that the result holds for all steps at most $k-1$. Let $P:\nss\to Q$ be the projection associated with the extension. We have by assumption $Q= \prod_{\ell=1}^k \abph_{k,\ell}^{\,a_\ell}$, so we can define, for each $\ell\in [k]$ and $j\in [a_\ell]$, the translation $\alpha_{\ell,j}\in \tran_\ell(\nss')$ as the function that adds 1 in the $j$-th coordinate of the factor $\abph_{k,\ell}^{\,a_\ell}$. We are going to show that we can lift any such  translation, i.e.\ that there exists $\beta_{\ell,j}\in\tran_\ell(\nss)$ such that $P \co \beta_{\ell,j} = \alpha_{\ell,j} \co P$. 

Suppose that we can lift these translations as claimed. Since $\nss$ is $p$-homogeneous, Proposition \ref{prop:trans-of-p-strong} implies that the order of $\beta_{\ell,j}\in\tran_\ell(\nss)$ divides $p^{\left\lfloor\frac{k-\ell}{p-1}\right\rfloor +1} = p^{\left\lceil\frac{k-\ell+1}{p-1}\right\rceil}$. Hence, the following cross-section is well-defined, and is then  also a morphism:
\[
\begin{array}{cccc}
\cs: & Q & \longrightarrow & \nss \\
 & \big(\overline{x}_{\ell,j}\in \mb{Z}_{p^{\lfloor\frac{k-\ell}{p-1}\rfloor+1}}\big)_{\ell\in [k],\, j\in [a_\ell]} & \longrightarrow & \prod_{\ell=1}^k\prod_{j=1}^{a_l} \beta_{\ell,j}^{\overline{x}_{\ell,j}}(y)
\end{array}
\]
\noindent where $y\in \nss$ is any element of $P^{-1}(\underline{0})$. This would conclude the proof.

So the only missing ingredient is the possibility to  lift translations. To establish this, we start with the following diagram:
\begin{center}
\begin{tikzpicture}
  \matrix (m) [matrix of math nodes,row sep=2em,column sep=4em,minimum width=2em]
  {
     \nss & Q  \\
     \nss_{k-1} & Q_{k-1}. \\};
  \path[-stealth]
    (m-1-1) edge node [above] {$P$} (m-1-2)
    (m-1-1) edge node [right] {$\pi_{k-1}$} (m-2-1)
    (m-1-2) edge node [right] {$\pi_{k-1}$} (m-2-2)
    (m-2-1) edge node [above] {$P_{k-1}$} (m-2-2);
\end{tikzpicture}
\end{center}
\noindent First note that $P_{k-1}$ is an isomorphism. Indeed, since  $\nss$ is a degree-$k$ extension of $Q$ with associated projection $P$, it is readily checked that the $(k-1)$-step factors of $\nss$ and $Q$ are isomorphic nilspaces, with $P_{k-1}$ being an isomorphism. Now let us fix some notation. Let $A$ denote the group $\mb{Z}_p^m$ we use to extend $Q$ to get $\nss$, i.e.\
\begin{center}
\begin{tikzpicture}
\node (nss) {$P:\nss$};
\node (nss')[right=of nss] {$Q$.};
\draw[->] (nss.east) -- (nss'.west);
\drawloop[->,stretch=1.1]{nss}{50pt}{120pt} node[pos=0.4,right]{$\mc{D}_k(A)$};
\end{tikzpicture}
\end{center}
\noindent Let $B$ denote the $k$-th structure group of $Q$ (which is also a power of $\mb{Z}_p$), i.e.
\begin{center}
\begin{tikzpicture}
\node (nss) {$\pi_{k-1}:Q$};
\node (nss')[right=of nss] {$Q_{k-1}.$};
\draw[->] (nss.east) -- (nss'.west);
\drawloop[->,stretch=1.1]{nss}{50pt}{120pt} node[pos=0.4,right]{$\mc{D}_k(B)$};
\end{tikzpicture}
\end{center}
The $k$-th structure group of $\nss$ must be isomorphic to $A\times B$ because it is an elementary abelian $p$-group (by Proposition \ref{prop:p-hom-str-gps-intro}) and the fibers of $\pi_{k-1}\co P$ have cardinality $|A||B|$.

Now fix any $\alpha\in \tran_i(Q)$ and let us prove that  $\alpha$ can be lifted. First let $\alpha_{k-1}$ be the induced translation on $\tran_i(Q_{k-1})$, satisfying $\pi_{k-1}\co \alpha=\alpha_{k-1}\co\pi_{k-1}$, where $\pi_{k-1}:Q\to Q_{k-1}$. As $P_{k-1}$ is an isomorphism, we have that $\alpha_{k-1}$ is also an element of $\tran_i(\nss_{k-1})$. We shall use the criterion for lifting translations established in \cite{CamSzeg} (see also \cite[Proposition 3.3.39]{Cand:Notes1}); this criterion states that we can lift a translation if the associated nilspace $\mc{T}^*$ (see \cite[Definition 3.3.34 and (3.18)]{Cand:Notes1}) is a split extension (it is not difficult to see that $\mc{T}^*$ is $p$-homogeneous if $\nss$ is). By \cite[Lemma 3.3.38]{Cand:Notes1} we know that $\mc{T}^*$ is a degree-$(k-i)$ extension of $\nss_{k-1}\cong \prod_{\ell=1}^{k-1} (\abph^{(p)}_{k-1,\ell})^{a_\ell}$, and by Lemma \ref{lem:proof-spli-aux} (applied with $k'=k-1$) we know that this extension splits. This enables us\footnote{This works for $i<k$, and for $i=k$ it is trivial that we can lift translations, so without loss of generality we can safely assume $i<k$.} to lift $\alpha_{k-1}$ to a translation $\beta\in \tran_i(\nss)$. Now the problem is that $\beta$ may not agree with $\alpha$. As $P$ is a degree-$k$ extension of a $k$-step nilspace, it is easy to check that the conditions of \cite[Lemma 1.5]{CGSS} are satisfied, so there is an element $\beta^*\in\tran_i(Q)$ such that $\beta^* \co P = P \co \beta$. 

Now the issue is that $\alpha$ and $\beta^*$ may not be equal. But we know that they are in the same $B$-fiber (because the shadow of both through $\pi_{k-1}:Q\to Q_{k-1}$ is $\alpha_{k-1}$). Now observe that, as $P:\nss\to Q$ is a fibration with $A\times B$ being the $k$-th structure group of $\nss$ and $B$ being the $k$-th structure group of $Q$, if we let $\phi:A\times B \to B$ be the homomorphism such that $P(x+z)=P(x)+\phi(z)$ for all $x\in \nss$ and $z\in A\times B$, we know that there is a cross section that is a morphism. That is, there is a map $\cs:B \to A\times B$ such that $\phi \co \cs = \id$. The simple reason for this is that $A\times B= \mb{Z}_p^{n+m}$ and $B=\mb{Z}_p^m$ (and given a homomorphism from one to the other it is trivial that we can construct a homomorphism which is a cross-section). To conclude the proof, we define 
\[
\begin{array}{cccc}
\gamma: & \nss & \longrightarrow & \nss \\
 & y & \longrightarrow & \beta(y)+\cs((\alpha-\beta^*)(P(y)),
\end{array}
\]
and this is an element of $\tran_i(\nss)$ that agrees with $\alpha$ through $P$.
\end{proof}
We can now prove the main structure theorem.
\begin{proof}[Proof of Theorem \ref{thm:general-p-hom-intro}]
We prove the existence of the fibration $\psi:\nss\to\ns$ by induction on the step $k$. The case $k=1$ follows immediately from Proposition \ref{prop:strgps}.

Thus we suppose that the theorem holds for all steps less than $k$ and we prove the theorem for step $k$. By induction there exists a nilspace $\nss'=\prod_{\ell=1}^{k-1} (\abph_{k-1,\ell}^{(p)})^{a_\ell}$ for some $a_\ell \ge 0$ for all $\ell\in [k-1]$ and a fibration $\psi':\nss'\to \ns_{k-1}$ satisfying the conclusions of Theorem \ref{thm:general-p-hom-intro} for $\ns_{k-1}$. It is easy to see that for all $\ell$ there is the projection map $\pi_{k-1}:\abph_{k,\ell}\to \abph_{k-1,\ell}$ (these maps are different for each $\ell$, but it will be clear from the context that $\pi_{k-1}$ will always represent the projection to the $k-1$ factor of a certain nilspace). Thus we can define $Q:=\prod_{\ell\in [k-1]} \abph_{k,\ell}^{\,a_\ell}$ and we have the following diagram:
\begin{center}
\begin{tikzpicture}
  \matrix (m) [matrix of math nodes,row sep=3em,column sep=4em,minimum width=2em]
  {
     \ns & \ns_{k-1} \\
     Q & \nss'. \\};
  \path[-stealth]
    (m-1-1) edge node [above] {$\pi$} (m-1-2)
    (m-2-2) edge node [right] {$\psi'$} (m-1-2)
    (m-2-1) edge node [above] {$q$} (m-2-2);
\end{tikzpicture}
\end{center}
\noindent where $q=\pi_{k-1}$ is just the projection to the $k-1$ factor of $Q$ (thus $\nss'=Q_{k-1}$).

We now define $\nss$ as the following nilspace-subdirect-product of $Q$ and $\ns$:
\[
\nss := \{(a,b) \in Q\times \ns : \psi'(q(a)) = \pi (b)\}.
\]
\noindent Note that $\nss$ is a degree-$k$ extension of $Q$ by the $k$-th structure group of $\ns$ (by Proposition \ref{prop:sub-prod-ext}). Hence we have the following diagram:
\begin{center}
\begin{tikzpicture}
  \matrix (m) [matrix of math nodes,row sep=3em,column sep=4em,minimum width=2em]
  {
     \nss & \ns & \ns_{k-1} \\
      & Q & \nss',\\};
  \path[-stealth]
    (m-1-1) edge node [above] {$p_2$} (m-1-2)
    (m-1-1) edge node [above] {$p_1$} (m-2-2)
    (m-1-2) edge node [above] {$\pi$} (m-1-3)
    (m-2-3) edge node [right] {$\psi'$} (m-1-3)
    (m-2-2) edge node [above] {$q$} (m-2-3);
\end{tikzpicture}
\end{center}
where $Q = \prod_{\ell=1}^{k-1} \abph_{k,\ell}^{\,a_\ell}$ and $\nss' =\prod_{\ell=1}^{k-1} \abph_{k-1,\ell}^{\,a_\ell}$, and $p_1$ is the projection of the following degree-$k$ extension:
\begin{center}
\begin{tikzpicture}
\node (nss) {$p_1:\nss$};
\node (nss')[right=of nss] {$Q$\,.};
\draw[->] (nss.east) -- (nss'.west);
\drawloop[->,stretch=1.1]{nss}{50pt}{120pt} node[pos=0.4,right]{$\mc{D}_k(\mb{Z}_p^n)$};
\end{tikzpicture}
\end{center}
It now suffices to prove that $\nss$ is $p$-homogeneous, as then we can apply Proposition \ref{prop:ext-split} to conclude that $\nss\in \mc{Q}_{p,k}$, and thus complete the proof setting $\psi=p_2$. But if $f=(f_1,f_2)\in \hom(\mc{D}_1(\mb{Z}^n),\nss)$, then since $Q$ and $\ns$ are $p$-homogeneous, we have $f_1|_{[0,p-1]^n}\in \hom(\mc{D}_1(\mb{Z}_p^n),Q)$ and $f_2|_{[0,p-1]^n}\in \hom(\mc{D}_1(\mb{Z}_p^n),\ns)$, whence $\nss$ is also $p$-homogeneous.

As mentioned above, the last sentence in Theorem \ref{thm:general-p-hom-intro} follows by Corollary \ref{cor:liftthrufib}.
\end{proof}

\subsection{A refined structure theorem for $k\leq p$}\hfill\smallskip\\
In this subsection we prove Theorem \ref{thm:intro-3}. 

We begin by noting that, in the high characteristic case ($k<p$), Proposition \ref{prop:ext-split} readily implies the following strengthening of Theorem \ref{thm:general-p-hom-intro}.
\begin{corollary}\label{cor:high-char1}
Let $\ns$ be a $k$-step $p$-homogeneous \textsc{cfr} nilspace, and let $p$ be a prime with $k<p$. Then there are non-negative integers $a_1,a_2,\ldots,a_k$ such that $\ns\cong\prod_{\ell=1}^{k} \mc{D}_\ell(\mb{Z}_p^{a_\ell})$. In particular, $\ns\in \mc{Q}_{p,k}$.
\end{corollary}

\begin{proof}
We argue by induction on $k$. The case $k=0$ is trivial. For $k>0$, by induction we have $\ns_{k-1} = \prod_{l=1}^{k-1} \mc{D}_\ell(\mb{Z}_p^{a_\ell})$ for some integers $a_\ell\ge 0$. By general nilspace theory we know that $\ns$ is a degree-$k$ extension of $\ns_{k-1}$ by some compact abelian group $\ab_k$, and by Proposition \ref{prop:strgps} we know that $\ab_k\cong \mb{Z}_p^{a_k}$ for some $a_k\ge 0$. Now note that for all $\ell\le k$, since $k <p$ we have $\mc{D}_\ell(\mb{Z}_p)= \abph_{k,\ell}$, because $\lfloor \frac{k-l}{p-1}\rfloor+1 = 1$. Therefore, as a product of such nilspaces  $\abph_{k,\ell}$, the nilspace $\ns_{k-1}$ is in $\mc{Q}_{p,k}$. Hence we can apply Proposition \ref{prop:ext-split}, thus deducing that $\ns$ is a split extension of $\ns_{k-1}$, so $\ns = \ns_{k-1}\times \mc{D}_k(\mb{Z}_p^{a_k}) = \prod_{\ell=1}^k \mc{D}_\ell(\mb{Z}_p^{a_\ell})$.
\end{proof}
\noindent We shall now extend this result to include the case $k=p$. To do so we shall use the following fact, which is a very specific feature of this case.
\begin{proposition}\label{prop:aux-high-char}
For any prime $p$, the nilspace $\abph_{p,1}$ is isomorphic to the product nilspace $\mc{D}_1(\mb{Z}_p)\times \mc{D}_{p}(\mb{Z}_p)$.
\end{proposition}

\begin{proof}
It suffices to prove that there is a cross-section from $\mc{D}_1(\mb{Z}_p)$ to $\abph_{p,1}$ which is also a morphism, as this proves that $\abph_{p,1}$ is a split extension of $\mc{D}_1(\mb{Z}_p)$, which implies the result. Let $f\in\hom(\mc{D}_1(\mb{Z}_p),\mc{D}_1(\mb{Z}_p))$ be the identity map. Then, as $(\abph_{p,1})_{p-1}\cong \mc{D}_1(\mb{Z}_p)$, we can regard $f$ as an element of $\hom(\mc{D}_1(\mb{Z}_p),(\abph_{p,1})_{p-1})$. By Proposition \ref{prop:lifting-equiv}, there exists $g\in \hom(\mc{D}_1(\mb{Z}_p),\abph_{p,1})$ such that $\pi_{p-1}\co g =f$. Since $f$ is the identity map, the last equality implies that the morphism $g$ is also a cross-section $(\abph_{p,1})_{p-1}\to \abph_{p,1}$, as required.
\end{proof}
\noindent The following result is the announced extension of Corollary \ref{cor:high-char1}, and is the special case of Theorem \ref{thm:intro-3} for \textsc{cfr} nilspaces.

\begin{proposition}\label{prop:high-char}
 Let $p$ be a prime and $k\in \mb{N}$ with $k \leq p$. Let $\ns$ be a $k$-step $p$-homogeneous \textsc{cfr} nilspace. Then there are non-negative integers $a_1,a_2,\ldots,a_k$ such that $\ns$ is isomorphic to the product nilspace $\prod_{\ell=1}^k \mc{D}_\ell(\mb{Z}_p^{a_\ell})$. In particular, if $k<p$ then $\ns \in \mc{Q}_{k,p}$.
\end{proposition}

\begin{proof}
By Corollary \ref{cor:high-char1} it suffices to prove this for $k=p$. The idea is to start again from the situation we had in the proof of Theorem \ref{thm:general-p-hom-intro} with the following diagram:
\begin{center}
\begin{tikzpicture}
  \matrix (m) [matrix of math nodes,row sep=3em,column sep=4em,minimum width=2em]
  {
     \nss & \ns & \ns_{k-1} \\
      & Q & \nss'.\\};
  \path[-stealth]
    (m-1-1) edge node [above] {$p_2$} (m-1-2)
    (m-1-1) edge node [above] {$p_1$} (m-2-2)
    (m-1-2) edge node [above] {$\pi$} (m-1-3)
    (m-2-3) edge node [right] {$\psi'$} (m-1-3)
    (m-2-2) edge node [above] {$q$} (m-2-3);
\end{tikzpicture}
\end{center}
Note that in this case, by induction we can take $\nss'=\ns_{k-1}=\ns_{p-1}$, $\psi'$ an isomorphism and $\ns_{k-1}=\ns_{p-1}\cong \prod_{\ell=1}^{p-1} \mc{D}_\ell(\mb{Z}_p^{a_\ell})$. Now, recall that by definition of $Q$ we had to \emph{lift} these factors, in the sense that each factor $\mc{D}_\ell(\mb{Z}_p) = \abph_{\ell,\ell}$ in $Q$ is lifted to a  factor $\abph_{p,\ell}$. All these \emph{lifts} are trivial except for $\mc{D}_1(\mb{Z}_p)$, which is lifted to $\abph_{p,1}$. Hence in this case we have $Q=\abph_{p,1}^{\,a_1} \times \prod_{\ell=2}^{p-1} \mc{D}_\ell(\mb{Z}_p^{a_\ell})$. Hence, by Proposition \ref{prop:aux-high-char} there exists a cross-section $s:\nss'\to Q$ which is also a morphism. The rest of the proof goes as before: the map $p_1:\nss\to Q$ is a degree-$p$ extension that splits, so there is again a cross-section $s':Q\to \nss$ which is a morphism. Then we have the cross-section $p_2\co s'\co s\co\psi'^{-1}:\ns_{k-1}\to \ns$, which is also a morphism. Hence $\ns$ is a split extension of $\ns_{k-1}$ and the result follows.
\end{proof}

\begin{corollary}\label{cor:spli-klep} Let $k\le p$, let $\ns$ be a $k$-step $p$-homogeneous \textsc{cfr} nilspace, and let $\nss$ be a $p$-homogeneous nilspace which is a degree $k$-extension of $\ns$ by an elementary abelian $p$-group. Then $\nss$ is a split extension of $\ns$.
\end{corollary}

\begin{proof} The cases $k<p$ follow from combining Corollary \ref{cor:high-char1} with Proposition \ref{prop:ext-split}.

Now suppose that $k=p$. By Proposition \ref{prop:high-char} we have $\ns = \prod_{i=1}^p \mc{D}_i(\mb{Z}_p^{a_i})$ for some integers $a_i\ge 0$. Let $\ns^+:=\abph_{p,1}^{\,a_1}\times \prod_{i=2}^p \mc{D}_i(\mb{Z}_p^{a_i})$, and let $\phi:\ns^+\to \ns$ be the projection with deletes the component $\mc{D}_p(\mb{Z}_p^{a_1})$ from $\abph_{p,1}^{\,a_1}$, as made possible by Proposition \ref{prop:aux-high-char}. It is then clear that there exists a cross-section $s:\ns\to \ns^+$. Now let $T:=\ns^+\times_{\ns} \nss=\{(x,y)\in \ns^+\times \nss: \varphi(y)=\phi(x)\}$, where $\varphi$ here denotes the projection map $\nss\to\ns$ associated with the extension. It is easy to see that $T$ is a degree-$p$ extension of $\ns^+$ and, since $\ns^+\in \mc{Q}_{p,p}$, this extension splits. Thus, letting $p_1$ denote the associated projection $T\to \ns^+$, there exists a cross-section which is a morphism $s':\ns^+\to T$. To conclude, note that if $p_2:T\to \nss$ is the projection to the second coordinate, then $p_2\co s'\co s:\ns \to \nss$ is a cross-section which is also a morphism, and the result follows.
\end{proof}
\noindent If we just plugged Proposition \ref{prop:high-char} into the inverse limit theorem then we would obtain not quite Theorem \ref{thm:intro-3}, but rather a description of each factor in the inverse system. The following result will enable us to arrange the terms in the inverse system to express the inverse limit as the product nilspace claimed in Theorem \ref{thm:intro-3}.

\begin{proposition}\label{prop:very-final-s}
Let $p$ be a prime, let $k\leq p$, let $\ns,\nss$ be $k$-step, $p$-homogeneous nilspaces, and let $\varphi:\ns\to\nss$ be a fibration. Then $\ns \cong \nss \times Q$ for some $k$-step, $p$-homogeneous nilspace $Q$ and there exists a nilspace isomorphism $\phi:\nss\times Q \to \ns$ such that $\varphi \co \phi :\nss\times Q \to \nss$ is the projection $(y,q)\mapsto y$.
\end{proposition}

\begin{proof} 
We argue by induction on $k$. The case $k=0$ is trivial. For $k>0$, suppose that $\varphi:\ns\to \nss$ is a fibration and that, by induction, the fibration $\varphi_{k-1}:\ns_{k-1}\to \nss_{k-1}$ satisfies the following property: There exists a nilspace isomorphism $\phi:\ns_{k-1}\to \nss_{k-1}\times Q_{k-1}$ such that if $p_1:\nss_{k-1}\times Q_{k-1}\to \nss_{k-1}$ is the projection to the first coordinate, then $\varphi_{k-1} = p_1 \co \phi$. The situation is illustrated in the following diagram:
\begin{center}
\begin{tikzpicture}
  \matrix (m) [matrix of math nodes,row sep=3em,column sep=4em,minimum width=2em]
  {
      \ns & \nss \\
      \ns_{k-1} & \nss_{k-1}\\
      \nss_{k-1}\times Q_{k-1}& \\
      };
  \path[-stealth]
    (m-1-1) edge node [above] {$\varphi$} (m-1-2)
    (m-1-1) edge node [right] {$\pi$} (m-2-1)
    (m-1-2) edge node [right] {$\pi$} (m-2-2)
    (m-2-1) edge node [above] {$\varphi_{k-1}$} (m-2-2)
    (m-2-1) edge node [right] {$ \phi$} (m-3-1)
    (m-3-1) edge node [above] {$p_1$} (m-2-2)
    ;
\end{tikzpicture}
\end{center}
Now let $p_2:\nss_{k-1}\times Q_{k-1}\to Q_{k-1}$ be the projection to the second coordinate, and let
\[
\begin{array}{cccc}
\Psi: & \ns & \longrightarrow & \nss\times Q_{k-1} \\
 & x & \longrightarrow & \big(\varphi(x),p_2\co \phi\co \pi(x)\big)
\end{array}.
\]
We claim that this defines a degree-$k$ extension (in the sense of \cite[Definition 3.3.13]{Cand:Notes1}) of $\nss\times Q_{k-1}$ by the abelian group $\ker(\phi_k)$, where the homomorphism $\phi_k:\ab_k(\ns)\to\ab_k(\nss)$ is the $k$-th \emph{structure morphism} of $\varphi$ (so $\varphi(x+z)=\varphi(x)+\phi_k(z)$ for all $x\in \ns$ and $z\in \ab_k(\ns)$; see \cite[Definition 3.3.1]{Cand:Notes1}). In particular $\ker(\phi_k)$ is an elementary abelian $p$-group.

To prove this claim, first let us show that $\ns$ is an abelian bundle over $\nss\times Q_{k-1}$ with projection $\Psi$. To see that $\Psi$ is surjective, fix any $(y,q)\in \nss\times Q_{k-1}$ and consider the element $\phi^{-1}(\pi(y),q)\in \ns_{k-1}$. This equals $\pi(x)$ for some $x\in \ns$, by surjectivity of $\pi$. Now, as $\pi(\varphi(x))=\pi(\varphi(y))$, there is $z\in \ab_k(\nss)$ such that $\varphi(x)+z=\varphi(y)$. Since $\varphi$ is a fibration, we know that $\phi_k$ is surjective, so there exists $z'\in \ab_k(\ns)$ such that $\phi_k(z')=z$ and hence the element $x+z'$ satisfies $\Psi(x+z')=(y,q)$, which proves the surjectivity. Now let $x,x'\in \ns$ be such that $\Psi(x)=\Psi(x')$. This implies that $\varphi(x)=\varphi(x')$ which in turn means that $\varphi_{k-1}(\pi(x)) = \varphi_{k-1}(\pi(x'))$. Since we also have  $p_2\co\phi\co \pi(x)=p_2\co \phi \co \pi(x')$ we conclude that $\pi(x)=\pi(x')$. Thus, there exists $z\in \ab_k(\ns)$ such that $x+z=x'$. Applying $\varphi$ to both sides of this expression we obtain $\varphi(x+z)=\varphi(x)+\phi_k(z) = \varphi(x')=\varphi(x)$. This implies that $z\in \ker(\phi_k)$. The fact that $\ker(\phi_k)$ acts freely on the fibers of $\ns$ follows from the fact that $\ab_k(\ns)$ acts freely on $\ns$. This proves our claim.

Now let us see that $\Psi$ defines indeed a degree-$k$ extension as claimed. The first condition to verify is that $\Psi$ is cube-surjective. Let $\q_1\times \q_2 \in \cu^n(\nss\times Q_{k-1})$. In particular this means that $(\pi\co \q_1)\times \q_2 \in \cu^n(\nss_{k-1}\times Q_{k-1})$ and thus, $\phi^{-1} \co ((\pi\co \q_1)\times \q_2) \in \cu^n(\ns_{k-1})$. Let $\q\in \cu^n(\ns)$ be a lift of this cube. In particular, this means that $\pi \co \varphi \co \q = \pi \co \q_1$. Thus, there exists $d\in \cu^n(\mc{D}_k(\ab_k(\nss)))$ such that $\varphi \co \q +d = \q_1$. As $\phi_k$ is surjective it is easy to see that there exists $d'\in \cu^n(\mc{D}_k(\ab_k(\ns)))$ such that $\phi_k \co d' = d$. Thus we have that $\q+d'\in \cu^n(\ns)$ is a lift of $\q_1\times \q_2$ via $\Psi$. The second condition is that for any $\q_1\in \cu^n(\ns)$, if $\q_2\in \cu^n(\ns)$ is any cube such that $\Psi \co \q_1 = \Psi\co \q_2$ then there exists $d\in \cu^n(\mc{D}_k(\ker(\phi_k)))$ such that $\q_1+d = \q_2$. By similar arguments as before it follows that there exists $d\in \cu^n(\mc{D}_k(\ab_k(\ns)))$ such that $\q_1+d = \q_2$. This implies that $\varphi\co \q_1+\phi_k \co d = \varphi \co \q_2$. But by hypothesis we know that $\varphi\co \q_1 = \varphi \co \q_2$, which implies that $d\in \cu^n(\mc{D}_k(\ker(\phi_k)))$. This proves that $\Psi$ is an extension as claimed.

To finish the proof, note that by Corollary \ref{cor:spli-klep} the extension defined by $\Psi$ splits. Thus we have for any $k\le p$ that $\ns \cong \nss \times Q_{k-1}\times \mc{D}_k(\ker(\phi_k))$. Letting $Q:=Q_{k-1}\times \mc{D}_k(\ker(\phi_k))$, this means that there is a nilspace isomorphism $\phi':\ns \to \nss \times Q$ such that $\varphi = p_1\co \phi'$, as required.
\end{proof}
\begin{proof}[Proof of Theorem \ref{thm:intro-3}]
By the inverse limit theorem \cite[Theorem 2.7.3]{Cand:Notes2}, the given $k$-step $p$-homogeneous compact nilspace $\ns$ is the inverse limit of \textsc{cfr} $k$-step nilspaces $\ns_j$, which are $p$-homogeneous by Lemma \ref{lem:fibs-preserve-phoms}. By Propositon \ref{prop:high-char}, each nilspace $\ns_j$ is of the form $\prod_{\ell=1}^k \mc{D}_\ell(\mb{Z}_p^{a_{j,\ell}})$ for some coefficients $a_{j,\ell}$. It now only remains to arrange these factors as $j$ ranges in $\mb{N}$ to obtain the claimed product nilspace in Theorem \ref{thm:intro-3}. To carry out this arrangement we use Proposition \ref{prop:very-final-s}: It allows us to see the maps $\psi_{i,j}:\ns_j\to\ns_i$ as projections. Hence the inverse limit has the desired form. \end{proof}
\noindent This proof of Theorem \ref{thm:intro-3} used several fortunate facts occurring for $k\leq p$, including Proposition \ref{prop:aux-high-char}. The question of whether there are similar refinements of Theorem \ref{thm:general-p-hom-intro} for higher $k>p$ seems non-trivial (see Remark \ref{rem:gendiff}). For the case $k=p+1$ we can nevertheless prove the following result, which does refine Theorem \ref{thm:intro-3} and which will be used in the next section to give new applications in ergodic theory.
\begin{proposition}\label{prop:struc-k=p+1}
Let $\ns$ be a \textsc{cfr} $k$-step $p$-homogeneous nilspace with $k=p+1$. Then there is an integer $m\ge 0$ such that $\ns\times \mc{D}_{p}(\mb{Z}_p^m)$ is isomorphic to an abelian group nilspace. In particular, there exists an injective morphism from $\ns$ to an abelian group nilspace.
\end{proposition}

The proof will use the following fact.

\begin{lemma}\label{lem:Claim1}
Let $\varphi:\ns\to \nss$ be a fibration between $k$-step $p$-homogeneous \textsc{cfr} nilspaces. Suppose that all the structure morphisms are isomorphisms except maybe $\phi_{k-1}$. Then $\ns$ is a degree-$(k-1)$ extension of $\nss$ by the group $\ker(\phi_{k-1})$.
\end{lemma}

\begin{proof}
Consider the following diagram:
\begin{center}
\begin{tikzpicture}
  \matrix (m) [matrix of math nodes,row sep=3em,column sep=4em,minimum width=2em]
  {
     \ns & \nss  \\
     \ns_{k-1} & \nss_{k-1}. \\};
  \path[-stealth]
    (m-1-1) edge node [above] {$\varphi$} (m-1-2)
    (m-1-1) edge node [right] {$\pi_{k-1}$} (m-2-1)
    (m-1-2) edge node [right] {$\pi_{k-1}$} (m-2-2)
    (m-2-1) edge node [above] {$\varphi_{k-1}$} (m-2-2);
\end{tikzpicture}
\end{center}
Now let us define the fiber product $\ns_{k-1}\times_{\nss_{k-1}} \nss$ and the map $\Phi:\ns\to \ns_{k-1}\times_{\nss_{k-1}} \nss$ such that $x \mapsto (\pi_{k-1}(x),\varphi(x))$. We claim that this is a nilspace isomorphism. First, it is clear that this is well-defined and that it is a morphism. Second, to prove that $\Phi$ is injective, let $x,x'\in \ns$ and suppose that $\Phi(x)=\Phi(x')$. In particular, $\pi_{k-1}(x)=\pi_{k-1}(x')$ and thus $x = x'+z$ for some $z\in \ab_k(\ns)$. But this means that $\varphi(x)=\varphi(x')+\phi_k(z)$ and as $\varphi(x)=\varphi(x')$, this implies that $\phi_k(z)=0$. Using that $\phi_k$ is bijective we get that $z=0$. To prove that $\Phi$ is surjective, let $(\pi_{k-1}(x),y)\in \ns_{k-1}\times_{\nss_{k-1}} \nss$. As $\pi_{k-1}(y) = \varphi_{k-1}(\pi_{k-1}(x)) = \pi_{k-1}(\varphi(x))$ we have that there exists $z'\in \ab_k(\nss)$ such that $\varphi(x)+z'=y$. Take any $z\in \ab_k(\ns)$ such that $\phi_k(z)=z'$ and we have that $\Phi(x+z)=(\pi_{k-1}(x),y)$.

To complete the proof that $\Phi$ is a nilspace isomorphism, note that it now suffices to prove that $\Phi$ is cube-surjective, as then $\Phi^{-1}$ is easily deduced to be a morphism. Let $(\pi_{k-1}\co \q_1,\q_2)\in \cu^n(\ns_{k-1}\times_{\nss_{k-1}} \nss)$. In particular this means that $\pi_{k-1}\co \varphi \co \q_1 = \pi_{k-1} \co \q_2$ and therefore there exists $d'\in \cu^n(\mc{D}_k(\ab_k(\nss)))$ such that $\varphi \co \q_1+d'=\q_2$. By the surjectivity of $\phi_k$ there is $d\in \cu^n(\mc{D}_k(\ab_k(\ns)))$ such that $\phi_k \co d = d'$. Then $\q_1+d$ is a cube such that its image through $\Phi$ is $(\pi_{k-1}\co \q_1,\q_2)$, which gives us the desired surjectivity.

Finally, to see that this defines an extension, let $P:\ns_{k-1}\times_{\nss_{k-1}} \nss\to \nss$ be the map $(\pi_{k-1}(x),y)\mapsto y$. We leave as an exercise for the reader to check that $\varphi_{k-1}:\ns_{k-1}\to \nss_{k-1}$ defines a degree-$k-1$ extension of $\nss_{k-1}$ by the group $\ker(\phi_{k-1})$. In $\ns\cong \ns_{k-1}\times_{\nss_{k-1}} \nss$ we define the action of $\mc{D}_{k-1}(\ker(\phi_{k-1}))$ as $(\pi_{k-1}(x),y)+z:=(\pi_{k-1}(x)+z,y)$. Now, the only thing to check to prove that $P$ defines an extension is that given cubes $(\pi_{k-1}\co \q_1,\q_2),(\pi_{k-1}\co \q'_1,\q'_2)\in \cu^n(\ns_{k-1}\times_{\nss_{k-1}} \nss)$, if $P\co(\pi_{k-1}\co \q_1,\q_2)=P\co(\pi_{k-1}\co \q'_1,\q'_2)$, then $(\pi_{k-1}\co \q_1,\q_2)+d=(\pi_{k-1}\co \q'_1,\q'_2)$ for some $d\in\cu^n(\mc{D}_{k-1}(\ker(\phi_{k-1}))$. To prove this, note that $\q_2 = \q'_2$ and thus $\varphi_{k-1}\co \pi_{k-1}\co \q_1 = \varphi_{k-1}\co \pi_{k-1}\co \q'_1$. As $\varphi_{k-1}$ is a degree-$k-1$ extension, there exists $d\in \cu^n(\mc{D}_{k-1}(\ker(\phi_{k-1}))$ such that $\pi_{k-1}\co \q_1+d = \pi_{k-1}\co \q_1'$. Hence $(\pi_{k-1}\co \q_1,\q_2)+d=(\pi_{k-1}\co \q'_1,\q'_2)$.
\end{proof}

\begin{proof}[Proof of Proposition \ref{prop:struc-k=p+1}]
Let us outline the proof. Let $\ns$ be a $(p+1)$-step, $p$-homogeneous, \textsc{cfr} nilspace. By known theory (the case $k=p$) we know that $\ns_p\cong \prod_{i=1}^p \mc{D}_i(\mb{Z}_p^{a_i})$. Let us separate these into three terms: $\ns_p = \mc{D}_1(\mb{Z}_p^{a})\times \mc{D}_2(\mb{Z}_p^{b})\times Q$. Let $\nss:= \abph_{p,1}^{\,a} \times \abph_{p+1,2}^{\,b}\times Q$ with the natural map $L:\nss\to \ns_p$ defined as $(y_1,y_2,q) \mapsto (\pi_1(y_1),\pi_2(y_2),q)$. In particular note that $L_{p-1}:\nss_{p-1}\to\ns_{p-1}$ is an isomorphism. Let $T$ be the following subdirect product of $\ns$ and $\nss$:  
$T:=\nss\times_{\ns_p}\ns=\{(y,x)\in \nss\times\ns: L(y)=\pi_p(x) \}$.

It is easy to see that $T$ is a degree-$(p+1)$ extension of $\nss$. As $\nss\in \mc{Q}_{p,k=p+1}$ we know that this extension splits (by Proposition \ref{prop:ext-split}) and therefore $T\cong \abph_{p,1}^{\,a} \times \abph_{p+1,2}^{\,b}\times Q\times \mc{D}_{p+1}(\mb{Z}_p^n)$ where $Z_{p+1}(\ns) = \mb{Z}_p^n$. Let us denote by $\Psi:T=\nss\times_{\ns_p}\ns\to \ns$ the map $(y,x)\mapsto x$. This map is easily seen to be a fibration. Note also that $\Psi_{p-1}:T_{p-1}\to \ns_{p-1}$ is a nilspace isomorphism. To prove this, note that by Proposition \ref{prop:fac-fib-prod} we have that $T_{p-1}\simeq \nss_{p-1}\times_{\ns_{p-1}}\ns_{p-1}$ but as $L_{p-1}$ is an isomorphism, this space is simply $\ns_{p-1}$.

Now, the first thing we do is to factor by $\ker(\phi_{p+1})$ where $\phi_{p+1}$ is the $p+1$ structure morphism of $\Psi$. That is, consider the action of $\ker(\phi_{p+1})$ on $T$ and note that $\Phi:T\to \ns$ factors through this action and thus we have a fibration $\Psi':T/\ker(\phi_{p+1})\to \ns$. But by Proposition \ref{prop:factorization-last-str-group} we know that $T/\ker(\phi_{p+1})$ is an abelian group nilspace so what we have proved is that we can refine our covering of $\ns$ to a covering such that the only structure morphism which may be not an isomorphism is $\phi_p$.

Now we apply Lemma \ref{lem:Claim1} to $\Psi'$ and thus we obtain that $T/\ker(\phi_{p+1})$ defines a degree-$p$ extension of $\ns$ by an elementary abelian $p$-group. Consider the following diagram: 
\begin{center}
\begin{tikzpicture}
  \matrix (m) [matrix of math nodes,row sep=3em,column sep=4em,minimum width=2em]
  {
     T/\ker(\phi_{p+1}) & \ns  \\
     (T/\ker(\phi_{p+1}))_{p} & \ns_{p}. \\};
  \path[-stealth]
    (m-1-1) edge node [above] {$\Psi'$} (m-1-2)
    (m-1-1) edge node [right] {$\pi_{p}$} (m-2-1)
    (m-1-2) edge node [right] {$\pi_{p}$} (m-2-2)
    (m-2-1) edge node [above] {$\Psi'_{p}$} (m-2-2);
\end{tikzpicture}
\end{center}
By Proposition \ref{prop:proj-of-ext} we have that $T/\ker(\phi_{p+1})$ is isomorphic to the subdirect product $(T/\ker(\phi_{p+1}))_p\times_{\ns_p} \ns$. Also, by Proposition \ref{prop:proj-of-ext} we know that $\Psi'_p$ defined a degree-$p$ extension of $\ns_p$ and by Corollary \ref{cor:spli-klep} this extension splits. Hence, there exists a cross-section $s:\ns_p\to (T/\ker(\phi_{p+1}))_p$ which is a morphism.

We now have all the required ingredients. Let us define $f:\ns\to (T/\ker(\phi_{p+1}))_p\times_{\ns_p} \ns$ as $x\mapsto (s(\pi_p(x)),x)$. This is clearly a morphism and furthermore this defines a cross-section for the map $\Psi'$. Hence, the extension defined by $\Psi'$ splits and therefore there exists an integer $m\ge 0$ such that $\ns\times \mc{D}_p(\mb{Z}_p^m)\simeq (T/\ker(\phi_{p+1}))_p\times_{\ns_p} \ns \simeq T/\ker(\phi_{p+1})$ which is an abelian group nilspace (by Proposition \ref{prop:factorization-last-str-group}).\end{proof}

\begin{remark}\label{rem:open-Qs}
For $k>p+1$ we do not know whether there are more explicit descriptions of $p$-homogeneous nilspaces generalizing Theorem \ref{thm:intro-3}. There is in particular a possibility which we are not able to rule out in this paper, namely,  that all these $p$-homogeneous nilspaces could be not just fibration-images of abelian group nilspaces (as in Theorem \ref{thm:intro-3}), but actually be  isomorphic to abelian group nilspaces. Another approach to the problem of describing $p$-homogeneous nilspaces $\ns$ in more detail consists in examining, not the fibrations from simpler abelian nilspaces \emph{onto} $\ns$, but rather examining injective morphisms \emph{from} $\ns$ into simpler abelian nilspaces. This latter approach is explored in the next section, where it will help to make progress on Question \ref{Q:Abramov}.
\end{remark}

\section{Applications in ergodic theory}\label{sec:ergodic}

\subsection{Host--Kra factors of $\mb{F}_p^{\omega}$-systems as $p$-homogeneous nilspace systems.}\hfill\\
In this subsection we prove Theorem \ref{thm:main-ergodic-intro}, describing the $k$-th order Host--Kra factor of any ergodic $\mb{F}_p^\omega$-system as a compact nilspace system with the underlying nilspace being $p$-homogeneous. Our starting point is the following result from \cite{CScouplings}.

\begin{theorem}[Theorem 5.11 in \cite{CScouplings}]\label{thm:CCHK}
Let $G$ be a countable discrete group, let $G_{\bullet}$ be a filtration of finite degree on $G$, suppose that $G$ acts ergodically on a Borel probability space $(\Omega,\mc{A},\lambda)$, and let $k\in \mb{N}$. Then the $k$-th Host--Kra factor of $(\Omega,(G,G_{\bullet}))$ is isomorphic to an ergodic $k$-step filtered compact nilspace system $(\ns_k,(G,G_\bullet),\widehat{\gamma}_k)$. 
\end{theorem}
\noindent Thus $\gamma_k:\Omega\to \ns_k$ is a measure-preserving map and $\widehat{\gamma}_k$ is a filtered-group homomorphism $G\to \tran(\ns_k)$ (a group homomorphism such that $\widehat{\gamma}_k(G_i)\subset \tran_i(\ns_k)$ for all $i\geq 0$) and for every $g\in G$ we have the equivariance $\gamma_k\co g =_{\lambda} \widehat{\gamma}_k(g)\co \gamma_k $ (where $=_{\lambda}$ denotes equality up to a $\lambda$-null set). We refer to  \cite[Definition 3.31 and Lemma 3.32]{CScouplings} for the detailed definition of $\gamma_k$, and to  \cite[Theorem 4.5]{CScouplings} for the definition of $\widehat{\gamma}_k$. We shall apply Theorem \ref{thm:CCHK} with $G$ the additive group of $\mb{F}_p^\omega$. As usual in this paper, when the filtration on an abelian group $G$ is not explicitly mentioned, we are implicitly using the lower central series $G_0=G_1=G\geq G_i=\{0\}$, $\forall\,i\geq 2$. Accordingly, when we write $\cu^n(G)$ for an abelian group $G$ (rather than the more rigorous notation $\cu^n(G_\bullet)$), we are referring to the standard $n$-cubes on $G$ (i.e.\ the $n$-cubes relative to the lower central series on $G$).

The idea of the proof of Theorem \ref{thm:main-ergodic-intro} is to show that the map $\widehat{\gamma}_k$ induces an arbitrarily highly balanced morphism $\varphi\in\hom(\mc{D}_1(\mb{Z}_p^D),\ns_k)$; by Theorem \ref{thm:main1-intro} this will imply that the nilspace $\ns_k$ is $p$-homogeneous. In fact, this approach involving Theorem \ref{thm:CCHK} yields a rather strong form of ergodicity on this factor. To formalize this, recall that for every filtered nilspace system $(\ns,(G,G_\bullet))$, for each $n\in\mb{Z}_{\geq 0}$ the cube-set $\cu^n(G_\bullet)$ has a natural action on $\cu^n(\ns)$ thanks to the fact that $G$ acts by translations (see \cite[Definition 5.10]{CScouplings}). 

\begin{defn}[Fully ergodic nilspace system]
A filtered nilspace system $(\ns,(G,G_\bullet))$ is \emph{fully ergodic} if for every $n\geq 0$ the action of $\cu^n(G_\bullet)$ on $\cu^n(\ns)$ is uniquely ergodic.
\end{defn}
\noindent Note that the special case $n=0$ here means that $G$ itself acts uniquely ergodically on $\ns$. One of the main results of this section is the following theorem, which directly implies Theorem \ref{thm:main-ergodic-intro}:
\begin{theorem}\label{thm:main-ergodic-strong}
For every $k\in \mb{N}$, the $k$-th Host--Kra factor of every ergodic $\mb{F}_p^{\omega}$-system is isomorphic \textup{(}as a measure-preserving system\textup{)} to a $p$-homogeneous $k$-step nilspace system $(\ns, \mb{F}_p^\omega)$ that is fully ergodic.
\end{theorem}
\noindent Indeed, we will deduce that the nilspace system is $p$-homogeneous in Theorem \ref{thm:main-ergodic-strong} as a consequence of being fully ergodic. Thus, let us start by proving the latter property. 

Recall from \cite[\S 5]{CScouplings} that for any filtered group $(G,G_\bullet)$ such that $G$ acts by measure-preserving transformations on the probability space $\Omega$, we can define the sequence of associated \emph{Host--Kra couplings}, generalizing the sequence of cubic measures introduced for $G=\mb{Z}$ in \cite{HK-non-conv}; see \cite[Definition 5.4]{CScouplings}. We then have the following fact.
\begin{proposition}
Let $(\Omega,\lambda,G)$ be an ergodic $G$-system where $G$ is a countable discrete group, let $G_{\bullet}$ be a filtration on $G$, and let $(\mu^{\db{n}})_{n\ge 0}$ be the associated sequence of Host--Kra couplings. Then $\cu^n(G_\bullet)$ acts ergodically on the probability space $(\Omega^{\db{n}},\mu^{\db{n}})$ for every $n$.
\end{proposition}

\begin{proof}
This follows by a straightforward generalization of the arguments used to prove \cite[Corollary 3.5]{HK-non-conv}.\end{proof}

\begin{corollary}\label{cor:erg-all-cu-n}
Let $(\Omega,\lambda,G)$ be an ergodic $G$-system where $G$ is a countable discrete group, let $G_{\bullet}$ be a filtration on $G$, and let $\ns_k$ be the associated $k$-th Host--Kra factor. Then $\cu^n(\widehat{\gamma}_k(G))$ acts ergodically on $\cu^n(\ns_k)$ for all $n\ge 0$.
\end{corollary}

\begin{proof} This follows immediately from the definition of the Host--Kra factor.\end{proof}

\noindent Now we can prove the desired full ergodicity.

\begin{lemma}\label{lem:unique-erg}
Let $\ns$ be a $k$-step compact nilspace, let $H$ be a countable subgroup of $\tran(\ns)$, with filtration $H_\bullet=(H_i)_{i\geq 0}$ defined by $H_i:=H\cap \tran_i(\ns)$, and suppose that for every $n\ge 0$ the action of $\cu^n(H_\bullet)$ on $\cu^n(\ns)$ is ergodic relative to the Haar measure $\mu_{\cu^n(\ns)}$. Then the system $(\cu^n(\ns),\cu^n(H_\bullet))$ is uniquely ergodic, with $\cu^n(H_\bullet)$-invariant measure $\mu_{\cu^n(\ns)}$.
\end{lemma}

\begin{proof}
We adapt \cite[p.\ 65, Lemma 4, and p.\ 66, Proposition 5]{HKbook} to the case of these $H$-actions. Following \cite{HKbook} we argue by induction on $k$. The case $k=0$ is trivial, as $\ns$ is then the 1-point nilspace. Let $s:\cu^n(\ns_{k-1})\to \cu^n(\ns_k)$ be a Borel cross-section (as provided by the proof of \cite[Lemma 2.4.5]{Cand:Notes2}). Let $\ab_k=\ab_k(\ns)$, and let
\begin{eqnarray*}
    \Phi: \; \cu^n(\ns_{k-1})\times \cu^n(\mc{D}_k(\ab_k)) & \to & \cu^n(\ns) \\
       (\q',z) \hspace{1.8 cm} & \mapsto & s(\q')+z.
\end{eqnarray*}
This is a Borel-measurable map (relative to the product topology on its domain), and it is bijective, with inverse $\Phi^{-1}:\q \mapsto (\pi_{k-1}\co \q, \q-s(\pi_{k-1} \co \q))$. For any $T\in \cu^n(H_\bullet)$ we define $T':=\Phi^{-1}\co T \co \Phi$. As translations on $\ns$ commute with addition of elements of $\ab_k$ \cite[Lemma 3.2.37]{Cand:Notes1}, and $\pi_{k-1}\co T = T_{k-1}\co \pi_{k-1}$ for some $T_{k-1}\in \cu^n(\tran(\ns_{k-1}))$ \cite[Definition 3.3.1 and Proposition 3.3.2]{Cand:Notes1}, we have $T'(\q',z)= (T_{k-1}( \q'),[T(s(\q'))-s(T_{k-1}(\q'))]+z)$ for all $\q'\in\cu^n(\ns_{k-1})$, $z\in \cu^n(\mc{D}_k(\ab_k))$.

Let $\Upsilon^n:=\{T':T\in \cu^n(H_\bullet)\}$. We shall now show that 
if $\mu$ is any ergodic $\Upsilon^n$-invariant Borel probability measure on $\cu^n(\ns_{k-1})\times \cu^n(\ab_k)$, then $\mu= \mu_{\cu^n(\ns_{k-1})}\times m_{\cu^n(\ab_k)}$ where $\mu_{\cu^n(\ns_{k-1})}$ and $ m_{\cu^n(\ab_k)}$ are the Haar measures on $\cu^n(\ns_{k-1})$ and $\cu^n(\ab_k)$ respectively. This will prove the claimed unique ergodicity, as any $\Upsilon^n$-invariant Borel probability measure on $\cu^n(\ns_{k-1})\times \cu^n(\ab_k)$ is the convex combination of ergodic $\Upsilon^n$-invariant Borel probability measures by \cite[Theorem 4.2.6]{HassKat}. It will also establish that $\mu=\mu_{\cu^n(\ns)}\co \Phi^{-1}$, by construction of the Haar measure $\mu_{\cu^n(\ns)}$ (see \cite[Proposition 2.2.5]{Cand:Notes2}). 

Note that if $\pi:\cu^n(\ns_{k-1})\times \cu^n(\ab_k)\to \cu^n(\ns_{k-1})$ is the projection to the first coordinate, we have that $\mu$ is a $\cu^n(H_{k-1})$-invariant measure of $\cu^n(\ns_{k-1})$ and by induction on $k$ this measure $\nu$ is precisely the Haar measure on $\cu^n(\ns_{k-1})$. 

Following the proof of \cite[p.\ 63, Lemma 4]{HKbook}, we shall use the  disintegration of the measure $\mu$ with respect to $\pi$, that is $\mu = \int \delta_y \times \mu_y \; \ud\nu(y)$. Fix any $T'\in \Upsilon^n$ and let $\rho_{T}(y):=T(s(y))-s(T_{k-1}(y))$ for any $y\in \cu^n(\ns_{k-1})$. Since $\mu$ is $T'$-invariant, we have that for $\nu$-almost every $y\in \cu^n(\ns_{k-1})$ we have $\mu_{T_{k-1}y}=\delta_{\rho_T(y)}\ast \mu_y$, in the sense that for any continuous $f:\cu^n(\mc{D}_k(\ab_k))\to \mb{C}$ we have $\int f(t) \ud \mu_{T_{k-1}y}(t) = \int \int f(a+b)\; \ud\delta_{\rho_T(y)}(a) \; \ud\mu_y(b)$.

Now, for any character $\chi:\cu^n(\mc{D}_k(\ab_k))\to \mb{C}$ the Fourier-Stieltjes coefficient $\widehat{\mu_y}(\chi)$ is well defined for  $\nu$-almost all $y\in \cu^n(\ns_{k-1})$, and equals
\[
\widehat{\mu_y}(\chi) = \int_{\cu^n(\mc{D}_k(\ab_k))} \chi(h) \; \ud\mu_y(h).
\]
Following \cite{HKbook}, we define $\phi_{\chi}(y,g):=\overline{\chi}(g)\widehat{\mu_y}(\chi)$ for all $g\in \cu^n(\mc{D}_k(\ab_k))$ and $\nu$-almost all $y\in \cu^n(\ns_{k-1})$. We have that for all $g,h\in \cu^n(\mc{D}_k(\ab_k))$ and $\nu$-almost all $y\in \cu^n(\ns_{k-1})$ we have
$\phi_{\chi}(y,g+h)=\overline{\chi}(h)\phi_{\chi}(y,g)$. Thus, using the fact that $\mu_{T_{k-1}y}=\delta_{\rho_T(y)}\ast \mu_y$, we conclude that $\phi_{\chi}(T'(y,g)) = \phi_{\chi}(y,g)$ for all $g\in \cu^n(\mc{D}_k(\ab_k))$ and $\nu$-almost all $y\in \cu^n(\ns_{k-1})$. We can repeat this argument with the \emph{countably} many elements of $\Upsilon^n$ and deduce that for all $g\in \cu^n(\mc{D}_k(\ab_k))$ and $\nu$-almost all $y\in \cu^n(\ns_{k-1})$, $\phi_{\chi}(T'(y,g)) = \phi_{\chi}(y,g)$. Thus, for $\mu$-almost all $(y,g)$ we have that for any $T'\in \Upsilon^n$, $\phi_{\chi}(T'(y,g)) = \phi_{\chi}(y,g)$. By ergodicity of $\mu$ we conclude that $\phi_{\chi}$ is constant for $\mu$-almost all $(y,g)$. If we denote this constant by $c_{\chi}$, it is easy to deduce that $c_{\chi} = \overline{\chi}(h)c_{\chi}$. Therefore, if $\chi$ is not the trivial character, $c_{\chi}=0$ for almost all $y\in \cu^n(\ns_{k-1})$, which in turn implies that each measure $\mu_y$ is the Haar measure on $\cu^n(\mc{D}_k(\ab_k))$ for almost all $y\in \cu^n(\ns_{k-1})$.
\end{proof}

\noindent The main ingredient for the proof of Theorem \ref{thm:main-ergodic-strong} is the following result, telling us that a fully ergodic action can be used to obtain arbitrarily balanced morphisms from the acting group to the Host-Kra factor.

\begin{proposition}\label{prop:ergodic-tool}
Let $\Omega$ be an ergodic $\mb{F}_p^{\omega}$-system and let $\ns_k$ be the corresponding $k$-th Host--Kra factor. Then for every $x\in \ns_k$ and $b>0$,  there exists $D=D(b,\Omega,\ns_k,x)$ such that $\phi:\mb{F}_p^D\to \ns_k$, $g\mapsto \widehat{\gamma_k}(g)(x)$ is a $b$-balanced morphism in $\hom(\mc{D}_1(\mb{Z}_p^D),\ns_k)$.
\end{proposition}

\begin{proof}[Proof of Theorem \ref{thm:main-ergodic-strong} using Proposition \ref{prop:ergodic-tool}] By Theorem \ref{thm:CCHK} we know that the $k$-th Host-Kra factor is isomorphic to a $k$-step compact nilspace system. By Corollary \ref{cor:erg-all-cu-n} and Lemma \ref{lem:unique-erg} we know that the action of $\cu^n(\mb{F}_p^{\omega})$ on $\cu^n(\ns_k)$ (via $\widehat{\gamma}_k$) is uniquely ergodic for all $n\ge 0$. We now prove that $\ns_k$ is $p$-homogeneous.

Fix any $x\in \ns_k$. By the inverse limit theorem (see \cite[Theorem 2.7.3]{Cand:Notes2}) we have $\ns_k= \varprojlim \ns_{k,i}$ where $\ns_{k,i}$ are $k$-step \textsc{cfr} nilspaces. Let $\psi_i:\ns_k\to \ns_{k,i}$ be the $i$-th limit map in this inverse limit, and recall that $\psi_i$ is a fibration. Consider the $b$-balanced morphism provided by Proposition \ref{prop:ergodic-tool} (supposing some -- any -- metrics have been fixed on $\ns_k$, $\ns_{k,i}$ and using Remark \ref{rem:metconv}). Arguing as in the proof of Proposition \ref{prop:b-bal-p-hom-finite}, we deduce that $\psi_i\co \phi$ is $b'$-balanced for some parameter $b'(b)$ which tends to $0$ as $b\to 0$. Thus, given any $b' > 0$, we have that there exists $D=D(b',\Omega, \ns_k,x,i)$ such that $\psi_i \co \phi$ is $b'$-balanced. Choosing $b'=b'(\ns_{k,i},p)$ as given by Theorem \ref{thm:main1-intro}, we can conclude that for sufficiently large $D=D(\Omega,\ns_k,p,x,i)$ the compact nilspace $\ns_{k,i}$ is $p$-homogeneous. Since this holds for every $i\in \mb{N}$, we deduce that $\ns_k$ is the inverse limit of $p$-homogeneous nilspaces, which implies that $\ns_k$ itself is $p$-homogeneous (this follows easily from the definitions).
\end{proof}

\noindent In order to prove Proposition \ref{prop:ergodic-tool} we will rely on the following technical result:

\begin{lemma}\label{lem:weak-star-conv}
Suppose that the system $(\cu^n(\ns_k),\cu^n(\widehat{\gamma}_k(G)))$ is uniquely ergodic. Let $(G_D)_{D\ge 0}$ be a F\o{}lner sequence for the group $G$ such that $G= \bigcup_{D=1}^{\infty} G_D$. Then for any cube $q\in \cu^n(\ns_k)$, the sampling measures $\mb{E}_{\q\in \cu^n(G_D)}\delta_{\widehat{\gamma}_k(\q)\cdot q}$ on $\cu^n(\ns_k)$ converge in the weak topology to $\mu_{\cu^n(\ns_k)}$ as $D\to \infty$.
\end{lemma}

\begin{proof}
This follows by similar arguments as in \cite[p.\ 30, Proposition 2]{HKbook} (see also \cite[p.\ 87 \S 4.3.\ a.]{HassKat}).
\end{proof}

\begin{proof}[Proof of Proposition \ref{prop:ergodic-tool}] By Theorem \ref{thm:CCHK} we know that the $k$-th Host-Kra factor is isomorphic to a $k$-step compact nilspace system. By Corollary \ref{cor:erg-all-cu-n} and Lemma \ref{lem:unique-erg} we know that the action of $\cu^n(\mb{F}_p^{\omega})$ on $\cu^n(\ns_k)$ (via $\widehat{\gamma}_k$) is uniquely ergodic for all $n\ge 0$.

Fix some point $x\in \ns_k$. For any $D$ we can define the map $\varphi:\mb{Z}_p^D\to \ns_k$ as $x\mapsto \widehat{\gamma}_k(g)\cdot x$. We need to prove that, given $b>0$, there exists $D$ such that for every $n\le 1/b$ we have
\begin{equation}\label{eq:cubemetineq}
d_n(\mu_{\cu^n(\ns_k)},\mu_{\cu^n(\mc{D}_1(\mb{F}_p^D))}\co (\varphi^{\db{n}})^{-1})<b,
\end{equation}
where $d_n$ is a prescribed metric on $\mc{P}(\cu^n(\ns_k))$ (see Remark \ref{rem:metconv}). We now apply Lemma \ref{lem:weak-star-conv} with $G=\mb{F}_p^{\omega}$, with $G_D=\mb{F}_p^D$ for each $D\geq 0$ (naturally embedded as a subgroup $\mb{F}_p^{\omega}$ so that $(G_D)_{D\geq 0}$ is a F\o{}lner sequence in $\mb{F}_p^{\omega}$), and $q\in\cu^n(\ns_k)$ the cube with constant value $x$. Thus, for each $n$ there is $D_n$ such that for $D\ge D_n$ the inequality \eqref{eq:cubemetineq} holds. Taking $D\ge \max_{n\le1/b}(D_n)$, the result follows.
\end{proof}

We can now apply this straightaway to describe the $k$-th Host-Kra factors for $k\leq p$.

\begin{proof}[Proof of Theorem \ref{thm:high-char-ergodic}]
The result follows from combining theorems \ref{thm:main-ergodic-intro} and \ref{thm:intro-3}.
\end{proof}

\noindent We end this subsection with the following explicit description of the translation group of the nilspaces occurring in Theorem \ref{thm:high-char-ergodic}.

\begin{theorem}\label{thm:HiCharTransDesc}
Let $k\le p$, for each $i\in [k]$ let $a_i\in \mb{N}\cup\{\infty\}$, and let $\ns$ be the compact $k$-step $p$-homogeneous nilspace $\prod_{i=1}^k\mc{D}_i(\mb{Z}_p^{a_i})$. Then the translation group $\tran(\ns)$ can be identified as a set with the Cartesian product \footnote{Note that in \eqref{eq:trasdescript} the product signs outside the bracket indicate  Cartesian products, and the product sign inside the bracket indicates a product of nilspaces.}
\begin{equation}\label{eq:trasdescript}
\mb{Z}_p^{a_1}\times \prod_{i=2}^k\hom\Big(\prod_{j=1}^{i-1}\mc{D}_j(\mb{Z}_p^{a_j}),\mc{D}_{i-1}(\mb{Z}_p^{a_i})\Big),
\end{equation}
and the action of an element $(T_1,\ldots,T_k)$ in this product as a translation $\alpha\in \tran(\ns)$ is given by the formula
\begin{equation}
\alpha(x_1,\ldots,x_k)=(x_1,\ldots,x_k)+(T_1,T_2(x_1),T_3(x_1,x_2),\ldots,T_k(x_1,\ldots,x_{k-1})).
\end{equation}
\end{theorem}
\noindent The group operation on $\tran(\ns)$ can be expressed directly on the set \eqref{eq:trasdescript} by deducing it from the definition of the action. Note also that the morphism sets in \eqref{eq:trasdescript} are sets of polynomial maps (see e.g.\  \cite[Theorem 2.2.14]{Cand:Notes1}). Thus Theorem \ref{thm:HiCharTransDesc} describes $\tran(\ns)$ in terms of polynomials.

\begin{proof}
First we prove that all the functions described are indeed translations. We argue by induction on $k$. The case $k=1$ is clear since in this case the function just adds the constant $T_1$, and is thus indeed a translation on $\mc{D}_1(\mb{Z}_p^{a_1})$. For $k\geq 2$, by induction it suffices to check that the map $\alpha: x=(x_1,,\ldots,x_k)\mapsto x+\big(0,\ldots,0,T_k(x_1,\ldots,x_{k-1})\big)$ is a translation. Recall from \cite[\S 3.1.4]{Cand:Notes1} the notation for \emph{arrow spaces}: in particular if $f,g:\db{n}\to \ns$ are any two maps, we define the 1-arrow $\langle f,g \rangle_1:\db{n+1}\to \ns$ as the map such that for $v\in \db{n}$ we have $\langle f,g \rangle_1(v,0)=f(v)$ and $\langle f,g \rangle_1(v,1)=g(v)$. By \cite[Lemma 3.2.32]{Cand:Notes1} it suffices to show that for every $\q=(\q_1,\ldots,\q_k)\in \cu^n\big(\prod_{i=1}^{k}\mc{D}_i(\mb{Z}_p^{a_i})\big)$ we have $\langle \q,\alpha\co\q\rangle_1\in \cu^{n+1}(\ns)$. But  $\langle \q,\alpha\co\q\rangle_1=\langle \q,\q\rangle_1 + g$, where $g$ is the map $\db{n+1}\to \prod_{i=1}^k \mb{Z}_p^{a_i}$ with values of the form $g(v)=(0,\ldots,0,\langle 0,T_k\co (\q_1,\ldots,\q_{k-1})\rangle_1(v))$. Thus it suffices to prove that for every such $n$-cube $\q$ we have $\langle 0,T_k\co (\q_1,\ldots,\q_{k-1})\rangle_1\in \cu^n(\mc{D}_k(\mb{Z}_p^{a_k}))$. For this, by \cite[Lemma 2.2.19]{Cand:Notes1} it suffices to have $T_k\co (\q_1,\ldots,\q_{k-1})\in \cu^n(\mc{D}_{k-1}(\mb{Z}_p^{a_k}))$. But this is  precisely what is ensured by our assumption that $T_k\in\hom\big(\prod_{j=1}^{k-1}\mc{D}_j(\mb{Z}_p^{a_j}), \mc{D}_{k-1}(\mb{Z}_p^{a_k})\big)$. 

Now we prove the converse, namely that every translation $\alpha$ has the form claimed in the theorem. By induction on $k$ we can assume that $\alpha$ has this form at least in the first $k-1$ components, so $\alpha(x)=(x_1,\ldots,x_{k-1},0)+(T_1,T_2(x_1),\ldots,T_{k-1}(x_1,\ldots,x_{k-1}),g(x))$ for some map $g:\ns\to\mb{Z}_p^{a_k}$. We know that translations commute with the action of the last structure group, so $g(x_1,\ldots,x_k)=g(x_1,\ldots,x_{k-1},0)+x_k$. Now it suffices to show that $g': (x_1,\ldots,x_{k-1})\mapsto g(x_1,\ldots,x_{k-1},0)$ is in $\hom(\prod_{i=1}^{k-1}\mc{D}_i(\mb{Z}_p^{a_i}),\mc{D}_{k-1}(\mb{Z}_p^{a_k}))$, i.e., that for every $\q\in \cu^n(\prod_{i=1}^{k-1}\mc{D}_i(\mb{Z}_p^{a_i}))$ we have $g'\co \q\in \cu^n(\mc{D}_{k-1}(\mb{Z}_p^{a_k}))$. Let $\q^*$ be the cube in $\cu^n(\prod_{i=1}^k\mc{D}_i(\mb{Z}_p^{a_i}))$ defined by $\q^*(v)=(\q(v),0^{a_k})$ for $v\in\db{n}$, and consider the map $\langle\q^*,\alpha\co\q^*\rangle_1$. On one hand, by \cite[Lemma 3.2.32]{Cand:Notes1} this map is a cube (since $\alpha$ is a translation), and on the other hand, by the above inductive expression of $\alpha$, this map equals  $\langle\q^*,\q^*\rangle_1+\langle0,\q'\rangle_1+\langle 0,\q''\rangle_1$ for some cube $\q'=(\q'_1,\ldots,\q'_{k-1},0^{a_k})\in \cu^n(\prod_{i=1}^n\mc{D}_i(\mb{Z}_p^{a_i}))$, and where $\q''(v)=(0^{a_1},\ldots,0^{a_{k-1}},g'\co \q(v))$. Then $\langle 0,\q''\rangle_1$ is in $\cu^{n+1}(\prod_{i=1}^{k}\mc{D}_i(\mb{Z}_p^{a_i}))$, since it is the combination of cubes $\langle\q^*,\alpha\co\q^*\rangle_1-\langle\q^*, \q^*\rangle_1 - \langle0,\q'\rangle_1$. Hence $\langle 0,g'\co\q\rangle_1\in \cu^{n+1}(\mc{D}_k(\mb{Z}_p^{a_k}))$ and therefore $g'\co\q\in \cu^n(\mc{D}_{k-1}(\mb{Z}_p^{a_k}))$ by \cite[Lemma 2.2.19]{Cand:Notes1}, as required.
\end{proof}

\begin{remark}
Combining Theorem \ref{thm:HiCharTransDesc} with Theorem \ref{thm:high-char-ergodic} we refine the description of the $k$-th Host-Kra factor for $k\leq p$, in that the $\mb{F}_p^\omega$-action is given by a homomorphism $\widehat{\gamma}_k$ from $\mb{F}_p^{\omega}$ to the group $\tran(\ns)$ with the above explicit description. It would be interesting to examine such homomorphisms further, possibly to refine the description even more using other available properties (e.g.\ full ergodicity). This goes beyond our aims in this paper.
\end{remark}

\subsection{$k$-step $p$-homogeneous nilspace systems as Abramov systems for $k\leq p+1$}\hfill\smallskip\\
Given a measure preserving $G$-system $(X,G)$, a function $f:X\to \mb{C}$ in $L^{\infty}(X)$ and $g\in G$, the corresponding multiplicative derivative of $f$ is the function $\Delta_g f(x):=f(g\cdot x)\overline{f(x)}$ in $L^{\infty}(X)$. We recall the notion of Abramov systems from \cite[Definition 1.13]{BTZ}, named after Leonid M.\ Abramov, who studied this type of systems in the setting of $\mb{Z}$-actions \cite{Ab}. 

\begin{defn}[Abramov system]\label{def:abramov}
Let $(X,G)$ be a $G$-system for a countable discrete agelian group $G$, and let $k\ge 0$ be an integer. We say that $\phi\in L^{\infty}(X)$ is a \emph{phase polynomial of degree $\leq k$} if for all $g_1,\ldots,g_{k+1}\in G$ we have $\Delta_{g_1}\cdots \Delta_{g_{k+1}}f = 1$ almost surely on $X$. We say that $X$ is an \emph{Abramov system} of order $\leq k$ if the linear span of the phase polynomials of degree $\leq k$ is dense in $L^2(X)$.
\end{defn}

\noindent It is proved in \cite[Theorem 1.19]{BTZ} that for every ergodic $\mb{F}_p^\omega$-system $X$, the $k$-th Host--Kra factor of this system  (denoted $\mc{Z}_{<k+1}(X)$ in \cite{BTZ}) is an Abramov system of order $\leq k$ for $k<p$. In this subsection we prove Theorem \ref{thm:Abramov}, establishing  that the $k$-th Host--Kra factor is Abramov also in the two new cases $k=p$ and $k=p+1$. To this end, we shall first reduce the problem to a question about nilspace systems.

Recall that every compact nilspace $\ns$ has a compact metric topology, relative to which every translation in $\tran(\ns)$ is a homeomorphism on $\ns$. Thus for every discrete countable group $G$ acting on $\ns$ by translations, the nilspace system $(\ns,G)$ can be treated as a topological dynamical system. It is then natural to introduce the following topological variant of Abramov systems. Given a metric space $X$, let $C(X,\mb{C})$ denote the algebra of complex-valued continuous functions on $X$ (a unital $^*$-algebra) with the uniform norm.

\begin{defn}[Topological Abramov systems]
Let $X$ be a compact metric space and let $G$ be a group acting by  homeomorphisms on $X$. A \emph{continuous phase polynomial} of degree $\leq k$ on $X$ is a function $\phi\in C(X,\mb{C})$ such that $\Delta_{g_1}\cdots \Delta_{g_{k+1}}\phi(x) = 1$ for all $g_1,\ldots,g_{k+1}\in G$, $x\in X$. We say that $(X,G)$ is \emph{topological Abramov of order $\leq k$} if the algebra generated by the continuous phase polynomials of degree $\leq k$ is dense in $C(X,\mb{C})$.
\end{defn}
\noindent By standard density arguments it is readily shown that if $\mu$ is a Borel probability measure on the compact metric space $X$ and $(X,G)$ is a topological Abramov system of order $\leq k$ then, provided $G$ acts by transformations preserving $\mu$, we have that $(X,G)$ is Abramov of order $\le k$ as a measure-preserving system. Our approach to Question \ref{Q:Abramov} is to study the question of when a nilspace system $(\ns,G)$ is a topological Abramov system. To this end, one of the main steps in this subsection consists in reformulating the topological Abramov property of a nilspace system $(\ns,G)$ as the following property of $\ns$.

\begin{defn}[Sub-abelian compact nilspace]
A compact nilspace $\ns$ is \emph{sub-abelian} of order $\leq k$ if there exists a compact (second-countable) abelian group nilspace $\nss$ of step $\leq k$ and an injective continuous morphism $\phi:\ns\to \nss$.
\end{defn}
\noindent Here recall that an abelian group nilspace is a group nilspace $(G,G_\bullet)$ where $G$ is abelian, and that a group nilspace $(G,G_\bullet)$  is of step $\leq k$ if and only if the filtration $G_\bullet$ has degree $\leq k$. The above-mentioned reformulation of Question \ref{Q:Abramov} consists in the following result.

\begin{proposition}\label{prop:Abrachar}
Let $\ns$ be a $k$-step compact nilspace, and let $G$ be a countable discrete abelian group acting on $\ns$ via a homomorphism $\varphi:G\to \tran(\ns)$. If $\ns$ is sub-abelian of order $\leq k$, then $(\ns,G)$ is topological Abramov of order $\leq k$. Conversely, if $(\ns,G)$ is topological Abramov of order $\leq k$ and
fully ergodic, then $\ns$ is sub-abelian of order $\leq k$.
\end{proposition}
\noindent To prove this we shall use the following couple of lemmas. The first one is just a convenient reformulation of the sub-abelian property.
\begin{lemma}\label{lem:SAequiv}
A compact nilspace $\ns$ is sub-abelian of order $\leq k$ if and only if there is an injective continuous morphism $\phi:\ns \to \mc{D}_k(\mb{T}^\mb{N})$.
\end{lemma}
\begin{proof}
The backward implication is clear. For the forward implication, suppose that $Z$ is a compact abelian group with a filtration $Z_\bullet$ of degree $\leq k$ and that the associated group nilspace $\nss$ admits a continuous injective morphism. Then it suffices to show that there is a continuous injective morphism $\phi':\nss\to \mc{D}_k(\mb{T}^\mb{N})$. By second-countability of $Z$, the dual group $\wh{Z}$ is countable, so we can list its elements as $\chi_i$ for $i\in \mb{N}$ and then define a map $\phi':Z\to \mb{T}^{\mb{N}}$, $x\mapsto (\chi_i(x))_{i\in \mb{N}}$. From the properties of characters it follows that $\phi'$ is continuous and injective. We claim that $\phi'\in \hom(\nss,\mc{D}_k(\mb{T}^\mb{N}))$. To see this it suffices to show that each $\chi\in \wh{Z}$ is in $\hom(\nss,\mc{D}_k(\mb{T}))$ (since $\mc{D}_k(\mb{T}^\mb{N})$ is isomorphic to the product nilspace of countably many copies of $\mc{D}_k(\mb{T})$). But $(Z,Z_\bullet)$ being of step $\leq k$ implies that for any $(k+1)$-cube $\q$ on this nilspace, the Gray-code alternating sum $\sigma_{k+1}(\q)$ is 0 (see \cite[Proposition 2.2.25]{Cand:Notes1}). Since $\chi$ commutes with the operations in this sum, we have $\sigma_{k+1}(\chi\co \q)=0$ in $\mb{T}$, so $\chi\co \q\in \cu^{k+1}(\mc{D}_k(\mb{T}))$. This proves that $\chi$ is a morphism.
\end{proof}
\noindent The second lemma uses full ergodicity to upgrade any continuous polynomial phase to a nilspace morphism.
\begin{lemma}\label{lem:cts-approx}
Let $\ns$ be a compact $k$-step nilspace and let $G$ be a countable discrete abelian group such that $(\ns,G)$ is fully ergodic. Let $f:\ns\to\mb{C}$ be a continuous polynomial phase of degree $\leq k$. Then there is a continuous morphism $\phi\in \hom\big(\ns,\mc{D}_k(\mb{T})\big)$ such that $f(x)=e(\phi(x))$ for all $x\in \ns$.
\end{lemma}
\begin{proof}
The phase polynomial property with $g_1=\cdots=g_{k+1}=\id_G$ implies that $|f(x)|^{2^{k+1}}=1$ for all $x\in X$, so there is a continuous function $\phi:\ns\to \mb{T}$ such that $f(x)=e(\phi(x))$, and it follows that for all $g_1,\ldots,g_{k+1}\in G$, the additive derivative $\nabla_{g_1}\cdots \nabla_{g_{k+1}} \phi(x)$ equals $0\in \mb{T}$ for every $x\in \ns$. We shall deduce from this that $\phi\in \hom\big(\ns,\mc{D}_k(\mb{T})\big)$. 

We claim that for every $\delta>0$ the map $\phi$ is a $\delta$-quasimorphism $\ns\to \mc{D}_k(\mb{T})$, in the sense that for every $\q\in\cu^{k+1}(\ns)$ there exists $\q'\in \cu^{k+1}(\mc{D}_k(\mb{T}))$ such that $|\phi\co \q(v)-\q'(v)|_{\mb{T}}\leq \delta$ for all $v\in \db{k+1}$ (see \cite[Definition 2.8.1]{Cand:Notes2}), where $|x|_{\mb{T}}$ denotes as usual the distance from $x\in\mb{T}$ to the nearest integer. To prove the claim, given any $\q\in \cu^{k+1}(\ns)$, let $B_{\delta'}(\q)$ be the open ball of center $\q$ and radius $\delta'$ in the $\ell^\infty$ norm in $\cu^{k+1}(\ns)$. 
  By unique ergodicity of the action of $\cu^{k+1}(G)$ on $\cu^{k+1}(\ns)$, if we fix any $x\in \ns$ and let $\q_0\in \cu^{k+1}(\ns)$ be the constant cube with value $x$, then the orbit of $\q_0$ under the action of $\cu^{k+1}(G)$ is dense. In particular there exists $\tilde \q\in \cu^{k+1}(G)$ such that $\tilde \q \cdot \q_0 \in B_{\delta'}(\q)$. Choosing $\delta'$ small enough, this implies $|\phi (\q(v))- \phi (\tilde \q \cdot \q_0(v))|_{\mb{T}}\leq \delta$ for all $v\in \db{k+1}$. Now note that the phase polynomial property of $\phi$ implies that $\phi\co (\tilde \q \cdot \q_0)\in \cu^{k+1}(\mc{D}_k(\mb{T}))$. This proves our claim.
 
By \cite[Theorem 2.8.2]{Cand:Notes2}, there is a continuous morphism $\phi_\delta:\ns\to \mc{D}_k(\mb{T})$ such that $|\phi(x)-\phi_\delta(x)|_{\mb{T}}\leq \varepsilon$ for every $x\in \ns$, where $\varepsilon(\delta)\to 0$ as $\delta\to 0$. Applying this for each $\delta$ in the sequence $(\delta_n:=1/n)_{n\in \mb{N}}$, we obtain a sequence of continuous morphisms $\phi_n:\ns\to \mc{D}_k(\mb{T})$ such that $\sup_{x\in \mb{T}}|\phi_n(x)- \phi(x)|_{\mb{T}}\to 0$ as $n\to\infty$. By the compactness (hence closure) of each cube set $\cu^n(\ns)$, we deduce that $\phi$ is a morphism.
\end{proof}

\begin{proof}[Proof of Proposition \ref{prop:Abrachar}]
Suppose that $\ns$ is sub-abelian of order $\leq k$ and let $\phi:\ns\to \mc{D}_k(\mb{T}^\mb{N})$ be an injective morphism. For each $i\in \mb{N}$ let $\pi_i:\mc{D}_k(\mb{T}^\mb{N})\to \mc{D}_k(\mb{T})$ be the projection to the $i$-th coordinate. Then for every $i$ and every character $\chi\in \widehat{\mb{T}}$,  the function $\chi\co \pi_i\co \phi:\ns\to \mb{C}$ is in $C(\ns,\mb{C})$, and the morphism property of $\phi$ implies that this function is a phase polynomial on $(\ns,G)$. Moreover, the injectivity of $\phi$ implies that the set of functions $S=\{\chi\co \pi_i\co \phi: i\in \mb{N}, \chi \in \widehat{\mb{T}}\}$ separates the points of $\ns$. By the Stone-Weierstrass theorem \cite[Ch.\ 6, Theorem 10]{Bol}, the unital $^*$-algebra generated by $S$ is dense in $C(\ns,\mb{C})$ and so $(\ns,G)$ is topological Abramov of order $\leq k$.

To prove the claim in the converse direction, suppose that $(\ns,G)$ is topological Abramov of order $\leq k$. Since $\ns$ is a compact metric space, the space $C(\ns,\mb{C})$ is separable, so there is a sequence $(h_i)_{i\in\mb{N}}\in C(\ns,\mb{C})$ that is dense in $C(\ns,\mb{C})$. For each $i\in \mb{N}$ and each $n\in \mb{N}$, there is then a finite combination of continuous phase polynomials on $\ns$ that is within distance $1/n$ of $h_i$ in $C(\ns,\mb{C})$. We thus obtain a countable collection $(f_j)_{j\in\mb{N}}$ of phase polynomials whose linear span is dense in $C(\ns,\mb{C})$. By Lemma \ref{lem:cts-approx}, for every $j$ there is a continuous morphism $\phi_j:\ns\to\mc{D}_k(\mb{T})$ such that $f_j=e\co \phi_j$. Let $\phi:\ns\to\mc{D}_k(\mb{T}^{\mb{N}})$ be the continuous morphism $x\mapsto  \big(\phi_j(x)\big)_{j\in \mb{N}}$. It remains only to see that $\phi$ is injective. This is equivalent to the injectivity of the map $F:\ns\to\mb{C}^\mb{N}$, $x\mapsto (f_j(x))_{j\in \mb{N}}$. The latter injectivity follows from the density of the linear span of $(f_j)_{j\in\mb{N}}$. Indeed, suppose for a contradiction that there exist $x\neq y$ in $\ns$ satisfying $F(x)=F(y)$. Then every linear combination of functions $f_j$ has the same value on $x$ and $y$. By Urysohn's lemma there is a continuous real-valued function $f$ on $\ns$ equal to 1 on a closed neighbourhood $U$ of $x$ and equal to 0 on a closed neighbourhood $V$ of $y$ with $V\cap U= \emptyset$. Then there is a linear combination $f'$ of the $f_j$ that is within $\varepsilon$ of $f$ in $C(\ns,\mb{C})$ and therefore $f'(x)\geq 1-\varepsilon$ and $f'(y)\leq \varepsilon$, which contradicts $f'(x)=f'(y)$ if $\varepsilon<1/2$.
\end{proof}

\noindent Equipped with Proposition \ref{prop:Abrachar}, we can now prove Theorem \ref{thm:Abramov} by showing that the nilspaces involved in the theorem are sub-abelian. For the case $k=p+1$ of the theorem, we shall use the following  additional small lemma, which tells us that the sub-abelian property is stable under taking inverse limits.

\begin{lemma}\label{lem:invlimSA}
Suppose that a compact nilspace $\ns$ is the inverse limit of compact nilspaces that are all sub-abelian of order $\leq k$. Then $\ns$ is sub-abelian of order $\leq k$.
\end{lemma}
\begin{proof}
By assumption $\ns$ is the inverse limit of a strict inverse system of  compact nilspaces $\ns_i$, $i\in \mb{N}$, where each $\ns_i$ is sub-abelian. For every $i\in \mb{N}$ let $\psi_i:\ns\to \ns_i$ be the $i$-th limit map (thus $\psi_i$ is a nilspace fibration). For any fixed $i$, let $(\phi_{i,j})_{j\in \mb{N}}$ be a sequence of continuous morphisms $\ns\to\mc{D}_k(\mb{T})$ such that the morphism $\phi_i:\ns\to\mc{D}_k(\mb{T}^\mb{N})$, $x\mapsto (\phi_{i,j}(x))_{j\in \mb{N}}$ is injective. Let $\sigma:\mb{N}\to \mb{N}^2$ be a bijection and let $\phi:\ns\to\mb{T}^\mb{N}$ be the map $x\mapsto (\phi_{\sigma(n)_1,\sigma(n)_2})_{n\in \mb{N}}$. Since the limits maps $\psi_i$ separate the points of $\ns$ and for each $i$ the maps $\phi_{i,j}$, $j\in \mb{N}$ separate the points of $\ns_i$, we deduce that $\phi$ is injective.
\end{proof}

We can now prove the main result of this section.

\begin{proof}[Proof of Theorem \ref{thm:Abramov}]
By Theorem \ref{thm:main-ergodic-intro} the $k$-th Host--Kra factor of an ergodic $\mb{F}_p^\omega$-system is isomorphic (as a measure-preserving $G$-system) to a $p$-homogeneous $k$-step nilspace system $(\ns,\mb{F}_p^\omega)$. For $k\leq p$, by Theorem \ref{thm:intro-3} the nilspace $\ns$ is a  $k$-step abelian group nilspace, so it is sub-abelian. Hence, by Proposition \ref{prop:Abrachar}, the nilspace system $(X,\mb{F}_p^\omega)$ is topological Abramov of order $\le k$ as required. For $k=p+1$, note first that by the inverse limit theorem for compact nilspaces, and Lemma \ref{lem:fibs-preserve-phoms}, $\ns$ is an inverse limit of $p$-homogeneous $k$-step $\textsc{cfr}$ nilspaces $\ns_j$, $j\in \mb{N}$. By Proposition \ref{prop:struc-k=p+1}, each $\ns_j$ is sub-abelian. Then $\ns$ is sub-abelian by Lemma \ref{lem:invlimSA}. By Proposition \ref{prop:Abrachar}, the result follows.
\end{proof}
\noindent Given the above results, a plausible way to answer Question \ref{Q:Abramov} for general $k$ may be to answer the following more specific question purely about $p$-homogeneous nilspaces, which also has the advantage of reducing the problem to a question concerning \emph{finite} structures.
\begin{question}\label{Q:SAphom}
Is every finite $p$-homogeneous $k$-step nilspace sub-abelian of order $\leq k$?
\end{question}
\noindent Our affirmative answer for $k=p+1$ relied on Proposition \ref{prop:struc-k=p+1}, which in turn relies on technical results including Proposition \ref{prop:factorization-last-str-group}. Generalizing these results to larger values of $k>p$ did not seem to be a simple task (see Remark \ref{rem:gendiff}), and we do not pursue this approach to Question \ref{Q:Abramov} further in this paper.

\section{\hspace*{-0.1cm}Regularity and inverse theorems for Gowers norms in characteristic $p$}\label{sec:inverse}

\noindent Recall that there are countably many isomorphism classes of \textsc{cfr} nilspaces (see \cite{CamSzeg} or \cite[Theorem 2.6.1]{Cand:Notes2}). This enables us to define a notion of complexity for \textsc{cfr} $k$-step nilspaces as a bijection from $\mb{N}$ to the set of isomorphism classes of such nilspaces. Throughout this section, we assume that some (any) such notion of complexity has been fixed. Thus for each $k$ we have fixed a sequence $(\nss\!\sbr{i})_{i\in \mb{N}}$ of $k$-step \textsc{cfr} nilspaces such that for every $k$-step \textsc{cfr} nilspace $\nss$ there is $i$ such that $\nss$ is isomorphic (as a compact nilspace) to $\nss\!\sbr{i}$; we then write $\textrm{Comp}(\nss)\leq m$ to mean that $i\leq m$. We shall also assume that some (any) compatible metric $d_i$ has been fixed on each nilspace $\nss\!\sbr{i}$. This fixes a meaning for the notion of balanced morphism into $\nss\!\sbr{i}$ for each $i$, using Remark \ref{rem:metconv}.

Let us also recall the notion of nilspace polynomials from \cite{CSinverse}, which constitute a general class of functions usable for inverse theorems for Gowers norms in various settings, and let us specify the special case of this notion in the characteristic-$p$ setting.

\begin{defn}[Nilspace polynomials]
Let $\ns$ be a compact nilspace. A function $f:\ns\to\mb{C}$ is a \emph{nilspace polynomial} of degree $k$ if $f=F\co\phi$ where $\phi:\ns\to \nss$ is a continuous morphism, $\nss$ is a $k$-step \textsc{cfr} nilspace, and $F$ is continuous. If $d$ is a compatible metric on $\nss$, then we say that the nilspace polynomial $F\co\phi$ is \emph{$b$-balanced} (with respect to $d$) if the morphism $\phi$ is $b$-balanced (using the metrics induced by $d$ as per Remark \ref{rem:metconv}). For a prime $p$, we say $F\co\phi$ is a \emph{$p$-homogeneous} nilspace polynomial if $\nss$ is $p$-homogeneous. 
\end{defn}
\noindent Having fixed a complexity notion $(\nss\!\sbr{i})_{i\in \mb{N}}$ as above, we say that a nilspace polynomial $f$ of degree $k$ on $\ns$ has \emph{complexity at most} $m$, denoted $\textrm{Comp}(f)\leq m$, if $f=F\co\phi$ where $\phi:\ns\to\nss\!\sbr{i}$ for some $i\leq m$ and $F$ has Lipschitz constant $\leq m$ (relative to the metric $d_i$ that we have fixed on $\nss\!\sbr{i}$). 

Our main aim in this section is to deduce the Tao-Ziegler inverse theorem from \cite{T&Z-Low}, by combining the results on $p$-homogeneous nilspaces in the present paper with the following (special case of the) general inverse theorem \cite[Theorem 5.2]{CSinverse}.

\begin{theorem}\label{thm:gen-inverse}
Let $k\in \mb{N}$, and let $b:\mb{R}_{>0}\to \mb{R}_{>0}$ be an arbitrary function. For every $\delta\in (0,1]$ there is $M>0$ such that for every \textsc{cfr} coset nilspace $\ns$, and every 1-bounded Borel function $f:\ns\to \mb{C}$ such that $\|f\|_{U^{k+1}}\geq \delta$, for some $m\leq M$ there is a $b(m)$-balanced 1-bounded nilspace-polynomial $F\co\phi$ of degree $k$ and complexity at most $m$ such that $\langle f, F\co\phi\rangle \geq \delta^{2^{k+1}}/2$.
\end{theorem}

\noindent To recall the Tao--Ziegler inverse theorem, we first recall from \cite{T&Z-Low} the notion of a (non-classical) polynomial on a vector space $\mb{F}_p^n$.
\begin{defn}[Polynomials]\label{def:poly}
Let $k \geq 0$ be an integer, and let $\ab$ be an abelian group.  A function $P: \mb{F}_p^n \to \ab$ is said to be a \emph{polynomial} of degree $\leq k$ if
\[
\forall\, h_1,\ldots,h_{k+1},x \in \mb{F}_p^n,\quad \Delta_{h_1} \ldots \Delta_{h_{k+1}} P(x) = 0,
\]
where $\Delta_h P(x) := P(x+h)-P(x)$ is the additive derivative of $P$ in the direction $h$. The space of polynomials of degree $\leq k$ is denoted by $\poly_{\leq k}(\mb{F}_p^n \to \ab)$.
\end{defn}
\noindent We now state the inverse theorem for vector spaces over $\mb{F}_p^n$ that we shall prove, which implies the Tao-Ziegler inverse theorem (stated as Conjecture 1.10 in \cite{T&Z-Low}). Recall that for $N\in \mb{N}$ we denote by $\frac{1}{N}\cdot \mb{Z}_N$ the subgroup of $\mb{T}$ isomorphic to $\mb{Z}_N$.
\begin{theorem}\label{thm:InvOver-FF}
Let $\delta > 0$, let $k\geq 0$, let $p$ be a prime, and let $r=\lfloor \frac{k-1}{p-1}\rfloor +1$. Then there exists $\varepsilon =\varepsilon_{\delta,k,p} > 0$ such that for every 1-bounded function $f : \mb{F}_p^n \to \mb{C}$ with $\| f\|_{U^{k+1}}\geq \delta$, there exists $P\in \poly_k(\mb{F}_p^n\to \frac{1}{p^r}\cdot \mb{Z}_{p^r})$ such that $|\mb{E}_{x\in \mb{F}_p^n} f(x) e(-P (x))|\geq \varepsilon$.
\end{theorem}

\noindent To prove this we first establish the following fact, which uses Theorem \ref{thm:general-p-hom-intro} to describe $p$-homogeneous nilspace polynomials in terms of phase polynomials on vector spaces $\mb{F}_p^n$.

\begin{lemma}\label{lem:phomchar}
For every prime $p$ and $k\in\mb{N}$, there is an increasing function $D:\mb{N}\to\mb{N}$ with the following property. Let $f$ be a $1$-bounded $p$-homogeneous nilspace polynomial of degree $k$ and complexity at most $m$ on $\mb{F}_p^n$. Then for some $R\leq D(m)$, for each $i\in [R]$ there is $\lambda_i\in \mb{C}$, $|\lambda_i|\leq 1$, and $P_i\in \poly_k(\mb{F}_p^n\to \frac{1}{p^r}\cdot \mb{Z}_{p^r})$, where $r=\lfloor\frac{k-1}{p-1}\rfloor +1$, such that $f=\sum_{i=1}^R\lambda_i\, e\co P_i$.
\end{lemma}

\begin{proof}
We have $f = F\co \phi$ where $\phi:\mb{F}_p^n\to \ns$ is a morphism to a $k$-step $p$-homogeneous finite nilspace $\ns$, and $F$ is 1-bounded. By Theorem \ref{thm:general-p-hom-intro}, there is a fibration $\psi:\nss\to \ns$ where $\nss=\prod_{\ell=1}^k \abph_{k,\ell}^{\,a_\ell}$, where $a_\ell \in \mb{Z}_{\ge 0}$ for $\ell\in [k]$, and a morphism $g:\mc{D}_1(\mb{Z}_p^n)\to \nss$, such that $\phi=\psi\co g$. Letting $F':=F\co \psi:\nss\to\mb{C}$, we have $f=F'\co g$. By Definition \ref{def:bblocks-intro}, we know that $\nss$ is a direct product of filtered cyclic groups $\mb{Z}_{p^{\lfloor\frac{k-\ell}{p-1}\rfloor +1}}$, each of which can be isomorphically embedded in $\mb{Z}_{p^r}$. Indeed, for any $a\le r$ we can embed $\mb{Z}_{p^a}\to \mb{Z}_{p^r}$ via the monomorphism $i_a:j\mapsto p^{r-a}j$. Then, letting $p_a:\mb{Z}_{p^r}\to \mb{Z}_{p^a}$ be the map that takes every element of the form $p^{r-a}j$ to $j$ and the rest to 0, it is clear that $\pi_a \co i_a = \id_{\mb{Z}_{p^a}}$.  Let $R=R(\nss):=\sum_{\ell=1}^k a_\ell$, let $i:\nss \to \abph_{k,1}^R$ be the product of these monomorphisms, and let $\pi:\mb{Z}_{p^r}^R \to \nss$ be the corresponding product of the projections. Then $\pi\co i=\id_{\nss}$, and $i$ is a morphism. Thus, letting $F'':=F'\co \pi$ and $\varphi:=i\co g$, we have that $\varphi\in\hom(\mc{D}_1(\mb{Z}_p^n),\abph_{k,1}^R)$, $F'':\abph_{k,1}^R\to \mb{C}$, and $F''\co \varphi = f$.

By classical Fourier analysis on the group $\mb{Z}_{p^r}^R$, we have $F''(t)=\sum_{\xi\in \mb{Z}_{p^r}^R} \wh{F''}(\xi) e(\xi\cdot t)$ where $\xi\cdot t$ is the standard $\mb{T}$-valued non-degenerate symmetric bilinear form on the finite abelian group $\mb{Z}_{p^r}^R$. Since $\mb{Z}_{p^r}^R$ is a power of $\mb{Z}_{p^r}$, the form $\xi\cdot t$ takes values in the subgroup $\frac{1}{p^r}\cdot \mb{Z}_{p^r}\subset \mb{T}$. We therefore have $f(x)=\sum_{\xi\in \mb{Z}_{p^r}^M} \lambda_\xi e(\xi\cdot \varphi(x))$, where the coefficients $\lambda_\xi=\wh{F''}(\xi)$ have modulus at most 1.

To complete the proof it now suffices to show that each function $x\mapsto \xi\cdot \varphi(x)$ is in $\poly_k(\mb{F}_p^n\to \frac{1}{p^r}\cdot \mb{Z}_{p^r})$. We have $\varphi(x)=\big(\varphi_1(x),\ldots,\varphi_R(x)\big)$, where it follows from the definitions that each map $\varphi_i$ is a morphism $\mc{D}_1(\mb{Z}_p^n)\to \abph_{k,1}$. It then follows from standard properties of such morphisms that $\varphi_i\in \poly_k(\mb{F}_p^n\to \mb{Z}_{p^r})$ for each $i$, and the result then follows from the group properties of $\poly_k(\mb{F}_p^n\to \mb{Z}_{p^r})$. The proof is now completed by setting $D(m):=\max\{|\mb{Z}_{p^r}^{R(\nss)}|: \textrm{\textsc{cfr} $p$-homogeneous nilspace $\nss$, }\textrm{Comp}(\nss)\leq m\}$.
\end{proof}
We can now prove the inverse theorem.
\begin{proof}[Proof of Theorem \ref{thm:InvOver-FF}]
We apply Theorem \ref{thm:gen-inverse} with the function $b$ to be fixed later. We thus obtain a nilspace polynomial $f_s:= F\co\phi$ such that $\langle f,f_s\rangle\geq \delta^{2^{k+1}}/2$.

We claim that, if $b$ decreases sufficiently fast as a function of $m$, then the nilspace polynomial $f_s$ is $p$-homogeneous. To see this, let us choose $b$ in Theorem \ref{thm:gen-inverse} so that for each $m$ we have $0< b(m)<\min \{ b'_{\nss\!\sbr{i},p}:  i\leq m\}$, where $b'_{\nss\!\sbr{i},p}$ is the constant given by Theorem \ref{thm:main1-intro} applied to the \textsc{cfr} nilspace $\nss\!\sbr{i}$ in our complexity notion (thus the minimum here is indeed positive). Then we conclude by Theorem \ref{thm:main1-intro} that  $f_s$ is a $p$-homogeneous nilspace polynomial of degree $k$ and complexity at most $m\leq M$. By Lemma \ref{lem:phomchar} we then have $f_s(x)=\sum_{i=1}^R\lambda_i e(P_i(x))$ where $R\leq D(M)$ for $D$ the function provided by that lemma (thus $R$ is bounded above depending only on $\delta,k,p$), and for each $i$ we have $|\lambda_i|\leq 1$ and $P_i\in \poly(\mb{F}_p^n\to \frac{1}{p^r}\cdot\mb{Z}_{p^r})$. Hence, for some $i\in [R]$ we have $|\langle f,e(P_i)\rangle|\geq \delta^{2^{k+1}}/(2 D(M))$. Letting $\varepsilon=\delta^{2^{k+1}}/(2 D(M))$, the result follows.
\end{proof}
\noindent We can also establish the special case for $k\leq p$ in terms of \emph{classical} phase polynomials, as mentioned at the end of the introduction.

\begin{theorem}\label{thm:InvThm-k=p}
Let $\delta > 0$, let $p$ be a prime, and let $0\le k \le p$. Then there exists $\varepsilon =\varepsilon_{\delta,k,p} > 0$ such that for every 1-bounded function $f : \mb{F}_p^n \to \mb{C}$ with $\| f\|_{U^{k+1}}\geq \delta$, there exists a classical polynomial $P\in \poly_k(\mb{F}_p^n\to \mb{F}_p)$ such that $|\mb{E}_{x\in \mb{F}_p^n} f(x) e(-P (x))|\geq \varepsilon$.
\end{theorem}

\begin{proof}
The argument is the same as the proof of Theorem \ref{thm:InvOver-FF}, except that instead of using Theorem \ref{thm:general-p-hom-intro} in the proof of Lemma \ref{lem:phomchar}, we use Proposition \ref{prop:high-char}.
\end{proof}

\noindent We finish by noting that an application of Theorem \ref{thm:main1-intro} similar to the one above yields the following regularity result specific to the characteristic-$p$ setting.

\begin{theorem}\label{thm:reglem-gen}
Let $k\in\mb{N}$ and let $b:\mb{R}_{>0}\times\mb{N}\to\mb{R}_{>0}$ be a  function decreasing sufficiently fast in the second variable. For every $\epsilon>0$ there exists $N=N(\epsilon,b)>0$ such that the following holds. For every function $f:\mb{F}_p^n\to \mb{C}$ with $|f|\leq 1$, there is a decomposition $f=f_s+f_e+f_r$ and number $m\leq N$ such that the following properties hold:
\begin{enumerate}[leftmargin=0.9cm]
\item $f_s$ is a $b(\epsilon,m)$-balanced $p$-homogeneous nilspace polynomial of degree $k$, $|f_s| \leq 1$, $\textup{Comp}(f_s)\leq m$, 
\item $\|f_e\|_{L^1}\leq\epsilon$,
\item $\|f_r\|_{U^{k+1}}\leq b(\epsilon,m)$, $|f_r|\leq 1$ and $\max\{ |\langle f_r,f_s\rangle|,\,|\langle f_r,f_e\rangle|\}\leq b(\epsilon,m)$.
\end{enumerate}
\end{theorem}

\noindent This follows from the general regularity result \cite[Theorem 1.5]{CSinverse}, by adding the assumption that for every $\epsilon>0$ and $m\in \mb{N}$ we have $b(\epsilon,m)\leq \min \{ b'_{\nss\sbr{i},p}: i\leq m\}$, where $b'_{\nss\sbr{i},p}$ is the constant given by Theorem \ref{thm:main1-intro}. Then, again thanks to the latter theorem, we can conclude that $f_s$ is $p$-homogeneous.

\appendix
\section{Auxiliary results on nilspaces}\label{app:res-nil}

\noindent In this first appendix we collect several results from general nilspace theory used in the paper. Most of these results are new and seem of independent interest as additional tools to work with nilspaces.

Let us first fix some terminology and notation. By a \emph{box} (or hyperrectangle) in $\mb{Z}^m$, for $m\in \mb{N}$, we mean as usual a Cartesian product of $m$ intervals in $\mb{Z}$. Given a base-point $a=(a\sbr{1},\ldots,a\sbr{m})\in \mb{Z}^m$ and a vector $\ell=(\ell\sbr{1},\ldots,\ell\sbr{m})\in \mb{Z}_{\geq 0}^m$, we denote the corresponding box $\prod_{i=1}^m [a\sbr{i},a\sbr{i}+\ell\sbr{i}]\subset \mb{Z}^m$ by $B_{a,\ell}$.

For $m\ge 1$ and $n\ge 0$, we shall work with cubes in $\cu^n(\mc{D}_1(\mb{Z}^m))$ whose images are contained in a given large box. It will then be useful to associate with each box in $\mb{Z}^m$ a certain cube (in the nilspace sense) on $\mc{D}_1(\mb{Z}^m)$ which covers the entire box, which we shall call the associated \emph{maximal cube}. For example, given a box $[a\sbr{1},a\sbr{1}+\ell\sbr{1}]\times [a\sbr{2},a\sbr{2}+\ell\sbr{2}]$ in $\mb{Z}^2$, the corresponding maximal cube is the $(\ell\sbr{1}+\ell\sbr{2})$-dimensional cube on $\mc{D}_1(\mb{Z}^2)$ that maps $v\in\db{\ell\sbr{1}+\ell\sbr{2}}$ to $a + \big(v\sbr{1}+\cdots + v\sbr{\ell\sbr{1}}, 0\big) + \big(0, v\sbr{\ell\sbr{1}+1}+\cdots + v\sbr{\ell\sbr{1}+\ell\sbr{2}}\big)$. Recall the notation $e_i$ for the elements of the standard basis of $\mb{Z}^m$, and the notation $|\ell|$ for the \emph{height} $\ell\sbr{1}+\cdots+\ell\sbr{m}$ of any $\ell\in \mb{Z}_{\geq 0}^m$.
\begin{defn}[Maximal cube associated with a box]\label{def:max-cube}
Let $m\in \mb{N}$, let $a=(a\sbr{i})_{i\in [m]}\in \mb{Z}^m$ and $\ell \in \mb{Z}_{\geq 0}^m$. The \emph{maximal cube} associated with the box $B_{a,\ell}$ is the cube $\q_{a,\ell} \in \cu^{|\ell|}(\mc{D}_1(\mb{Z}^m))$ defined as follows:
\[
\forall\,v\in \db{\,|\ell|\,},\quad \q_{a,\ell} (v) := a+ \sum_{j\in [m]} \Big(v(1+\sum_{i=1}^{j-1}\ell\sbr{i})+v(2+\sum_{i=1}^{j-1}\ell\sbr{i})+\cdots+v(\sum_{i=1}^j\ell\sbr{i}))\Big)\, e_j.
\]
\end{defn}
\noindent Maximal cubes will help us to understand when a morphism defined on a box in $\mb{Z}^m$ can be extended to a morphism on all of $\mc{D}_1(\mb{Z}^m)$. To this end we introduce the following definition.
\begin{defn}
Let $m\in \mb{N}$, let $a\in \mb{Z}^m$ and $\ell\in\mb{Z}_{\ge 0}^m$, and let $\ns$ be a nilspace. Then $\hom_{a,\ell}(\ns):=\{f:B_{a,\ell}\to \ns:f\co \q_{a,\ell}\in \cu^{|\ell|}(\ns)\}$.\end{defn} 
\begin{remark}\label{rem:abuse-not}
In the sequel, if we have a function $f:S\to\ns$ for some $S\subset \mb{Z}^m$ and there exists $a\in \mb{Z}^m$ and $\ell\in \mb{Z}_{\ge 0}^m$ such that $B_{a,\ell}\subset S$, we may abuse the notation by writing $f\in \hom_{a,\ell}(\ns)$, by which we mean that $f|_{B_{a,\ell}}\in \hom_{a,\ell}(\ns)$.
\end{remark}
\noindent To treat the above-mentioned extension problem, we begin with the following observation.
\begin{lemma}\label{lem:app}
Let $\ns$ be a nilspace, let $m \in\mb{N}$, and let $B_{a,\ell}$ be a box in $\mb{Z}^m$. Suppose that $f\in \hom_{a,\ell}(\ns)$. Then for any $n\ge 0$ and any $g\in \cu^n(\mc{D}_1(\mb{Z}^m))$ such that $\tIm(g) \subset B_{a,\ell}$, we have $f \co g \in \cu^n(\ns)$.
\end{lemma}
\begin{proof}
Let $x,y_1,\ldots,y_n\in \mb{Z}^m$ be the components of $g$, thus $g(v)=x+v\sbr{1}\, y_1+\cdots+v\sbr{n}\,y_n$. It suffices to prove that $g = \q_{a,\ell}\co h$ for some discrete-cube morphism $h:\db{n}\to \db{|\ell|}$. We shall explain in detail how the first $\ell\sbr{1}$ coordinates of $h$ can be defined in order to satisfy this last equality (the argument is the same for each interval $[1+\sum_{i=0}^{j-1}\ell\sbr{i},\sum_{i=0}^j\ell\sbr{i}]$ of coordinates of $h$, which will correspond to the $j$-th coordinate of $g$). 

The first coordinate of $g$ equals $x\sbr{1}+y_1\sbr{1}v\sbr{1}+\cdots+y_n\sbr{1}v\sbr{n}$. Note that $x\sbr{1}\in [a\sbr{1},a\sbr{1}+\ell\sbr{1}]$ and that $\sum_{i=1}^n|y_i\sbr{1}|\le \ell\sbr{1}$ (as otherwise it is easy to check that the image of $g$ would not lie in $B_{a,\ell}$). Now, for simplicity of the notation, assume that the coordinates $y_1\sbr{1},\ldots,y_{t_1}\sbr{1}$ are all strictly positive, the coordinates $y_{t_1+1}\sbr{1},\ldots,y_{t_1+t_2}\sbr{1}$ are all strictly negative and $y_{t_1+t_2+1}\sbr{1},\ldots,y_n\sbr{1}$ are all zero (the general argument is similar, modulo taking care of the actual positions of the positive, negative, and zero coordinates, but tracking this only adds difficulty to the reading of the proof). 

We start by defining the first $\sum_{i=1}^n|y_i\sbr{1}|\le \ell\sbr{1}$ coordinates of the discrete-cube morphism $h$. We take these to be
\[
(\underbrace{v_1,\ldots,v_1}_{y_1\sbr{1} \text{ times}},\underbrace{v_2,\ldots,v_2}_{y_2\sbr{1} \text{ times}},\ldots,\underbrace{v_{t_1},\ldots,v_{t_1}}_{y_{t_1}\sbr{1} \text{ times}},\underbrace{1-v_{t_1+1},\ldots,1-v_{t_1+1}}_{|y_{t_1}\sbr{1}| \text{ times}},\ldots,\underbrace{1-v_{t_1+t_2},\ldots,1-v_{t_1+t_2}}_{|y_{t_1+t_2}\sbr{1}| \text{ times}}).
\]
Now we just have to define the next $\ell\sbr{1}-\sum_{i=1}^n|y_i\sbr{1}|$ coordinates of $h$ (and thus we would have defined in total the first $\ell\sbr{1}$ coordinates of $h$). Note that $x\sbr{1}\ge a\sbr{1}+\sum_{i=t_1+1}^{t_1+t_2}|y_i\sbr{1}|,$ as otherwise it is again easy to check that the image of $g$ would not lie in $B_{a,\ell}$. Similarly we have that $x\sbr{1}\le a\sbr{1}+\ell\sbr{1}-\sum_{i=1}^{t_1}|y_i\sbr{1}|$. Hence, we define the next $\ell\sbr{1}-\sum_{i=1}^n|y_i\sbr{1}|$ coordinates of $h$ as
\[
(1^{x\sbr{1}-a\sbr{1}-\sum_{i=t_1+1}^{t_1+t_2}|y_i\sbr{1}|},0^{\ell\sbr{1}-\sum_{i=1}^{t_1}|y_i\sbr{1}|-x\sbr{1}+a\sbr{1}}).
\]
It is now seen by straightforward summation that the first coordinate of $g$ is thus equal to the first coordinate of $\q_{a,\ell}\co h$. The result follows.
\end{proof}
\begin{lemma}[Corners of a box]\label{lem:cor-of-box} Let $\ns$ be a nilspace, let $B_{a,\ell}$ be a box in $\mb{Z}^m$, and let $f:B_{a,\ell}\setminus\{a+\ell\}\to \ns$ be a map such that for every $j\in [m]$ with $\ell\sbr{j}>0$ we have $f\in\hom_{a,\ell-e_j}(\ns)$ \textup{(}recall Remark \ref{rem:abuse-not} here\textup{)}. Then $f\co \q_{a,\ell}\in \cor^{|\ell|}(\ns)$.
\end{lemma}
Here ``$\cor^n(\ns)$" denotes the space of $n$-corners on $\ns$ (see \cite[Lemma 2.1.12]{Cand:Notes2}).
\begin{proof} We show that all lower faces of $f\co \q_{a,\ell}:\db{|\ell|}\setminus\{1^{|\ell|}\}\to \ns$ are cubes. For any $u\in [\,|\ell|\,]$, let $\phi_u:\db{\,|\ell|-1\,}\to \db{\,|\ell|\,}$ be the map $(v_1,\ldots,v_{|\ell|-1})\mapsto (v_1,\ldots,v_{u-1},0,v_u,\ldots,v_{|\ell|-1})$. Thus $\q_{a,\ell}\co\phi_u = \q_{a,\ell-e_{j(u)}}$ where $j(u)\in [m]$ is such that $u\in [1+\sum_{i=1}^{j(u)-1}\ell(i),\sum_{i=1}^{j(u)}\ell(i)]$. Thus $f\co \q_{a,\ell}\co\phi_u\in \cu^{|\ell|-1}(\ns)$.\end{proof}
We shall now derive some useful corollaries.
\begin{corollary}\label{cor:liftthrufibgen}
Let $\ns$ and $\nss$ be nilspaces, and let $\psi:\ns \to \nss$ be a fibration. Let $g\in \hom(\mc{D}_1(\mb{Z}^m),\nss)$, let $B_{a,\ell}$ be a box in $\mb{Z}^m$, and let $f\in \hom_{a,\ell}(\ns)$ satisfy $\psi\co f = g|_{B_{a,\ell}}$. Then there is $g'\in \hom(\mc{D}_1(\mb{Z}^m),\ns)$ such that $g'|_{B_{a,\ell}}=f$ and $\psi \co g' = g$.
\end{corollary}
\begin{proof}
Recall that by definition of fibrations, given any corner $\q'\in\cor^m(\ns)$ and any cube $q\in \cu^m(\nss)$ such that $\psi \co \q' = q|_{\db{m}\setminus 1^m}$, there exists $\q\in \cu^m(\ns)$ such that $\psi \co \q = q$.

The idea of the proof is to extend the definition of $f$ point by point in an inductive process, defining values of $f$ at new points in $\mb{Z}^m$ of the form $a+(\ell\sbr{1}+1,t_2,\ldots,t_m)$ for varying $t_j$, in order to  extend $f$ eventually to the whole greater box $B_{a,\ell+e_1}$ (thus we have increased the first coordinate of $\ell$ by 1) while ensuring that $f\in\hom_{a,\ell+e_1}(\ns)$. For the induction, we can use the lexicographic order $\prec$ on $\{\ell\sbr{1}+1\}\times\prod_{j=2}^m [0,\ell\sbr{j}]$ (noting that if $v\sbr{i}\leq w\sbr{i}$ for all $i$ and $v\neq w$ then $v\prec w$). We illustrate the process in the case $m=2$ and $B_{(0,0),(1,2)}$. The points where $f$ is defined initially are
\begin{center}
\includegraphics[width=30mm]{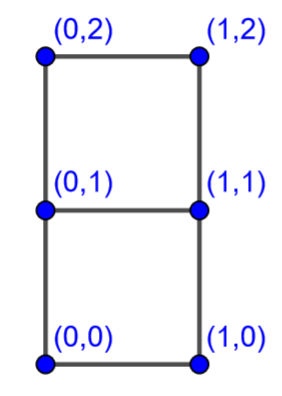}
\end{center}
and we will assign new values to the points $(2,0),(2,1)$ and $(2,2)$ (in that order).

The base case for the induction corresponds to $t_2=\cdots=t_m=0$. First we want to prove that $f\co\q_{a,(\ell\sbr{1}+1,0,\ldots,0)}|_{v\not= 1^{1+\ell(1)}}$ is in $\cor^{\ell(1)+1}(\ns)$. By Lemma \ref{lem:cor-of-box} it is enough to check that $f\co \q_{a,(\ell\sbr{1},0,\ldots,0)}\in \cu^{\ell\sbr{1}}(\ns)$. As $\q_{a,(\ell\sbr{1},0,\ldots,0)}\in \cu^{\ell\sbr{1}}(\mc{D}_1(\mb{Z}^m))$ and its image lies in $B_{a,\ell}$, by Lemma \ref{lem:app} the result follows in this case. Furthermore, by assumption we have $\psi \co f \co \q_{a,(\ell\sbr{1}+1,0,\ldots,0)} (v)= g \co \q_{a,(\ell\sbr{1}+1,0,\ldots,0)}(v)$ for all $v \neq 1^{\ell\sbr{1}+1}$, and $g \co \q_{a,(\ell\sbr{1}+1,0,\ldots,0)}\in \cu^{\ell\sbr{1}+1}(\ns)$. As $\psi$ is a fibration, we can complete the corner (i.e.\ assign a value to $f(a+(\ell\sbr{1}+1)e_1)$ making $f \co \q_{a,(\ell\sbr{1}+1)e_1}$ a cube) in such a way that $(\psi \co f)(a+(\ell\sbr{1}+1)e_1) = g(a+(\ell\sbr{1}+1)e_1)$. In our example, we would thus assign a value to $(2,0)$ and now the points where $f$ is defined are the following:
\begin{center}
\includegraphics[width=40mm]{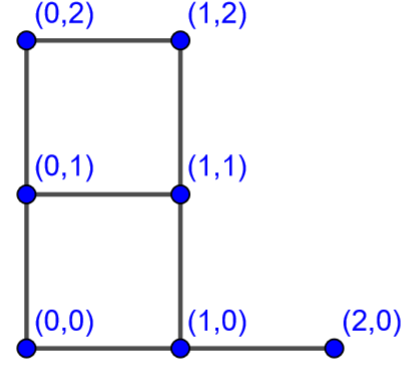}
\end{center}
For the general case, suppose that we want to assign the value of $f(a +(\ell\sbr{1}+1,t_2,\ldots,t_m))$. By induction, for all $(\ell\sbr{1}+1,t'_2,\ldots,t'_m)\prec (\ell\sbr{1}+1,t_2,\ldots,t_m)$ we have assigned a value to $f(a+(\ell\sbr{1}+1,t'_2,\ldots,t'_m))$ so that $f \co \q_{a,(\ell\sbr{1}+1,t'_2,\ldots,t'_m)}\in \cu^{\ell\sbr{1}+1+\sum_{i=2}^m t'_i}(\ns)$ and $\psi \co f \co \q_{a,(\ell\sbr{1}+1,t'_2,\ldots,t'_m)} = g \co \q_{a,(\ell\sbr{1}+1,t'_2,\ldots,t'_m)}$. Now we claim that $f \co \q_{a,(\ell\sbr{1}+1,t_2,\ldots,t_m)}(v)$ for $v\in\db{\ell\sbr{1}+1+\sum_{i=2}^m t_i} \setminus\{1^{\ell\sbr{1}+1+\sum_{i=2}^m t_i}\}$ in an element of $\cor^{\ell\sbr{1}+1+\sum_{i=2}^m t_i}(\ns)$. In order to prove this, we will rely again on Lemma \ref{lem:cor-of-box}. To apply it we need to check two different cases. First, we need that $f\co \q_{a,(\ell\sbr{1},t_2,\ldots,t_m)}\in \cu^{\ell\sbr{1}+\sum_{i=2}^mt_i}(\ns)$ (corresponding to subtracting $e_1$). As $\q_{a,(\ell\sbr{1},t_2,\ldots,t_m)}\in \cu^{\ell\sbr{1}+\sum_{i=2}^mt_i}(\ns)$ and its image is contained in $B_{a,\ell}$, by Lemma \ref{lem:app} the result follows in this case. Now let $j\ge 2$. We need to prove that $f\co \q_{a,(\ell\sbr{1}+1,t_2,\ldots,t_m)-e_j}\in \cu^{\ell\sbr{1}+\sum_{i=2}^m t_i}(\ns)$. But this case follows by induction hypothesis as $(\ell\sbr{1}+1,t_2,\ldots,t_m)-e_j\prec (\ell\sbr{1}+1,t_2,\ldots,t_m)$. 

In our example let us assume that we are trying to assign a value to $f(2,1)$ (in red in the diagram). The previous paragraph says that first we have to check that $f\co \q_{(0,0),(2,0)}$ and $f\co \q_{(0,0),(1,1)}$ are in $\cu^2(\ns)$. The images of $\q_{(0,0),(2,0)}$ and $\q_{(0,0),(1,1)}$ are represented in purple and green respectively.
\begin{center}
\includegraphics[width=80mm]{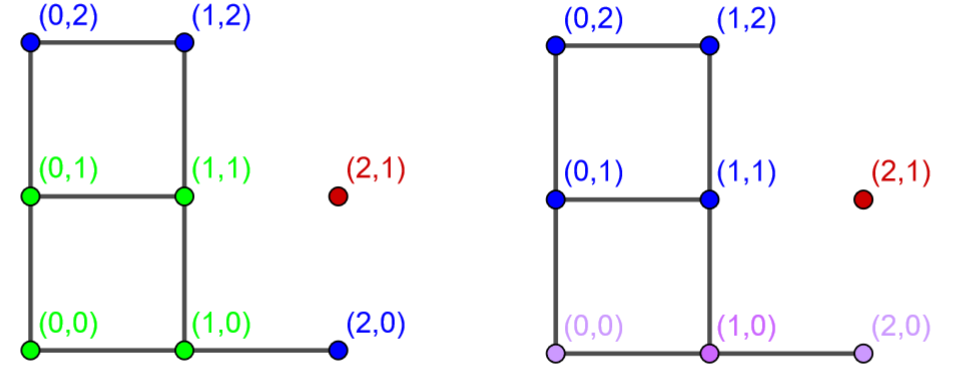}
\end{center}
From the diagram we see that for the green cube we have to use our initial assumption and an application of Lemma \ref{lem:app} and for the purple one the induction hypothesis.
We also have $\psi\co f\co \q_{a,(\ell\sbr{1}+1,t_2,\ldots,t_m)}(v) = g \co \q_{a,(\ell\sbr{1}+1,t_2,\ldots,t_m)}(v)$ for all $v\not=1^{\ell\sbr{1}+1+\sum_{j=2}^m t_j}$ and $g \co \q_{a,(\ell\sbr{1}+1,t_2,\ldots,t_m)} \in \cu^{\ell\sbr{1}+1+\sum_{j=2}^m t_j}(\nss)$ (by construction). Thus, using that $\psi$ is a fibration, we can complete the corner $f\co \q_{a,(\ell\sbr{1}+1,t_2,\ldots,t_m)}$ in such a way that $f\co \q_{a,(\ell\sbr{1}+1,t_2,\ldots,t_m)}\in \cu^{\ell\sbr{1}+1+\sum_{j=2}^m t_j}(\ns)$ and $\psi\co f\co \q_{a,(\ell\sbr{1}+1,t_2,\ldots,t_m)}=g\co \q_{a,(\ell\sbr{1}+1,t_2,\ldots,t_m)}$. The value at the top-vertex of this completion is the value that we assign to $f(a +(\ell\sbr{1}+1,t_2,\ldots,t_m))$.

At the end of this process, we obtain $f:B_{a,\ell+e_1}\to \ns$ such that $f\co \q_{(a,\ell+e_1)}\in \cu^{1+|\ell|}(\ns)$. It is fairly easy to see now that we can repeat this process in every direction (i.e.\ thus adding $e_j$ to $\ell$, for any $j\in [m]$), and thus extend $f$ to a map $\tilde f: B_{a,(L,\ldots,L)}\to \ns$ such that $\tilde f\co \q_{a,(L,\ldots,L)}\in \cu^{mL}(\ns)$, for any $L\in\mb{N}$. Moreover, if we reflect $\tilde f$ defining $f':B_{(0,a(2),\ldots,a(m)),(L,\ldots,L)}\to \ns$ as $f'(v)=\tilde f(a\sbr{1}+L-v\sbr{1},v\sbr{2},\ldots,v\sbr{m})$, then extend this by $e_1$ as above, and then reflect again, we obtain an extension of $\tilde f$ to $B_{a-e_1,(L,L,\ldots,L)}$. Arguing similarly and iteratively in each direction, we see that $f$ can be extended to any cube of size $[-L,L]^m$ for any sufficiently large $L$.\footnote{To be precise, we need $L\ge \max_{i\in\{1,\ldots,m\}}(|a(i)|+|\ell(i)|)$ as the result consists in enlarging the original box $B_{a,\ell}$ and therefore $L$ has to be large enough so that $B_{a,\ell}\subset [-L,L]^m$.} Hence, we can define inductively the extension of $f$ to all $\mb{Z}^m$. This extension is our morphism $g'\in\hom(\mc{D}_1(\mb{Z}^m),\ns)$. To check that this is indeed a morphism, we just have to note that given any cube $q\in \cu^n(\mc{D}_1(\mb{Z}^m))$, we have $\tIm(q)\subset [-L,L]^m$ for some $L$ large enough, so the result follows using Lemma \ref{lem:app}.
\end{proof}
The following consequence is the special case of Corollary \ref{cor:liftthrufibgen} with $B_{a,\ell}=\db{m}$.
\begin{corollary}\label{cor:liftthrufib}
Let $\ns$ and $\nss$ be nilspaces and let $\psi:\ns \to \nss$ be a fibration. Let $g\in \hom(\mc{D}_1(\mb{Z}^m),\nss)$, let $\q\in \cu^m(\ns)$ for some $m\ge 0$, and suppose that $\psi \co \q = g|_{\db{m}}$. Then there exists $g'\in \hom(\mc{D}_1(\mb{Z}^m),\ns)$ such that $g'|_{\db{m}}=\q$ and $\psi \co g' = g$.
\end{corollary}
\begin{lemma}[Corner completion of a box]\label{lem:cor-compl-box}
Let $\ns$ be a nilspace, let $B_{a,\ell}$ be a box in $\mb{Z}^m$, and let $f:B_{a,\ell}\setminus\{a+\ell\}\to \ns$ be a map such that for every $j\in [m]$ with $\ell(j)>0$ we have $f\in \hom_{a,\ell-e_j}(\ns)$. Then there exists an element $x\in \ns$ such that, extending $f$ to all of $B_{a,\ell}$ by setting $f(a+\ell)=x$, we have $f\in\hom_{a,\ell}(\ns)$.
\end{lemma}
\begin{proof}
By Lemma \ref{lem:cor-of-box} we have that $f\co \q_{a,\ell}|_{v\neq 1^{|\ell|}}$ is a corner in $\cor^{|\ell|}(\ns)$. Then by the completion axiom for nilspaces there exists a completion of $f\co \q_{a,\ell}$, and then letting $f(a+\ell)$ be the top-vertex value of this completion, the result follows.
\end{proof}

Next we prove some useful results concerning coset nilspaces.

\begin{lemma}\label{lem:correction-cubes}
Let $(G,G_{\bullet})$ be a filtered group, let $\Gamma$ be a subgroup of $G$, and let $\ns$ denote the associated coset nilspace. For any $n\ge 0$, let  $\q,\q'\in \cu^n(\ns)$ satisfy $\q(v)=\q'(v)$ for all $v\neq 1^n$. Then there exists $g\in G_n$ such that $g\q(1^n)=\q'(1^n)$.
\end{lemma}

\begin{proof}
By definition of cubes on $\ns$, there exists $q,q'\in\cu^n(G_\bullet)$ such that $\q=q\Gamma$ and $\q'=q'\Gamma$. Then, considering $q$ and $q'$ as functions on $\db{n}\setminus \{1^n\}$, we have that $q^{-1}q'$ is an $n$-corner on the group nilspace $(\Gamma,\Gamma_{\bullet})$ where $\Gamma_i:=\Gamma \cap G_i$ for all $i\ge 0$. Let $t\in \cu^n(\Gamma_\bullet)$ be a completion of $q^{-1}q'$. Then (since $t$ is $\Gamma$-valued) we have $qt\Gamma = q\Gamma =\q$. As $(qt)(v)=q'(v)$ for all $v\neq 1^n$, we know that $q'(qt)^{-1} \in \cu^n(G_\bullet)$ and that $(q'(qt)^{-1})(v)=1$ for all $v\neq 1^n$. Hence, by basic properties of Host--Kra cubes (see \cite[Lemma 2.2.26]{Cand:Notes1}) we have $(q'(qt)^{-1})(1^n)\in G_n$. Setting $g:=(q'(qt)^{-1})(1^n)$, the result follows.
\end{proof}

We shall also use the following definitions.

\begin{defn}[Simplicial set]\label{def:simplicial-set} Let $n\in \mb{N}$ be an integer. We say that a set $S\subset \mb{Z}^n_{\ge 0}$ is a \emph{simplicial set} if it has the following property: for any $v\in S$ and any $w\in \mb{Z}^n_{\ge 0}$, if $w\sbr{i}\le v\sbr{i}$ for all $i\in [n]$, then $w\in S$.
\end{defn}

\begin{defn}[Simplicial corner]\label{def:simplicial-corner} Let $n\in \mb{N}$ and $S\subset \mb{Z}^n_{\ge 0}$ be a simplicial set. We say that a vertex $v\notin S$ is a \emph{corner-vertex} for $S$ if for all $w\in \mb{Z}^n_{\ge 0}\setminus \{v\}$ such that $w\sbr{i}\le v\sbr{i}$ for all $i\in [n]$, we have $w\in S$.
\end{defn}

\noindent The following figure illustrates these two definitions, indicating in  red an example of a simplicial set in $\mb{Z}_{\ge 0}^2$, and in blue its corner vertices.
\begin{center}
\includegraphics[width=100mm]{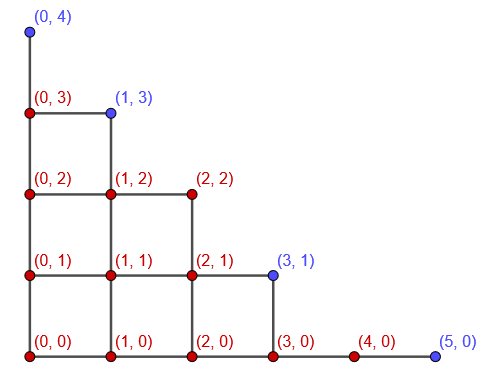}
\end{center}
\noindent With this, we can now prove the following result.
\begin{lemma}[Taylor coefficients]\label{lem:taylor} 
Let $n\ge 0$ be an integer, $S\subset \mb{Z}^n_{\ge 0}$ a simplicial set, and let $w\in \mb{Z}_{\ge 0}^n$ be a corner of $S$. Let $(G,G_{\bullet})$ be a filtered group, let $\Gamma$ be a subgroup of $G$, and let $\ns$ denote the corresponding coset nilspace. Let $f,f'\in \hom(\mc{D}_1(\mb{Z}^n),\ns)$ satisfy $f(v)=f'(v)$ for all $v\in S$. Then there exists $g\in G_{|w|}$ such that $g^{\binom{v}{w}} f(v) = f'(v)$ for all $v\in S\cup \{w\}$.
\end{lemma}

\begin{proof}
Let $\q_{0^n,w}\in \cu^{|w|}(\mc{D}_1(\mb{Z}^n))$ be the maximal cube associated with the box $B_{0^n,w}$. Applying Lemma \ref{lem:correction-cubes} to $f\co \q_{0^n,w}$ and $f' \co \q_{0^n,w}$, we obtain the value of $g\in G_{|\underline{w}|}$. It is readily seen that $v\mapsto g^{\binom{v}{w}} f(v)$ is in $\hom(\mc{D}_1(\mb{Z}^n),\ns)$ and satisfies the desired properties.
\end{proof}

\begin{lemma}[Completion of a simplicial set]\label{lem:completion-simplicial}
Let $\ns$ be a nilspace and let $S$ be a simplicial set included in a box $B_{0^n,\ell}\subset \mb{Z}^n$. Let $f:S\to \ns$ be a function such that for any box $B_{0^n,d}\subset S$ we have $f\in\hom_{0^n,d}(\ns)$. Then there exists $g\in\hom_{0^n,\ell}( \ns)$ such that $g|_S=f$.\end{lemma}

\begin{proof}
This is a straightforward generalization of \cite[Lemma 3.1.5]{Cand:Notes1} using Lemma \ref{lem:cor-compl-box}.
\end{proof}

Recall by Definition \ref{def:hom-n-p-set} that $\hom_p^m(\ns):=\{f:[0,p-1]^m\to \ns: f\co \q_{0^m,(p-1)^m}\in \cu^{m(p-1)}(\ns)\}$. Next we prove Lemma \ref{lem:fordimreduc-2}, which we recall here for convenience.
\begin{lemma}\label{lem:fordimreduc-2-app}
Let $\ns$ be a $k$-step nilspace and $n\ge k+1$. Let $f:[0,p-1]^n\to \ns$ satisfy $f\co \phi\in \hom_p^{k+1}(\ns)$ for every $p$-face-map $\phi:[0,p-1]^{k+1}\to [0,p-1]^n$. Then $f\in \hom_p^n(\ns)$.
\end{lemma}
\noindent Recall the notation the notation $\cor^{k+1}(\ns)$ for the set of $(k+1)$-corners on $\ns$. For $v\in \mb{Z}^n$, let us recall also the notation $\supp(v)$ for the set of indices $i\in [n]$ such that $v\sbr{i}\neq 0$. Finally, let us call a function $f':\{v\in [0,p-1]^n:|\supp(v)|\leq n-1\}\to \ns$ a \emph{$p$-corner of dimension $n$ on} $\ns$ if for every $p$-face-map $\phi:[0,p-1]^{n-1}\to [0,p-1]^n$ that fixes some coordinate equal to 0, we have $f\co \phi\in \hom_p^{n-1}(\ns)$. 
\begin{proof}
Let $S:=\{v\in [0,p-1]^{k+1}:|\supp(v)|\leq k\}=\bigcup_{i=1}^{k+1} \{v\in [0,p-1]^{k+1}:v\sbr{i}=0\}$. We first claim that every $p$-corner $f'$ of dimension $k+1$ on $\ns$ has a unique completion, that is, there is a unique $f\in \hom_p^{k+1}(\ns)$ with $f(v)=f'(v)$ for every $v\in S$. To see the existence of $f$, note that the set $S$ (on which $f'$ is defined) is simplicial and $f':S\to \ns$ satisfies the assumptions in Lemma \ref{lem:completion-simplicial}, so the existence of $f$ follows from that lemma. To see the uniqueness, consider first the cube $\phi\in \cu^{k+1}(\mc{D}_1(\mb{Z}^{k+1}))$ defined by $\phi(v)=v$ (this just embeds $\db{k+1}$ in $[0,p-1]^{k+1}$). Then $f'\co \phi|_{\db{k+1}\setminus\{1^{k+1}\}}\in \cor^{k+1}(\ns)$ and therefore it has a unique completion. It follows that the value $f(1^{k+1})$ is uniquely determined by $f'$. Now we argue similarly for every remaining $v\in[0,p-1]^{k+1}$, showing inductively that the determination of the values of $f$ by $f'$ propagates to all of $[0,p-1]^{k+1}$. Suppose that we have a simplicial set $S'\subset [0,p-1]^{k+1}$ and a simplicial corner $w$ of $S'$ (see Definitions \ref{def:simplicial-set} and \ref{def:simplicial-corner}). Furthermore, suppose that $S'\supset S$ and assume inductively that for every $v\in S'$, the value $f(v)$ is uniquely determined by $f'$. We are going to prove that $f(w)$ is also uniquely determined by $f'$. Consider the maximal cube $\q_{(0^{k+1},w)}\in \cu^{|w|}(\mc{D}_1(\mb{Z}^{k+1}))$. Then since $f\in \hom_p^{k+1}(\ns)$, we have $f\co \q_{(0^{k+1},w)}\in \cu^{|w|}(\ns)$. As $w\sbr{i}\ge 1$ for all $i\in [k+1]$, we have $|w|\ge k+1$, so by uniqueness of completion $f\co \q_{(0^{k+1},w)}(1^{|w|}) = f(w)$ is uniquely determined by the other values of the cube $f\co \q_{(0^{k+1},w)}$. But since these other values correspond to points of $S'$, they are determined by $f'$, whence $f(w)$ is also uniquely determined by $f'$. This proves our claim.

Now, to prove the lemma, let $f:[0,p-1]^n\to \ns $ be a function such that for all $p$-face-maps $\phi:[0,p-1]^{k+1}\to [0,p-1]^n$ we have $f\co \phi\in \hom_p^{k+1}(\ns)$. We have to prove that $f\in \hom_p^n(\ns)$. Consider the map $g'$ defined by $g'(v):=f(v)$ for all $v\in [0,p-1]^n$ such that $|\supp(v)|\leq k$. By Lemma \ref{lem:completion-simplicial}, we can complete $g'$ to an element $g\in \hom_p^n(\ns)$. We claim that $f=g$. To prove this, we can argue by contradiction using the claim in the previous paragraph. Indeed, suppose that for some $w\in [0,p-1]^n$ we had $f(w)\neq g(w)$, and let $|\supp(w)|$ be minimal with this property. By our initial assumption on $g'$, we have $s:=|\supp(w)|\geq k+1$. Without loss of generality, suppose that $w=(w\sbr{1},\ldots,w\sbr{s},0,\ldots,0)$. Consider the $p$-face-map $\phi:[0,p-1]^{k+1}\to [0,p-1]^n$, $(v\sbr{1},\ldots,v\sbr{k+1})\mapsto (v\sbr{1},\ldots,v\sbr{k+1},w\sbr{k+2},\ldots,w\sbr{s},0,\ldots,0)$. Then, if $v\in \cup_{i=1}^{k+1}\{v\in [0,p-1]^{k+1}:v\sbr{i}=0\}$, we have $|\supp(\phi(v))|\leq k$, so by assumption  $f\co\phi(v)=g\co \phi(v)$. But now both $f\co \phi$ and $g\co \phi$ are elements in $\hom_p^{k+1}(\ns)$, so by the previous paragraph we have $f\co\phi(v) = g\co \phi(v)$ for all $v\in [0,p-1]^{k+1}$, so  $f(w)=g(w)$, a contradiction.
\end{proof}

Let us recall the following useful construction in nilspace theory.
\begin{defn}[Fiber-product of nilspaces]
Let $\ns_1,\ns_2,$ and $\ns_3$ be nilspaces and let $\psi_1:\ns_1\to \ns_3$ and $\psi_2:\ns_2\to \ns_3$ be fibrations. We define the fiber-product (or sub-direct product) $\ns_1 \times_{\ns_3} \ns_2$ to be the nilspace $\{(x_1,x_2)\in \ns_1\times \ns_2: \psi_1(x_1)=\psi_2(x_2)\}$ with cube sets $\cu^n(\ns_1 \times_{\ns_3}\ns_2):=\{\q_1\times \q_2 \in \cu^n(\ns_1)\times \cu^n(\ns_2): \psi_1 \co \q_1=\psi_2 \co \q_2\}$.
\end{defn}
\noindent To see that this defines indeed a nilspace see \cite[Lemma 4.2]{CGSS}. We leave it as an exercise for the reader to check that the projections $p_i:\ns_1 \times_{\ns_3}\ns_2\to \ns_i$ for $i=1,2$ are fibrations, and that if $\ns_1$ and $\ns_2$ are $k$-step, then so is $\ns_1 \times_{\ns_3}\ns_2$.

\begin{proposition}\label{prop:sub-prod-ext} Let $\ns$ and $Q$ be $k$-step nilspaces and let $\varphi:Q \to \ns_{k-1}$ be a fibration. Then $p_1:Q\times_{\ns_{k-1}} \ns\to Q$ is a degree-$k$ extension whose structure group is the $k$-th structure group of $\ns$. \end{proposition}

\begin{proof} Let us denote by $Z$ the $k$-th structure group of $\ns$. First, let us define the action of $Z$ on $\nss:=Q\times_{\ns_{k-1}} \ns$. Given $(q,x)\in Y$ and $z\in Z$, $(q,x)+z:=(q,x+z)$. To see that this is well defined, note that $\pi_{k-1}(x)=\pi_{k-1}(x+z)$ for all $x\in \ns$ and $z\in Z$. The action is free, because if $(q,x)=(q,x+z)$ then $x=x+z$ and this implies that $z=0$. The action is also transitive over the fibers of $p_1$. To see this, let $(q_1,x_1),(q_2,x_2)\in \nss$ be such that $p_1(q_1,x_1)=p_1(q_2,x_2)$. Then $q_1=q_2$ and $\pi_{k-1}(x_1)=\varphi(q_1)=\varphi(q_2)=\pi_{k-1}(x_2)$. Thus, there exists $z\in Z$ such that $x_1=x_2+z$, which implies that $(q_1,x_1)=(q_2,x_2)+z$. We also need to prove that $p_1:\cu^n(\nss)\to \cu^n(Q)$ is a surjection, but this follows from the fact that $p_1$ is a fibration. Finally, we have to check that given two cubes $(\q_1,d_1),(\q_2,d_2)\in \cu^n(\nss)$ such that $p_1 \co (\q_1,d_1) = p_1 \co (\q_2,d_2)$, there exists $f\in \cu^n(\mc{D}_k(Z))$ such that $(\q_1,d_1)=(\q_2,d_2)+f$. Proceeding as before, for all $v\in \db{n}_{\le k}:= \{v\in \db{n}:|v|\le k\}$ we have that there exists $f(v)\in Z$ such that $d_1(v)=d_2(v)+f(v)$. Now consider the (unique) extension of $f$ to an element of $\cu^n(\mc{D}_k(Z))$. Thus, $d_1(v) = (d_2+f)(v)$ for all $v\in \db{n}_{\le k}$. as $\ns$ is $k$-step, this implies that for all $v\in \db{n}$ we have $d_1(v) = (d_2+f)(v)$. Therefore $(\q_1,d_1) = (\q_2,d_2)+f$.
\end{proof}

\begin{proposition}\label{prop:ext-fib}
Let $q:\ns\to \nss$ be a degree-$k$ extension by an abelian group $Z$. Then $q$ is a fibration.
\end{proposition}

\begin{proof} Let $\q\in \cu^n(\nss)$ be a cube and $\q'\in \cor^n(\ns)$ be a corner such that $q\co \q' = \q$ for all $v\not=1^n$. By \cite[Definition 3.3.13, (i)]{Cand:Notes1}, let $\q^*\in \cu^n(\ns)$ be such that $q\co \q^* = \q$. Thus, $q\co \q'=q\co \q^*$ for all $v\not= 1^n$. As $\ns$ is a bundle over $\nss$, this means that $\q'-\q^*$ takes values in $Z$ and by \cite[Definition 3.3.13, (ii)]{Cand:Notes1} we know that $\q'-\q^*\in\cor^n(\mc{D}_k(Z))$. Let $d\in \cu^n(\mc{D}_k(Z))$ be a completion of that corner. Then $\q^*+d$ is a cube such that $\q^*+d=\q^*$ for all $v\not=1^n$ and $q\co (\q^*+d) = \q$. \end{proof}

\begin{proposition}\label{prop:proj-of-ext} Let $\ns$ and $\ns'$ be nilspaces such that $p:\ns\to \ns'$ is a degree-$t$ extension by an abelian group $Z$. Let $p_t:\ns_t\to \ns'_t$ denote the induced morphism\footnote{See \cite[Definition 3.3.1.(i) and Proposition 3.3.2.]{Cand:Notes1}} between the $t$-th factors. Then $p_t$ defines a degree-$t$ extension with structure group $Z$ and $\ns \cong \ns' \times_{\ns'_t} \ns_t$. \end{proposition}

\begin{proof} First, let us see that $p_t$ defines an abelian bundle with structure group $Z$. We define the action of $Z$ on $\ns_t$ by $\pi_t(x)+z:=\pi_t(x+z)$. Let us check that this is well-defined. If $\pi_t(x)=\pi_t(y)$, then by definition of $\pi_t$ (see \cite{Cand:Notes1}) there exists a cube $\q\in \cu^{t+1}(\ns)$ such that $\q(1^{t+1}) = y$ and $\q(v)=x$ for all $v\not=1^{t+1}$. Then, the cube $\q+z$ (adding $z$ to all values $\q(v)$) is an element of $\cu^{t+1}(\ns)$. Thus $\pi_t(x+z)=\pi_t(y+z)$, so the action is well-defined. To see that the action is free, note that if $\pi_t(x)=\pi_t(x+z)$ then there exists a $(t+1)$-dimensional cube in $\ns$ with value $x$ at all vertices except $1^{t+1}$, where the value will be $x+z$. Therefore, as the constant cube with value $x$ is in $\cu^{t+1}(\ns)$, the map $\db{t+1}\to Z$ with value 0 at all vertices except $1^{t+1}$ and value $z$ at $1^{t+1}$, would be a cube on $\mc{D}_t(Z)$. But this implies that $z=0$. To see that this action of $Z$ is transitive on the fibers of $p_t$, suppose that $p_t(\pi_t(x))=p_t(\pi_t(y))$. Then $\pi_t(p(x))=\pi_t(p(y))$ which means that if we let $\q'$ be the function such that $\q'(1^{t+1})=p(y)$ and $\q'(v)=p(x)$ for all $v\not=1^{t+1}$, then $\q'\in\cu^{t+1}(\ns')$. Now let $\q''\in\cor^{t+1}(\ns)$ be the corner such that $\q''(v)=x$ for all $v\in\db{t+1}\setminus \{1^{t+1}\}$. By Proposition \ref{prop:ext-fib}, $p$ is a fibration, and as $p\co \q'' = \q'$ for all $v\not=1^{t+1}$, there exists a completion of $\q''$ such that (abusing a little the notation) $p(\q''(1^{t+1}))=p(y)$. Thus, as $p$ is a degree-$t$ extension by $Z$, this means that $\q''(1^{t+1})=y+z$. Thus, $\pi_t(x)=\pi_t(y+z)=\pi_t(y)+z$.

Now, let us check that $p_t$ satisfies the conditions of \cite[Definition 3.3.13]{Cand:Notes1}. The only non-trivial part is to prove that for any $\pi_t \co \q_1\in \cu^n(\ns_t)$,
\[
\{\pi_t \co \q_2\in\cu^n(\ns_t): p_t \co \pi_t \co \q_1 = p_t \co \pi_t \co \q_2 \} = \{(\pi_t \co \q_1)+d : d\in \cu^n(\mc{D}_t(Z))\}.
\]
Let $\pi_t \co \q_2$ be a cube in the set on the left side above (where, as usual, we assume that $\q_1,\q_2\in \cu^n(\ns)$). Then $\pi_t \co p\co \q_1 = \pi_t \co p\co \q_2$. Now fix any $v\in \db{n}_{\le t}$. Since $\pi_t \co p\co \q_1(v) = \pi_t \co p\co \q_2(v)$, there exists a cube $\q\in \cu^{t+1}(\ns')$ such that $\q(w)=p(\q_1(v))$ for all $w\not=1^{t+1}$ and $\q(1^{t+1}) = p(\q_2(v))$. By an argument similar as before (using that $p$ is a fibration), we conclude that there exists $z(v)\in Z$ such that $\pi_t(\q_1(v)) = \pi_t(\q_2(v))+z(v)$. Now let $d\in \cu^n(\mc{D}_t(Z))$ be the (unique) cube such that $d(v)=z(v)$ for all $v\in \db{n}_{\le t}$. Then, we have that both $\pi_t \co \q_1$ and $\pi_t \co \q_2+d$ are cubes in $\cu^n(\ns_t)$ and that they coincide in the set $\db{n}_{\le t}$. As $\ns_t$ is $t$-step, this implies that $\pi_t \co \q_1= \pi_t \co \q_2+d$. This proves that $\pi_t \co \q_2$ is in the set on the right above. We leave the other inclusion for the reader.

To complete the proof, let us see that $\ns$ is (nilspace) isomorphic to $\ns' \times_{\ns'_t} \ns_t$. Let $\varphi:\ns \to \ns' \times_{\ns'_t} \ns_t$ be defined by $x \mapsto (p(x),\pi_t(x))$. We want to show that this is a nilspace isomorphism. To prove that it is injective, suppose that $(p(x),\pi_t(x))=(p(y),\pi_t(y))$. Then, as $p(x)=p(y)$, we have $y = x+z$ for some $z\in Z$. Likewise, as $\pi_t(x)=\pi_t(y)$, there exists $\q\in \cu^{t+1}(\ns)$ such that $\q(1^{t+1})=y=x+z$ and $\q(v)=x$ for all $v\not=1^{t+1}$. This implies that $\q'$, defined as $\q'(1^{t+1})=z$ and $\q'(v)=0$ for all $v\not=1^{t+1}$ is an element of $\cu^{t+1}(\mc{D}_t(Z))$. Thus, $z=0$. To prove the surjectivity, let $(a,b)\in \ns' \times_{\ns'_t} \ns_t$. Let $x\in \ns$ be such that $p(x)=a$. Thus, $\pi_t(a)=p_t(b)=p_t(\pi_t(x))$. As $p_t$ is a degree-$t$ extension by $Z$, there exists $z\in Z$ such that $b=\pi_t(x)+z$. Now it is straightforward to check that $\varphi(x+z) = (a,b)$. Finally, we need to check that both $\varphi$ and $\varphi^{-1}$ are morphisms. As $\varphi$ is clearly a morphism, let $\q_1\times \q_2 \in \cu^n(\ns' \times_{\ns'_t} \ns_t)$. Let $\q\in \cu^n(\ns)$ be a cube such that $p\co \q = \q_1$. Then, $p_t\co\pi_t\co\q = \pi_t \co p\co \q = \pi_t \co \q_1 = p_t \co \q_2$. As $p_t$ is a degree-$t$ extension, this means that there exists $d\in \cu^n(\mc{D}_t(Z))$ such that $\pi_t\co \q+d = \q_2$. Thus, $\varphi^{-1} \co (\q_1\times \q_2) = \q+d\in \cu^n(\ns)$.
\end{proof}

We shall also use the following construction of an auxiliary nilspace. 

\begin{proposition}\label{prop:H-nil} Let $\nss$ be a $k$-step nilspace and let $H<\ab_k(\nss)$ be any subgroup. Let us define the following  relation on $\nss$: for $y_1,y_2\in\nss$, we have $y_1\sim y_2$ if and only if $y_1=y_2+h$ for some $h\in H$. Then the following holds:

\begin{enumerate}[leftmargin=0.8cm]
    \item The relation $\sim$ is an equivalence relation.
    \item The set $\tilde{\nss}:=\nss/\sim$ together with the sets $\cu^n(\tilde{\nss}):=\{\pi_{\sim}\co \q:\q\in\cu^n(\nss)\}$ is a nilspace.
    \item $\tilde{\nss}$ is $k$-step, $\ab_k(\tilde{\nss})=\ab_k(\nss)/H$ and $\tilde{\nss}_{k-1}\cong \nss_{k-1}$.
\end{enumerate}
\end{proposition}

\begin{proof} 
To prove $(i)$, the only non-trivial part is the transitivity of $\sim$. If $y_1\sim y_2$ and $y_2\sim y_3$ then $y_1=y_2+h$ and $y_2=y_3+h'$. Thus, $y_1=y_3+(h+h')$.

To prove $(ii)$, note first that the composition and ergodicity axioms follows easily from the definitions. To prove the completion axiom, let $\q'\in \cor^n(\tilde{\nss})$ for any $n\ge 1$. For every $v\in \db{n}_{\le k}$, let $y(v)\in\nss$ be any element such that $\pi_{\sim}(y(v))=\q'(v)$. Let $\q:\db{n}_{\le k}\to \nss$ be defined as  $\q(v)=y(v)$ for all $v\in\db{n}_{\le k}$.

\textbf{Case $n\le k+1$:} In this case, we have $\q$ defined in $\db{n}\setminus\{1^n\}$. Let $F$ be any lower face of dimension $n-1$. As $\q'|_F\in \cu^{n-1}(\tnss)$ then there exists $d\in \cu^{n-1}(\nss)$ such that $\q'|_F = \pi_{\sim}\co d$. Then, $\q|_F:\db{n-1}\to \nss$ is a function such that $\pi_{\sim}\co d(v) = \pi_{\sim}\co \q|_F(v)$ for all $v\in\db{n-1}$. Therefore, $d-\q|_F:\db{n-1}\to H$. As $n-1\le k$ we have that $d-\q|_F\in \cu^{n-1}(\mc{D}_k(H))$. Thus $\q|_F = d-(d-\q|_F)\in \cu^{n-1}(\nss)$. As this holds for every lower face $F$, we have that $\q\in\cor^n(\nss)$ and if we complete it to an element of $\cu^n(\nss)$ (abusing the notation, let us denote by $\q$ this completion), we have that $\pi_{\sim} \co \q$ is a completion of the corner $\q'$.

\textbf{Case $n\ge k+2$:} In this case, we have $\q$ defined in $\db{n}_{\le k}$. By a similar argument as before we can conclude that $\q\in\hom(\db{n}_{\le k},\nss)$ (seeing $\db{n}_{\le k}$ as a simplicial cubespace). Abusing the notation, let us denote again by $\q$ its (unique) completion in $\cu^n(\nss)$ (using simpicial completion \cite[Lemma 3.1.5]{Cand:Notes1}). Let now $F$ be a lower face of $\db{n}$ of dimension $n-1$. Then $\q'|_F = \pi_{\sim} \co d$ for some $d\in \cu^{n-1}(\nss)$. Thus, $\pi_{\sim} \co d(v) = \pi_{\sim} \co \q|_F(v)$ for all $v\in \db{n-1}_{\le k}$. Note that $(\q|_F-d):\db{n-1}_{\le k}\to H$ is a function that can be completed to an element $f\in \cu^{n-1}(\mc{D}_k(H))$. Thus $(d+f)(v)=\q|_F(v)$ for all $v\in \db{n-1}_{\le k}$. But using that $\nss$ is $k$-step, we have that $(d+f)(v)=\q|_F(v)$ for all $v\in \db{n-1}$. Therefore $\pi_{\sim}\co d = \pi_{\sim} \co \q|_F = \q'|_{F}$. To conclude, note that $\pi_{\sim} \co \q$ is an element of $\cu^n(\tnss)$ that completes the corner $\q'$.

Finally, let us prove $(iii)$. First of all, to prove that $\tnss$ is $k$-step, suppose that we have $\pi_{\sim}\co \q_1 =\pi_{\sim}\co \q_2 $ for all $v\not=1^{k+1}$ where $\q_1,\q_2\in \cu^{k+1}(\nss)$. Let $d:\db{k+1}\setminus\{1^{k+1}\}\to H$ be defined as $d:=\q_1-\q_2$. Abusing the notation, let us denote by $d$ its unique completion in $\cu^{k+1}(\mc{D}_k(H))$. Thus, $\q_1=\q_2+d$ for all $v\not=1^{k+1}$, but $\nss$ is $k$-step, and therefore those cubes must be equal. This implies that $\q_1(1^{k+1})=\q_2(1^{k+1})+d(1^{k+1})$ which in turns means that $\pi_{\sim}(\q_1(1^{k+1}))=\pi_{\sim}(\q_2(1^{k+1}))$.

Now let us define $\phi:\tnss_{k-1} \to \nss_{k-1}$ as $\pi_{k-1}(\pi_{\sim}(y))\mapsto \pi_{k-1}(y)$. We want to prove that this is a nilspace isomorphism. First we need to prove that this is well-defined. Let $y_1,y_2\in \nss$ be such that $\pi_{k-1}(\pi_{\sim}(y_1))=\pi_{k-1}(\pi_{\sim}(y_2))$. Then there exists a cube $\q\in\cu^k(\nss)$ such that $\pi_{\sim}(\q(v))=\pi_{\sim}(y_1)$ for all $v\not=1^k$ and $\pi_{\sim}(\q(1^k))=\pi_{\sim}(y_2)$. Using that every function $f:\db{k}\to H$ is an element of $\cu^k(\mc{D}_k(H))$, we have that for some $f\in \cu^k(\mc{D}_k(H))$ the function $\q+f$ such that $(\q+f)(v)=y_1$ for all $v\not=1^k$ and $(\q+f)(1^k)=y_2$ is an element of $\cu^{k}(\nss)$. Thus $\pi_{k-1}(y_1)=\pi_{k-1}(y_2)$.

To prove that $\phi$ is injective, take two elements $y_1,y_2\in \nss$ such that $\pi_{k-1}(y_1)=\pi_{k-1}(y_2)$. Thus there exists a cube $\q\in \cu^k(\nss)$ such that $\q(v)=y_1$ for all $v\not=1^k$ and $\q(1^k)=y_2$. Then compose with $\pi_{\sim}$ and conclude that $\pi_{k-1}(\pi_{\sim}(y_1))=\pi_{k-1}(\pi_{\sim}(y_2))$. The fact that $\phi$ is surjective is trivial. To prove that $\phi$ is a morphism, let $\pi_{k-1}\co\pi_{\sim}\co \q\in \cu^n(\tnss_{k-1})$, where $\q\in\cu^n(\nss)$, be any cube. Then $\phi\co \pi_{k-1}\co\pi_{\sim}\co \q = \pi_{k-1}\co \q\in \cu^n(\nss_{k-1})$. And to prove that $\phi^{-1}$ is a morphism, for any $\q\in \cu^n(\nss)$ we have that $\phi^{-1}\co \pi_{k-1}\co \q = \pi_{k-1}\co \pi_{\sim}\co \q \in \cu^n(\tnss_{k-1})$.

To conclude the proof, let us prove that $\ab_k(\tnss) = \ab_k(\nss)/H$. First, let us define the action of $\ab_k(\nss)/H$ over $\tnss$ as $(z+H,\pi_{\sim}(y))\mapsto \pi_{\sim}(y+z)$. Let us check that this is well-defined. Suppose that $\pi_{\sim}(y_1)=\pi_{\sim}(y_2)$ and $z_1+H=z_2+H$. Then the first equality implies that $y_1=y_2+h$ for some $h\in H$. Similarly, the second equality implies that $z_1=z_2+h'$ for some $h'\in H$. Thus, we have that $\pi_{\sim}(y_1+z_1)=\pi_{\sim}(y_2+h+z_2+h')=\pi_{\sim}(y_2+z_2)$. To prove that the action is transitive, it is enough to prove that the fibers over a single element are covered by the action of $\ab_k(\nss)/H$. Suppose that we have $\pi_{k-1}(\pi_{\sim}(y_1))=\pi_{k-1}(\pi_{\sim}(y_2))$. We have proved earlier that $\phi$ was an isomorphism. Thus, $\pi_{k-1}(y_1)=\pi_{k-1}(y_2)$ and there exists $z\in \ab_k(\nss)$ such that $y_1=y_2+z$. With this we conclude that $\pi_{\sim}(y_1)=\pi_{\sim}(y_2+z)=\pi_{\sim}(y_2)+(z+H)$. To conclude the proof, we need to see that this action is free. Let $z+H$ be such that $\pi_{\sim}(y)=\pi_{\sim}(y)+(z+H)=\pi_{\sim}(y+z)$. this implies that for some $h\in H$ we have that $y=y+z+h$. But this is an equality in $\nss$, and as the action of $\ab_k(\nss)$ is free, we have that $z+h=0$. Therefore $z+H=H$.
\end{proof}
\begin{proposition}\label{prop:fac-fib-prod}
Let $\nss,\nss'$ and $\ns$ be $k$-step nilspaces. Let $\varphi:\nss\to\ns$ and $\psi:\nss'\to\ns$ be fibrations. Then for every $t\le k$ we have that $(\nss\times_{\ns}\nss')_t\simeq \nss_t\times_{\ns_t}\nss'_t$. \end{proposition}
\begin{proof}
We prove this by induction on $k-i$ for $i\in [k]$. Note that it suffices to prove this result for the case $i=1$, as then the general result will  follow from applying repeatedly this case. Hence, let us prove that $(\nss\times_{\ns}\nss')_{k-1}\simeq \nss_{k-1}\times_{\ns_{k-1}}\nss'_{k-1}$.

Let us define the map $T:\nss\times_{\ns}\nss'\to \nss_{k-1}\times_{\ns_{k-1}}\nss'_{k-1}$ as $(y,y')\mapsto (\pi(y),\pi(y'))$ where $\pi$ denotes throughout this proof the projection to the $k-1$ factor of any nilspace. This map is easily seen to be a well-defined morphism. Let us see that it is indeed a fibration. Take $(d,d')\in \cor^n(\nss\times_{\ns}\nss')$ and $(\q,\q')\in \cu^n(\nss_{k-1}\times_{\ns_{k-1}}\nss'_{k-1})$ such that $T\co (d,d')=(\q,\q')$ for any $v\not=1^n$. As this means that $\pi\co d=\q$ and $\pi\co d'=\q'$ for $v\not=1^n$, abusing the notation, let $d$ and $d'$ be completions of $\q$ and $\q'$ respectively. By definition $\varphi\co d = \psi\co d'$ for all $v\not=1^n$. Also, as $\varphi_{k-1}\co\q = \psi_{k-1}\co \q'$ for every $v\in\db{n}$ this means that for every $v\in \db{n}$ we have $\pi\co\varphi\co d = \varphi_{k-1}\co\q = \psi_{k-1}\co \q' = \pi\co\psi\co d'$. Thus, there exists a cube $h\in \cu^n(\mc{D}_k(\ab_k(\ns)))$ such that $\varphi\co d = \psi\co d'+h$. But on the other hand we know that $\varphi\co d = \psi\co d'$ for all $v\not=1^n$. Thus the cube $h$ must have zero value for every point except maybe for $v=1^n$. If $n\ge k+1$, as $\ns$ is $k$-step this means that $h(1^n)=0$ and we are done, as $(d,d')\in \cu^n(\nss\times_{\ns}\nss')$ is a cube that lifts $(\q,\q')$. If $n\le k$, let $z'\in \ab_k(\nss')$ be any element such that $\psi_k(z')=z$ where $\psi_k$ is the $k$-th structure morphism of the fibration $\psi$ (and hence, it is surjective). Define $h'\in \cu^n(\mc{D}_k(\ab_k(\nss')))$ as $h'(v)=0$ for $v\not=1^n$ and $h'(1^n)=z'$. It is then easy to see that $(d,d'+h')\in \cu^n(\nss\times_{\ns}\nss')$ is a cube that lifts $(\q,\q')$.

Now, in order to prove that $(\nss\times_{\ns}\nss')_{k-1}\simeq \nss_{k-1}\times_{\ns_{k-1}}\nss'_{k-1}$ note that it is enough to see that $T_{k-1}$ is injective. The reason is the following. We already know that $T_{k-1}:(\nss\times_{\ns}\nss')_{k-1}\to \nss_{k-1}\times_{\ns_{k-1}}\nss'_{k-1}$ is a fibration. In particular, it is a surjective map. If in addition it is injective, then it is invertible. Hence, for any cube $\q\in \cu^n(\nss_{k-1}\times_{\ns_{k-1}}\nss'_{k-1})$ let $\q'\in \cu^n((\nss\times_{\ns}\nss')_{k-1})$ be such that $T_{k-1}\co \q'=\q$. As $T_{k-1}$ is invertible we have that $T_{k-1}^{-1}\co\q = \q'\in \cu^n((\nss\times_{\ns}\nss')_{k-1})$ and thus $T_{k-1}^{-1}$ would be a morphism and the proof would be completed.

Thus, let us see the injectivity of $T_{k-1}$. Let $\pi(y_1,y_1')=\pi(y_2,y_2')$ be any pair of elements in $(\nss\times_{\ns}\nss')$ such that $T_{k-1}(\pi(y_1,y_1'))=T_{k-1}(\pi(y_2,y_2'))$. But now $T_{k-1}(\pi(y_1,y_1')) = \pi(T(y_1,y_1')) = \pi(\pi(y_1),\pi(y_1'))=(\pi(y_1),\pi(y_1')$ where the last equality follows from the fact that $\nss_{k-1}\times_{\ns_{k-1}}\nss'_{k-1}$ is already $k-1$-step. By a similar argument with $(y_2,y_2')$ we conclude that $(\pi(y_1),\pi(y_1'))=(\pi(y_2),\pi(y_2'))$. But this by definition implies that there exists a cube $\q\in \cu^k(\nss)$ such that $\q(v)=y_1$ for all $v\not=1^n$ and $\q(1^k)=y_2$ and a cube $\q'\in \cu^k(\nss')$ such that $\q'(v)=y_1'$ for all $v\not=1^n$ and $\q(1^k)=y_2'$. Hence the cube $(\q,\q')$ is in $\cu^k(\nss\times_{\ns}\nss')$ which by definition means that $\pi(y_1,y_1')=\pi(y_2,y_2')$.\end{proof}

\section{Auxiliary results on $p$-homogeneous nilspaces}\label{app:aux-p-hom}
\noindent In this appendix we record some technical results and definitions that are used several times in the paper. 

Let $\ns$ be a group nilspace associated with a filtered group $(G,G_{\bullet})$. To prove that a function $f:\mc{D}_1(\mb{Z}_p)\to \ns$ is a morphism, it suffices (see e.g. \cite[Theorem 2.2.14]{Cand:Notes1}) to take derivatives of $f$ and check that the resulting functions take values in the correct subgroup in the filtration. That is, it suffices to ensure that $\partial_{a_1}\cdots\partial_{a_{\ell}}f \in G_{i_1+\cdots+i_{\ell}}$ where $a_j\in G_{i_j}$ for all $j=1,\ldots,{\ell}$. It is easy to see that it is enough to check this with $a_j=1$ for all $j$, so we can focus on computing $\partial_1^t f$ and checking that it takes values in $G_t$. Furthermore, it will actually suffice to consider the case $G=\mb{Z}$ with some filtration. The reason is that, for more general groups $G$, we will actually want to prove that some functions of the form $f(x)=g^{m(x)}$ for $g\in G_j$ are morphisms, and taking derivatives of such functions is equivalent to taking derivatives of $m(x)$ over $\mb{Z}$ (relative to some filtration on $\mb{Z}$).

To calculate such derivatives, let us think of $f$ as a vector $v\in\mb{Z}^p$, namely $v=(f(0),f(1),\ldots,f(p-1))$. It is then easy to check that the values of $\partial_1^t f$ will be given by the entries of the matrix $A_p^t v$ where
\begin{equation}\label{eq:Ap}
A_p:=\begin{pmatrix}
-1 & 1 & 0 & \cdots & 0 \\
0& -1 & 1  & \cdots & 0 \\
0 & 0 & -1  & \cdots & 0 \\
\vdots &   &   & \ddots & \\
1 & 0 & \cdots & 0 & -1.
\end{pmatrix}.
\end{equation}
The case $p=2$ will always be treated separately, but typically it will be easier. For any $n\in \mb{Z}$, let us denote by $(n)_p$ the residue of $n$ modulo $p$ that lies in $[p]=\{1,2,\ldots,p\}$. The following concept will also be important in the arguments below.
\begin{defn}\label{def:circular-vector} Let $p$ be a prime. If $p$ is odd, we call a vector $v\in \mb{Z}^p$ \emph{circular} if there exists $i\in [p]$ such that $v_i=0$ and for all $j \in [\frac{p-1}{2}]$, $v_{(i+j)_p}=-v_{(i-j)_p}$. If $p=2$, we say that a vector is \emph{circular} if $v_1=-v_2$.
\end{defn}
\noindent With this definition, let us prove the following result.
\begin{proposition}\label{prop:circular-mult-p} Let $p$ be a prime, let $v\in \mb{Z}^p$ be a circular vector, and let $A_p\in M_{p\times p}(\mb{Z})$ be as defined in \eqref{eq:Ap}. Then $A_p^{p-1}v$ is a circular vector such that all its coordinates are multiples of $p$.
\end{proposition}

\begin{proof}
The case $p=2$ follows from a simple calculation, so let us assume that $p$ is odd. It can be proved by induction that $A_p^{p-1}$ has the following form:
\[
A_p^{p-1}:=\begin{pmatrix}
\binom{p-1}{0} & -\binom{p-1}{1} & \binom{p-1}{2} & \cdots & \binom{p-1}{p-1} \\
\binom{p-1}{p-1}& \binom{p-1}{0} & -\binom{p-1}{1}  & \cdots & -\binom{p-1}{p-2} \\
-\binom{p-1}{p-2} & \binom{p-1}{p-1} & \binom{p-1}{0}  & \cdots & \binom{p-1}{p-3} \\
\vdots &   &   & \ddots & \\
-\binom{p-1}{1} & \binom{p-1}{2} & \cdots & \cdots & \binom{p-1}{0}
\end{pmatrix}.
\]
Denoting by $(t)^*_p$ the residue of $t$ modulo $p$ that lies in $\{0,\ldots,p-1\}$, we have $(A_p^{p-1})_{i,j} =(-1)^{(j-i)^*_p} \binom{p-1}{(j-i)^*_p}$. To prove that all entries of $A_p^{p-1}v$ are multiples of $p$, just note that, viewing every entry of $A_p^{p-1}$ modulo $p$, we get that if $r:=(j-i)^*_p$ then $(A_p^{p-1})_{i,j} =(-1)^{r} \binom{p-1}{r}= (-1)^{r}\frac{(p-1)(p-2)\cdots (p-r)}{r(r-1)\cdots 1}= (-1)^{r}\frac{(-1)(-2)\cdots (-r)}{r(r-1)\cdots 1}=1 \mod p$. Hence, when we multiply $A_p^{p-1}$ by a circular vector and we view  it modulo $p$, the sum is 0 (essentially because the matrix $(A_p^{p-1})_{i,j}=1 \mod p$ for all $i,j$).

To complete the proof, we need to show that $A_p^{p-1}v$ is circular. Let us prove only that there is some coordinate such that its value is 0 (the proof of the circularity is essentially the same). As $v$ is circular, suppose that $v_i=0$ and let $j\in [p]$ be the row such that the term $(-1)^{\frac{p-1}{2}}\binom{p-1}{\frac{p-1}{2}}$ is in the position $i$, i.e., $(j-i)^*_p = \frac{p-1}{2}$. We want to prove that $(A_p^{p-1}v)_j=0$ (indeed, this will be the \textit{centre} of the circular vector $A_p^{p-1}v$). The idea is simple: we just write
\[
(A_p^{p-1}v)_j = (-1)^{(j-i)^*_p} \binom{p-1}{(j-i)^*_p}v_i+\sum_{m=1}^{\frac{p-1}{2}} (-1)^{(j-i-m)^*_p} \binom{p-1}{(j-i-m)^*_p}v_{(i-m)^*_p}
\]\[
+\sum_{m=1}^{\frac{p-1}{2}} (-1)^{(j-i+m)^*_p} \binom{p-1}{(j-i+m)^*_p}v_{(i+m)^*_p}
\]
and note that the first term cancels because $v_i=0$ and the other terms cancel pairwise, by the identity $\binom{n}{r}=\binom{n}{n-r}$, the fact that $v$ is circular, and the fact that $(j-i)^*_p = \frac{p-1}{2}$.
\end{proof}

\begin{corollary}\label{cor:der-circ} Let $f:\mc{D}_1(\mb{Z}_p)\to \mb{Z}$ be a function. If $f$ is circular \textup{(}viewed as a vector\textup{)}, then $\partial_1^{p-1} f\in p\mb{Z}$.
\end{corollary}
\noindent Next, let us recall from \eqref{eq:gnsdef} the definition of the group nilspaces $H^{(p)}_i$ (for $i\ge 1$), consisting of $\mb{Z}$ equipped with the filtration
\[   
\left(H^{(p)}_i\right)_j = 
     \begin{cases}
       \mb{Z} &\quad\text{if } j=0,1,\ldots,i\\
       p^{\lfloor\frac{j-i-1}{p-1}\rfloor+1}\mb{Z} &\quad\text{if } j\ge i+1. \\ 
     \end{cases}
\]
\noindent By definition we take $H_0^{(p)}:=H_1^{(p)}$. 

\begin{lemma}\label{lem:prod-m}
Let $n\in \mb{N}$, let $t=(t_1,\ldots,t_n)\in \mb{Z}_{\ge 0}^n$, and for $j\in [n]$ let $m_{t_j}^{(p)}\in \hom(\mc{D}_1(\mb{Z}),H_{t_j}^{(p)})$. Let $g':\mc{D}_1(\mb{Z}^n)\to \mb{Z}$ be defined by $g'(x)=m_{t_1}^{(p)}(x_1)\, m_{t_2}^{(p)}(x_2)\cdots m_{t_n}^{(p)}(x_n)$. Then $g'\in \hom\big(\mc{D}_1(\mb{Z}^n),H^{(p)}_{|t|}\big)$.
\end{lemma}

\begin{proof}
Let $(e_j)_{j\in [n]}$ be the standard basis of $\mb{Z}^n$. We just have to check that if we take derivatives of $g'$, we land in the correct subgroup in the filtration. We can focus on derivatives involving the generators, i.e.\  we just need to check that for every $a=(a_1,\ldots,a_n)\in \mb{Z}_{\ge 0}^n$ we have that $ 
\partial_{e_1}^{a_1}\cdots \partial_{e_n}^{a_n} g'$ takes values in $(H^{(p)}_{|t|})_{|a|}$.

This derivative equals $\partial_{e_1}^{a_1}m_{t_1}^{(p)}\cdots \partial_{e_n}^{a_n}m_{t_n}^{(p)}$. By what we know about the morphisms $m_{t_j}^{(p)}$, this derivative takes values in $p^r\mb{Z}$ where $r=\sum_{j=1}^n \max\left(0,\left\lfloor \frac{a_j-t_j-1}{p-1}\right\rfloor+1\right)$. On the other hand, the exponent $r'$ satisfying $p^{r'}\mb{Z}=(H^{(p)}_{|t|})_{|a|}$ is $r'=\max\left(\left\lfloor\frac{\sum_{j=1}^n (a_j-t_j)-1}{p-1}\right\rfloor+1,0\right)$. Hence, to ensure that the above derivative takes values in the appropriate subgroup, we just have to check that $r'\leq r$.

To prove this, note that if $|a|\le |t|$ then $r'=0$ so there is nothing to prove. If $|a| > |t|$ then we are going to show that
\begin{equation}\label{eq:min-der}
    \min_{a': |a'|=|a|}\left( \sum_{j=1}^n \max\left(0,\left\lfloor \frac{a_j'-t_j-1}{p-1}\right\rfloor+1\right)\right)
\end{equation}
is attained for $a_1'=|a|-\sum_{j=2}^n t_j$ and $a_j'=t_j$ for $j\ge 2$ (there are other $n$-tuples that attain the minimum, this one here is just one of them). If we prove this then we are done, because for this particular $n$-tuple $a'$ the inequality is trivial.

To do this, let us think of the coordinates $a_j$ of an $n$-tuple $a$ as containers of derivatives, so if we say move $\ell$ derivatives from $a_1$ to $a_2$ we mean that we consider the $n$-tuple $a_1-{\ell},a_2+{\ell},a_3\ldots,a_n$ (and this of course will preserve $|a|$). Now let $a=(a_1,\ldots,a_n)$ be an $n$-tuple that attains the minimum in \eqref{eq:min-der}. First note that we can always assume that $a_j\ge t_j$ for all $j\ge 1$; this is because we can move derivatives from variables with an \textit{excess} to variables with \textit{lack} of derivatives and the minimum must not change. So if we move ${\ell}$ derivatives from (say) $a_1$ to $a_2$, the summand $\max\left(0,\left\lfloor \frac{a_1-{\ell}-t_1-1}{p-1}\right\rfloor+1\right)$ must be equal to $\max\left(0,\left\lfloor \frac{a_1-t_1-1}{p-1}\right\rfloor+1\right)$ (otherwise we contradict our assumption that $a$ achieves the minimum) because as long as $a_2+{\ell}\le t_2$, there is no increment in changing $a_2$ by $a_2+{\ell}$.

Next, note that we can move blocks of $p-1$ derivatives from any $a_j$ ($j\ge 2$) to $a_1$ without modifying the minimum until we get that $t_j\le a_j< t_j+p-1$ for all $j\ge 2$. Now, for every $j\ge 2$, we can move ${\ell}_j:=a_j-t_j$ derivatives from $a_j$ to $a_1$. There are two cases for every $j$. If ${\ell}_j=0$ then nothing happens. If ${\ell}_j\ge 1$ then, since $a$ achieves the minimum, we must have $\max\left(0,\left\lfloor \frac{a_1+{\ell}_j-t_1-1}{p-1}\right\rfloor+1\right)=\max\left(0,\left\lfloor \frac{a_1-t_1-1}{p-1}\right\rfloor+1\right)+1$ (this will be with the new configuration where we have $a_1+{\ell}_j$ and $a_j-{\ell}_j$).
\end{proof}

\subsection{On $p$-homogeneous extensions of the elements of $\mc{Q}_{p,k}$}\hfill\\
Recall from Proposition \ref{prop:H-nil} that given a $k$-step nilspace $\ns$, and a subgroup $H$ of the last structure group $\ab_k(\ns)$, we can define the quotient nilspace $\ns/H$ under the relation $x\sim_H y$ if and only if $x=y+h$ for some $h\in H$. Recall also the class $\mc{Q}_{p,k}$ of $p$-homogeneous $k$-step group nilspaces from Definition \ref{def:bblocks-intro}.

\begin{proposition}\label{prop:factorization-last-str-group}
Let $\ns$ be a nilspace in $\mc{Q}_{p,k}$ and let $H$ be a subgroup of $\ab_k(\ns)$. Then $\ns/H \cong \nss\times\nss'$ where $\nss\in \mc{Q}_{p,k-1}$ and $\nss'\in \mc{Q}_{p,k}$.
\end{proposition}
\noindent Since our proof below is technical, it may be useful first to describe a motivating example. Supposing that $\ab_k=\mb{Z}_p^m$ and that $H\leq\mb{Z}_p^m$ has a simple structure (for example, that it is generated by a subset of  $(e_i)_{i\in [m]}$, where $e_i$ is the element with $e_i(i)=1$ and $e_i(j)=0$ for $j\neq i$), then we can explicitly describe $\ns/H$, using in particular the basic fact that the quotient of the nilspace $\abph_{k,\ell}$ by the action of its last structure group is the nilspace $\abph_{k-1,\ell}$. To illustrate this in detail, let $p=3$, $m=2$, and suppose that $\ns$ is the product nilspace $\abph_{4,2}^{(3)} \times \abph_{4,4}^{(3)}$ (thus $k=4$ here). If $H=\mb{Z}_3^2$ then it is easily seen that $\ns/H \cong \mc{D}_2(\mb{Z}_3)$. Next, suppose that $H=\langle (1,0)\rangle$. In this case $\ns/H\cong \mc{D}_2(\mb{Z}_3)\times \mc{D}_4(\mb{Z}_3)$. Similarly, if $H=\langle (0,1)\rangle$ then $\ns/H=\abph^{(3)}_{4,2}$. A more interesting case is when $H=\langle (1,1)\rangle$. To discuss this case, we shall use the injective homomorphism of abelian groups $i:\mb{Z}_3\to \mb{Z}_9$, $y\!\!\mod 3\mapsto 3y\!\!\mod 9$. Denoting elements of $\ns$ as couples $(x,y)\in \abph_{4,2}^{(3)}\times \abph_{4,4}^{(3)}$, we define the bijection $\varphi:\ns\to \ns$, $(x,y)\mapsto (x+i(y),y)$, and we note that this is a nilspace isomorphism, because $i$ respects the required filtrations (that is, for every $j$, the map $i$ sends the $j$-th subgroup of $\mb{Z}_3$ in the $\abph_{4,4}^{(3)}$ filtration into the $j$-th subgroup of $\mb{Z}_9$ in the $\abph_{4,2}^{(3)}$ filtration) and therefore $i$ is a filtered group homomorphism, hence a nilspace morphism. Now we observe that, through $\varphi^{-1}$, the action of $H=\langle(1,1)\rangle$ by addition becomes the action of the subgroup $\langle(0,1)\rangle$. More precisely, letting $\phi_4$ be the last structure morphism of $\varphi$ (see \cite[Definition 3.3.1]{Cand:Notes1}), it can be seen that $\phi_4^{-1}(H)=\langle(0,1)\rangle$, because of the easily checked equality $\varphi^{-1}((x,y)+(1,1))=\varphi^{-1}(x,y) + (0,1)$. Hence, instead of computing $\ns/H$, we can compute $\varphi^{-1}(\ns)/\phi_4^{-1}(H)$ more easily, thus concluding that $\ns/H \cong \varphi^{-1}(\ns)/\phi_4^{-1}(H) \cong \abph^{(3)}_{4,2}$.

\begin{proof}[Proof of Proposition \ref{prop:factorization-last-str-group}]
Since $\ns\in \mc{Q}_{p,k}$, there exist non-negative integers $a_0,\ldots,a_t$,  $t=\lfloor k/(p-1)\rfloor$,  such that
\[
\ns= \abph_{k,k-t(p-1)}^{\,a_0} \times \abph_{k,k-(t-1)(p-1)}^{\, a_1}\times \cdots\times  \abph_{k,k}^{\,a_t} \times Q',
\]
where $Q'\in \mc{Q}_{p,k-1}$. The key point of this expression of $\ns$ is that it isolates the terms that \emph{contribute} to the $k$-th structure group of $\ns$. Thus, the elements of the $k$-th structure group of $\ns$ can be written as tuples $(f_0,f_1,\ldots,f_t,0)\in (p^t\mb{Z}_{p^{t+1}})^{a_0}\times (p^{t-1}\mb{Z}_{p^t})^{a_1}\times \cdots \times \mb{Z}_p^{a_t}\times Q'$ and the action of this group on $\ns$ is by coordinate-wise addition. Note that this $k$-th structure group is isomorphic to $\mb{Z}_p^{a_0+\cdots+a_t}$.

The second observation is that we have the following chain of inclusions:
\begin{equation}\label{eq:inclusions}
    \abph_{k,k-t(p-1)} \supset \abph_{k,k-(t-1)(p-1)} \supset  \cdots\supset  \abph_{k,k}.
\end{equation}
With this we mean that for any $r\le {\ell}$ we can define a homomorphism $\mb{Z}_{p^r}\to \mb{Z}_{p^{{\ell}}}$, namely $x\mod p^r\mapsto p^{{\ell}-r}x\mod p^{{\ell}}$. Moreover, this is a filtered homomorphism (with the filtrations defining $\abph^{(p)}_{k,k-j(p-1)}$ for $j=0,\ldots,t$) and thus it is also a nilspace morphism. With this notation, let us define $\varphi:\ns\to \ns$ as the map sending $(x_0,\ldots,x_t,q)$ to
\[
(A_{0,0} x_0+ A_{0,1}(px_1)+\cdots+A_{0,t}(p^{t}x_t),\;   A_{1,1}x_1+A_{1,2}(px_2)+\cdots+A_{1,t}(p^{t-1}x_t),\;   \ldots,\; A_{t,t}x_t,\; q), 
\]
where $x_i\in (\abph_{k,k-(t-i)(p-1)})^{a_i}$, $A_{i,j}$ a matrix in $\mb{Z}^{a_i\times a_j}$ with $\det(A_{i,i})$ coprime with $p$ for all $i=0,\ldots,t$ and $j\in [i,t]$.

Let us see why this construction makes sense. First note that we are using \eqref{eq:inclusions} to be able to sum any element $x_j$ with $x_i$ for $j>i$. This already implies that $\varphi$ is a nilspace morphism. Now  note that $A_{i,i}$ is invertible as a matrix over $\mb{Z}_{p^r}$ \emph{for all $r\ge 1$} (since $\det(A_{i,i})$ is coprime with $p$ and hence with $p^r$ for all $r\ge 1$). This fact is crucial to prove that $\varphi$ is invertible and thus a nilspace isomorphism.

To prove this last sentence, let us compute the inverse of $\varphi$. If $\varphi(x_0,\ldots,x_t,q) = (y_0,\ldots,y_t,q)$ then $A_{t,t}x_t = y_t$. Let  $A_{t,t}^{-1}\in \mb{Z}^{a_t\times a_t}$ be defined as $D_t\adj(A_{t,t})$ where $D_t\in \mb{Z}$ is any integer such that $D_t\det(A_{t,t}) =1\!\mod p$ (note that $\det(A_{t,t})\!\mod p$ is non-zero by hypothesis). It is clear that, as  linear maps on $\mb{Z}_p^{a_t}$, the matrices $A_{t,t}$, $A_{t,t}^{-1}$ are inverses of each other, and thus $x_t=A_{t,t}^{-1}y_t$. Next, let us solve the equation $A_{t-1,t-1}x_{t-1}+A_{t-1,t}px_t = y_{t-1}$. Using the previous result we have that $A_{t-1,t-1}x_{t-1}=y_{t-1}-A_{t-1,t}\, p\, A_{t,t}^{-1}\, y_t$, which equals $y_{t-1}-A_{t-1,t}A_{t,t}^{-1}\, p\, y_t$.

Now we repeat the same trick as before, but this time we define $A_{t-1,t-1}^{-1}\in \mb{Z}^{a_{t-1}\times a_{t-1}}$ as the matrix $D_{t-1}\adj(A_{t-1,t-1})$ where $D_t \det(A_{t-1,t-1})=1\! \mod p^2$ (note that, since $\det(A_{t-1,t-1})$ is coprime with $p$, it is coprime with $p^2$ whence $\det(A_{t-1,t-1})$ is invertible $\!\!\mod p^2$). Thus it is clear that $x_{t-1}=A_{t-1,t-1}^{-1}y_{t-1}-A_{t-1,t-1}^{-1}A_{t-1,t}A_{t,t}^{-1}py_t$. Repeating this process, we end up obtaining an inverse of the function $\varphi$ that has the same structure as $\varphi$ and is therefore also a nilspace morphism.

We now explain how we use such an automorphism $\varphi$. The idea is that instead of computing $\ns/H$ we can compute $\varphi^{-1}(\ns)/\phi_k^{-1}(H)$ (where $\phi_k:\ab_k(\ns)\to \ab_k(\ns)$ is the $k$-th structure morphism of $\varphi$). We have $\varphi^{-1}(\ns)=\ns$, and  our goal is then to choose $\varphi$ so that $\phi_k^{-1}(H)$ is a subspace generated by a subset of the standard basis $\{e_i\}_{i\in [a_0+\cdots+a_t]}$.

We claim that the linear map $\phi_k$ is represented by the following block matrix: 
\begin{equation}
    A=\begin{pmatrix}
    A_{0,0} & A_{0,1} & \cdots & A_{0,t} \\
    0 & A_{1,1} & \cdots & A_{1,t} \\
    \vdots &&& \\
    0 & 0 & \cdots & A_{t,t}
    \end{pmatrix}\in \mb{Z}_p^{(a_0+\cdots+a_t)\times (a_0+\cdots+a_t)},
\end{equation}
where the elements of the matrices $A_{i,j}$ are inserted modulo $p$ in $A$. The proof of this claim is just a routine computation: take $(z_0,\ldots,z_t)\in \mb{Z}_p^{a_0+\cdots+a_t}$ and note that by definition $(x_0,\ldots,x_t,q)+(z_0,\ldots,z_t) = (x_0+p^tz_0,x_1+p^{t-1}z_1,\ldots,x_t,q)$. Then apply $\varphi$ and by commutativity the claim follows. 

Hence, to complete the proof we just have to find the matrix $A$ (i.e.\ $\varphi$) adequately. We shall construct $A$ in such a way that some of its columns are the vectors generating the subspace $H$. Thus $H$ will be $\phi_k(\langle e_{i_1},\ldots,e_{i_w}\rangle)$ where $w=\dim(H)$. The process for constructing $A$ is as follows. First let us define the subspaces $U_i:=\mb{Z}_p^{a_0+\cdots+a_i}\times 0^{a_{i+1}+\cdots+a_t}$ for $i=0,\ldots, t$. Now we define the columns of $A$ iteratively as follows. First let $v^{(0)}_1,\ldots,v^{(0)}_{b_0}\in H$ be a basis of the subspace $H\cap U_0$, and complete this to a basis of the subspace $U_0$ with vectors $w^{(0)}_{b_0+1},\ldots,w^{(0)}_{a_0}\in U_0$. These vectors will constitute the first $a_0$ columns of $A$. (To be more precise, technically the matrix $A$ must have integer values and the values of the vectors $v^{(0)}_i$ and $w^{(0)}_j$ are in $\mb{Z}_p$. By abuse of notation, when we say that the vectors $v^{(0)}_i$ and $w^{(0)}_j$ constitute the first columns of $A$, we mean that we take the representative of any element of $\mb{Z}_p$ in $[0,p-1]$.)

Next, consider $H\cap U_1$ and complete the linearly independent set $\{v^{(0)}_1,\ldots,v^{(0)}_{b_0}\}\subset H\cap U_0\subset H \cap U_1$ to a basis of $H\cap U_1$. Let these vectors be $\{v^{(1)}_1,\ldots,v^{(1)}_{b_1}\}$. We claim that the set $\{v^{(0)}_1,\ldots,v^{(0)}_{b_0},w^{(0)}_{b_0+1},\ldots,w^{(0)}_{a_0},v^{(1)}_1,\ldots,v^{(1)}_{b_1}\}$ is linearly independent. To prove this, let $\lambda_1v^{(0)}_1+\cdots+\lambda_{b_0}v^{(0)}_{b_0}+\mu_1w^{(0)}_{b_0+1}+\cdots+\mu_{a_0}w^{(0)}_{a_0}+\gamma_1v^{(1)}_1+\cdots+\gamma_{b_1}v^{(1)}_{b_1}=0$ for some coefficients $\lambda_i,\mu_j,\gamma_k\in \mb{Z}_p$. This implies that $\lambda_1v^{(0)}_1+\cdots+\lambda_{b_0}v^{(0)}_{b_0}+\mu_1w^{(0)}_{b_0+1}+\cdots+\mu_{a_0}w^{(0)}_{a_0}=-\gamma_1v^{(1)}_1-\cdots-\gamma_{b_1}v^{(1)}_{b_1}$. And now the left hand side is in $U_0$ whereas the right hand side is in $H$. Thus both sides are in $H\cap U_0$. Therefore we know that for some coefficients $\rho_i\in\mb{Z}_p$ we have $\lambda_1v^{(0)}_1+\cdots+\lambda_{b_0}v^{(0)}_{b_0}+\mu_1w^{(0)}_{b_0+1}+\cdots+\mu_{a_0}w^{(0)}_{a_0}=\rho_1v^{(0)}_1+\cdots+\rho_{b_0}v^{(0)}_{b_0}$. Thus we have that $(\lambda_1-\rho_1)v^{(0)}_1+\cdots+(\lambda_{b_0}-\rho_{b_0})v^{(0)}_{b_0}+\mu_1w^{(0)}_{b_0+1}+\cdots+\mu_{a_0}w^{(0)}_{a_0}=0$. As $\{v^{(0)}_1,\ldots,v^{(0)}_{b_0},w^{(0)}_{b_0+1},\ldots,w^{(0)}_{a_0}\}$ is a basis of $U_0$ we know that $\mu_i=0$ for all $i=b_0+1,\ldots,a_0$. We conclude that, since $\{v^{(0)}_1,\ldots,v^{(0)}_{b_0},v^{(1)}_1,\ldots,v^{(1)}_{b_1}\}$ is a basis of $H\cap U_1$, we have $\lambda_j=\gamma_k=0$ for all $j\in[b_0]$ and $k\in [b_1]$. And finally we define the vectors $w^{(1)}_{b_1+1},\ldots,w^{(1)}_{a_1}\in U_1$ as any vectors that complete $\{v^{(0)}_1,\ldots,v^{(0)}_{b_0},w^{(0)}_{b_0+1},\ldots,w^{(0)}_{a_0},v^{(1)}_1,\ldots,v^{(1)}_{b_1}\}$ to a basis of $U_1$. The vectors $v^{(1)}_1,\ldots,v^{(1)}_{b_1}$ followed by the vectors $w^{(1)}_{b_1+1},\ldots,w^{(1)}_{a_1}$ will be the next $a_1$ columns of $A$ (with the previous convention of choosing a representative in $[0,p-1]$).

Continuing this process, we construct the matrix $A$ putting together the vectors $v^{(i)}_j$ and $w^{(i')}_{j'}$ in the order described above. The resulting matrix $A$ has the following structure:
\begin{equation}
  \begin{pmatrix}[c|c|c|c|c|c|c|c|c|c|c|c|c|c]
    &&&&&&&&&&&&&\\
    v_1^{(0)}&\cdots & v_{b_0}^{(0)}  & w_{b_0+1}^{(0)} & \cdots & w_{a_0}^{(0)} & v_1^{(1)}&\cdots & v_{b_1}^{(1)}  & w_{b_1+1}^{(1)} & \cdots  & w_{a_1}^{(1)} & \cdots & w_{a_t}^{(t)}
    \\
    &&&&&&&&&&&&&\\
\end{pmatrix}.
\end{equation}
By construction this matrix has the desired shape and also, as $\det(A)=\prod_{i=0}^t\det(A_{i,i})\! \mod p$ and the vectors $\{v^{(i)}_j\}_{i\in[0,t],j\in [1,b_i]}\cup \{w^{(i')}_{j'}\}_{i'\in[0,t],j'\in [b_{i'}+1,a_{i'}]}$ form a basis of $U_{t}=\mb{Z}_p^{a_0+\cdots+a_t}$, we have that $\det(A)\not=0\!\mod p$, and thus $\det(A_{i,i}) \not=0\!\mod p$ for $i\in [0,t]$.

Now note that, letting $H'$ be the subspace generated by vectors of the form $e_{a_0+\cdots+a_i+j}$ for $i=0,\ldots,t$ and $j=1,\ldots,b_i$, we have $\phi_k(H')=H$. But now this subspace has the form of the subspaces for which we understand the quotient $\ns/H'$. Thus $\ns/H\simeq \ns/H'$ and the latter equals
\[
\abph_{k-1,k-t(p-1)}^{\,b_0}\times \abph_{k,k-t(p-1)}^{\,a_0-b_0} \times \abph_{k-1,k-(t-1)(p-1)}^{\,b_1}\times \abph_{k,k-(t-1)(p-1)}^{\,a_1-b_1} \times \cdots\times  \abph_{k,k}^{\,a_t-b_t} \times Q'. \qedhere
\]\end{proof}

\begin{remark}\label{rem:gendiff}
It may be tempting to generalize the previous result to lower structure groups, but there are obstacles to a straightforward generalization. For example, let $\ns$ be the group nilspace consisting of $G=\mb{Z}_{25}$ with filtration $G_0=G_1=G_2=\mb{Z}_{25}$, $G_3=\cdots=G_7=5\mb{Z}_{25}$ and $G_i=\{0\}$ for $i\ge 8$. It can be checked that $\ns$ is not isomorphic to any nilspace that is a product of nilspaces in $\mc{Q}_{5,k}$, for any $k$ (note that the only possibility, given the structure groups of $\ns$, would be for $\ns$ to be isomorphic to $\mc{D}_2(\mb{Z}_5)\times \mc{D}_7(\mb{Z}_5)$; we leave it as an exercise to prove that this does not hold). However, we have the fibration $\varphi: \abph^{(5)}_{6,2}\times \mc{D}_7(\mb{Z}_5)\to \ns$, $(x,y)\mapsto x+5y$, where with $5y$ we mean that we take $5(y+5\mb{Z})\mod 25$ (the natural monomorphism $\mb{Z}_5\to  \mb{Z}_{25}$). It can be checked that $\abph^{(5)}_{6,2}\times \mc{D}_7(\mb{Z}_5)$ is thus a degree-6 extension of $\ns$, where the addition of $z\in \mc{D}_6(\mb{Z}_5)$ can be defined as $(x,y)+z:=(x+5z,y-z)$. Thus, to generalize Proposition \ref{prop:factorization-last-str-group}, we would have to take into account nilspaces such as $\ns$, that are not products of nilspaces in $\mc{Q}_{5,k}$.
\end{remark}

\section*{Acknowledgements}
\noindent We thank the anonymous referee for useful comments that helped to improve this paper. All authors received funding from Spain's MICINN project PID2020-113350GB-I00. The second-named author received funding from projects KPP 133921 and
Momentum (Lend\"ulet) 30003 of the Hungarian Government. The research was also supported partially by the NKFIH ``\'Elvonal'' KKP 133921 grant and partially by the Hungarian Ministry of Innovation and Technology NRDI Office within the framework of the Artificial Intelligence National Laboratory Program.


\begin{thebibliography}{1}
\bibitem{Ab} L. M. Abramov, \emph{Metric automorphisms with quasi-discrete spectrum}, Izv. Akad. Nauk SSSR Ser.Mat. \textbf{26} (1962), 513--530; English transl.: Amer. Math. Soc. Transl.\ (2) \textbf{39} (1964), 37--56.

\bibitem{BTZ} V. Bergelson, T. Tao, T. Ziegler, \emph{An inverse theorem for the uniformity seminorms associated with the action of $\mb{F}_p^{\infty}$}, Geom. Funct. Anal. \textbf{19} (2010), no. 6, 1539--1596.

\bibitem{BTZ2} V. Bergelson, T. Tao, T. Ziegler, \emph{Multiple recurrence and convergence results associated to $\mb{F}_p^{\omega}$-action}, J. Anal. Math. \textbf{127} (2015), 329--378.

\bibitem{Berger&al} A. Berger, A. Sah, M. Sawhney, J. Tidor, \emph{Non-classical polynomials and the inverse theorem}, Math. Proc. Cambridge Philos. Soc., to appear.

\bibitem{Bol} B. Bollob\'as, \emph{Linear analysis. An introductory course}. Second edition. Cambridge University Press, Cambridge, 1999.

\bibitem{CamSzeg} O. A. Camarena, B. Szegedy, \emph{Nilspaces, nilmanifolds and their morphisms}, preprint. \url{http://arxiv.org/abs/1009.3825}

\bibitem{Cand:Notes1} P. Candela, \emph{Notes on nilspaces: algebraic aspects}, Discrete Analysis, 2017, Paper No. 15, 59 pp.

\bibitem{Cand:Notes2} P. Candela, \emph{Notes on compact nilspaces}, Discrete Analysis, 2017, Paper No. 16, 57pp.

\bibitem{CGSS} P. Candela, D. Gonz\'alez-S\'anchez, B. Szegedy \emph{On nilspace systems and their morphisms}, Ergodic Theory Dynam. Systems \textbf{40} (2020), no. 11, 3015--3029.

\bibitem{CGSS-seq-CS-Eurocomb} P. Candela, D. Gonz\'alez-S\'anchez, B. Szegedy, \emph{A refinement of Cauchy-Schwarz complexity, with applications}. In J. Ne\v{s}et\v{r}il, G. Perarnau, J. Ru\'e, and Oriol Serra, editors, \emph{Extended Abstracts EuroComb 2021}, pages 293--298, Cham, 2021. Springer International Publishing.

\bibitem{CGSS-seq-CS} P. Candela, D. Gonz\'alez-S\'anchez, B. Szegedy, \emph{A refinement of Cauchy-Schwarz complexity},  European J. Combin. \textbf{106} (2022), Paper No. 103592. 

\bibitem{CScouplings} P. Candela, B. Szegedy, \emph{Nilspace factors for general uniformity seminorms, cubic exchangeability and limits}, Mem. Amer. Math. Soc., to appear. \url{https://arxiv.org/abs/1803.08758}

\bibitem{CSinverse} P. Candela, B. Szegedy, \emph{Regularity and inverse theorems for uniformity norms on compact abelian groups and nilmanifolds}, J. Reine Angew. Math.  \textbf{789} (2022), 1--42.

\bibitem{GSz} W. T. Gowers, \emph{A new proof of Szemer\'edi's theorem}, Geom. Funct. Anal. \textbf{11}  (2001), no. 3, 465--588.

\bibitem{GM2} W. T. Gowers, L. Mili\'cevi\'c, \emph{An inverse theorem for Freiman multi-homomorphisms}, preprint. \url{https://arxiv.org/abs/2002.11667}

\bibitem{GreenSurvey} B. J. Green, \emph{Finite field models in additive combinatorics}. In Bridget S Webb, editor, Surveys in combinatorics 2005, pages 1--27. Cambridge Univ. Press, Cambridge, Cambridge, 2005.

\bibitem{GT08}  B. Green, T. Tao, \emph{An inverse theorem for the Gowers $U^3$-norm}, Proc. Edinburgh Math. Soc. (1) \textbf{51} (2008), 73-153.

\bibitem{GTar} B. Green, T. Tao, \emph{An arithmetic regularity lemma, an associated counting lemma, and applications}.  An irregular mind, 261--334, Bolyai Soc. Math. Stud., 21, J\'anos Bolyai Math. Soc., Budapest, 2010. 

\bibitem{GTZ-U4} B. Green, T. Tao, T. Ziegler, \emph{An inverse theorem for the Gowers $U^4$-norm},  Glasgow Math. J. \textbf{53} (2011), 1--50. 

\bibitem{GTZ} B. Green, T. Tao, T. Ziegler, \emph{An inverse theorem for the Gowers $U^{s+1}[N]$-norm}, Ann. of Math. (2) \textbf{176} (2012), no. 2, 1231-1372.

\bibitem{GGY} E. Glasner, Y. Gutman, X. Ye, \emph{Higher order regionally proximal equivalence relations for general minimal group actions}, Adv. Math. \textbf{333} (2018), 1004--1041.

\bibitem{GL} Y. Gutman, Z. Lian, \emph{Strictly ergodic distal models and a new approach to the Host--Kra factors}, preprint. \url{https://arxiv.org/abs/1909.11349}

\bibitem{GMV1} Y. Gutman, F. Manners, P. P. Varj\'u, \emph{The structure theory of nilspaces I}, J. Anal. Math. \textbf{140} (2020), 299--369.

\bibitem{GMV2} Y. Gutman, F. Manners, P. P. Varj\'u, \emph{The structure theory of nilspaces II: Representation as nilmanifolds}, Trans. Amer. Math. Soc. \textbf{371} (2019), 4951--4992.

\bibitem{GMV3} Y. Gutman, F. Manners, P. P. Varj\'u, \emph{The structure theory of nilspaces III: Inverse limit representations and topological dynamics}, Adv. Math. \textbf{365} (2020), 107059.

\bibitem{HassKat} B. Hasselblatt, A. Katok, \emph{Principal structures. Handbook of dynamical systems, Vol. 1A}, 1-203, North-Holland, Amsterdam, 2002.

\bibitem{HK-non-conv} B. Host, B. Kra, \emph{Nonconventional ergodic averages and nilmanifolds}, Ann. of Math. (2) \textbf{161} (2005), no. 1, 397--488.

\bibitem{HK-par} B. Host, B. Kra \emph{Parallelepipeds, nilpotent groups, and Gowers norms}, Bull. Soc. Math. France \textbf{136} (2008), 405--437.

\bibitem{HKbook} B. Host, B. Kra, \emph{Nilpotent structures in ergodic theory}, Mathematical Surveys and Monographs Volume \textbf{236}; 2018; 427 pp.

\bibitem{K} A. S. Kechris, \emph{Classical descriptive set theory}, Graduate Texts in Mathematics, 156. Springer-Verlag, New York, 1995.

\bibitem{M} F. Manners, \emph{Quantitative bounds in the inverse theorem for the Gowers $U^{s+1}$-norms over cyclic groups}, preprint. \url{https://arxiv.org/abs/1811.00718}

\bibitem{MannersTC} F. Manners, \emph{True complexity and iterated Cauchy-Schwarz}, preprint. \url{https://arxiv.org/abs/2109.05731}

\bibitem{shalom1} O. Shalom, \emph{Host-Kra theory for $\bigoplus_{p\in P} \mb{F}_p$ systems and multiple recurrence}, Ergodic Theory Dynam. Systems, to appear.

\bibitem{shalom2} O. Shalom, \emph{Ergodic averages in abelian groups and Khintchine recurrence}, Trans. Amer. Math. Soc., to appear.

\bibitem{SzegHigh} B. Szegedy, \emph{On higher order Fourier analysis}, preprint. \url{https://arxiv.org/abs/1203.2260}

\bibitem{SzegFin} B. Szegedy, \emph{Structure of finite nilspaces and inverse theorems for the Gowers norms in bounded exponent groups}, preprint. \url{https://arxiv.org/abs/1011.1057}

\bibitem{TZ-High} T. Tao, T. Ziegler, \emph{The inverse conjecture for the Gowers norm over finite fields via the correspondence principle}, Anal. PDE \textbf{3} (2010), 1--20.

\bibitem{T&Z-Low} T. Tao, T. Ziegler, \emph{The inverse conjecture for the Gowers norm over finite fields in low characteristic}, Ann. Comb. \textbf{16} (2012), 121-188.

\bibitem{WolfSurvey} J. Wolf, \emph{Finite field models in arithmetic combinatorics--ten years on}, Finite Fields Appl. \textbf{32} (2015), 233--274.

\end{thebibliography}
\end{document}